\documentclass[11pt]{article}
\usepackage{amsmath}
\usepackage{amssymb,amsmath}
\usepackage{mathrsfs}
\usepackage{amsmath,amsthm,amsfonts,amssymb,bm,euscript}
\usepackage{fancyhdr,amscd}
\usepackage{anysize}
\usepackage{graphicx,epsfig}
\usepackage{amsthm,latexsym,amsmath,amssymb}
\usepackage{amssymb,amsmath}
\usepackage{mathrsfs}
\usepackage{graphicx}
\usepackage{color}
\usepackage{mathrsfs}
\usepackage{cite}
\usepackage{indentfirst}
 \usepackage[titletoc]{appendix}
 \usepackage [latin1]{inputenc}

\usepackage{amsmath, amssymb}
\usepackage{amsthm, amsfonts, mathrsfs}
\usepackage{mathptmx}
\usepackage{fullpage}
\usepackage{amsfonts,graphicx}
\usepackage[colorlinks=true,  linkcolor=blue, citecolor=blue, urlcolor=blue]{hyperref}
\let\r=\rho

\def\l{\frak l}

\def\C{\mathop{\bf C\kern 0pt}\nolimits}
\def\DD{\mathop{\bf D\kern 0pt}\nolimits}
\def\K{\mathop{\bf K\kern 0pt}\nolimits}
\def\N{\mathop{\bf N\kern 0pt}\nolimits}
\def\Q{\mathop{\bf Q\kern 0pt}\nolimits}
\def\R{\mathop{\bf R\kern 0pt}\nolimits}

\def \l {\langle}
\def \r {\rangle}
\def\V{{\mathbf{V}}}
\def\H{\mathbf{H}}
\def\W{\mathbf{W}}
\def\A{\mathbf{A}}

\def\uu{\textbf{\textit{u}}}
\def\vv{\textbf{\textit{v}}}

\def\C{\mathcal{C}}

\newcommand{\Dv}{{\rm div}}

\newcommand{\beq}{\begin{equation}}
\newcommand{\eeq}{\end{equation}}
\newcommand{\ben}{\begin{eqnarray}}
\newcommand{\een}{\end{eqnarray}}
\newcommand{\beno}{\begin{eqnarray*}}
\newcommand{\eeno}{\end{eqnarray*}}

\newtheorem{theorem}{Theorem}[section]

\newtheorem{lemma}[theorem]{Lemma}

\numberwithin{equation}{section}

\allowdisplaybreaks \numberwithin{equation} {section}
\begin{document}
\title{ On the uniqueness of strong solution to the   nonhomogeneous incompressible Navier-Stokes-Cahn-Hilliard system}
\author{ Lingxin  Jiang$^{1}$ \ \  Jiahong Wu$^{2}$ \ \   Fuyi  Xu$^{1\dag}$\\[2mm]
 { \small   $^1$School of Mathematics and  Statistics, Shandong University of Technology,}\\
  { \small Zibo,    255049,  Shandong,    China}\\
  { \small  $^2$Department of Mathematics, University of Notre Dame,}\\
   { \small Notre Dame, IN 46556, USA}
   }
         \date{}
\maketitle
\noindent{\bf Abstract} \ \ This paper is mainly concerned with an initial-boundary value problem of the  nonhomogeneous incompressible Navier-Stokes-Cahn-Hilliard system   with the Landau potential in a two and three dimensions. The  existence of  strong solutions with bounded and strictly positive density for this system was constructed by  Giorgini and Temam \cite{GT}. However, 
whether uniqueness holds 
has remained an open question. The present work solves this question and  we prove the  uniqueness of  strong solution. Our method mainly  relies on some extra time weighted estimates  and the Lagrangian approach.


\setlength{\baselineskip}{20pt}

\vskip .2in
\section{Introduction and Main Results}  In the present paper, we are interested in the study of incompressible diphasic nonhomogeneous mixtures flows. The model consists of a Cahn-Hilliard equation  coupled with a nonhomogeneous Navier-Stokes equations derived by Lowengrub and Truskinovsky \cite{LT} in 1998, which can be written as follows 
\begin{equation}\label{NSCH}
\begin{cases}
\partial_{t} \rho + \mathbf{u} \cdot \nabla \rho = 0, \\
\rho\partial_{t} \mathbf{u} + \rho (\mathbf{u}\cdot \nabla) u - \mathrm{div}(\nu(\phi) D u) + \nabla P = - \mathrm{div}(\nabla \phi \otimes \nabla \phi), \\
\mathrm{div} \, \mathbf{u} = 0 ,\\
\rho \partial_{t} \phi + \rho \mathbf{u} \cdot \nabla \phi = \Delta \mu, \\
\rho \mu = - \Delta \phi + \rho \Phi'(\phi),
\end{cases}
\end{equation}
in $\Omega\times(0,T)$, where  $\Omega$ is a bounded domain in $\mathbb{R}^d$($d=2,3$) with  a regular boundary $\partial\Omega$,  \(T > 0\) is a given positive time,  \(\rho = \rho(x,t)\) is the density of the mixture, \(\mathbf{u} = \mathbf{u}(x,t)\) is the (mass-averaged) velocity of the mixture, \(P = P(x,t)\) is the pressure of the mixture,  \(\varphi=\varphi(x,t)\) is the difference of fluids concentrations, \(\mu = \mu(x,t)\) is the chemical potential, $D\uu =  \frac12\big(\nabla \uu+ (\nabla \uu)^t\big)$ stands for the deformation tensor, whereas the viscosity function  $\nu  $ is assumed to satisfy:
 \begin{equation}\label{NSCH-N} \nu = \nu(s) \in W^{1, \infty}(\mathbb{R}), \quad  0 < \nu_* \leq \nu(s) \leq \nu^*\quad  \text{for all}\quad s \in \mathbb{R}.\end{equation}
 For the potential $ \Phi(s) $ (also called homogeneous free energy density), we will consider the following  physically relevant  Landau potential
 \begin{equation}\label{NSCH-W}\Phi_0(s) = \frac{1}{4} (s^2 - 1)^2 \quad \forall s \in \mathbb{R}. \end{equation}
Here,  we also supplement the system \eqref{NSCH} with the following boundary and initial conditions
\begin{equation}\label{NSCH1}
\begin{cases}
\boldsymbol{u} = 0, \quad \partial_{\mathbf{n} } \mu = \partial_{\mathbf{n} } \phi = 0 & \text{on } \partial \Omega \times (0, T), \\
\rho(\cdot, 0) = \rho_0, \quad \boldsymbol{u}(\cdot, 0) = \boldsymbol{u}_0, \quad \phi(\cdot, 0) = \phi_0 & \text{in } \Omega,
\end{cases}
\end{equation}
where $\mathbf{n} $ is the unit outward normal vector to  $\partial\Omega$.

It is well known that the application of the diffuse interface (or phase-field) theory has became a fundamental method in fluid mechanics to model and simulate large deformations and topological transitions in two-phase flows. Originally developed for phase transition/separation phenomena, the Diffuse Interface methodology is based on the description of the interface separating the two fluids as a narrow region with finite thickness across which continuous fields can change smoothly their values. The evolution of the macroscopic state variables is derived through the combination of the continuum theory of mixtures, classical thermodynamics and statistical mechanics. We refer the reader to the review articles \cite{AMW,GKL,GZ}. A paradigm model of the Diffuse Interface theory for two-phase flows is the  homogeneous Navier-Stokes-Cahn-Hilliard ( NSCH for short) system (i.e., $\rho(t, x) = \text{constant}$ in the system \eqref{NSCH}, also called Model H after the seminal work \cite{HH} on dynamic critical phenomena, which  was derived in \cite{GPV} within the framework of continuum mechanics and in \cite{LS} through an energetic variational approach (see also \cite{GKL} for a recent review). Over the past years there have been important developments concerning the mathematical modeling and analysis of NSCH system for binary mixtures (see, e.g., \cite{A, AAR, ADG, AW, BDM, Bo,  HW, UM, AMH, TD, B1, CG, SGM, T, GGM}) and references therein.

When the density of the mixture is an independent variable of the system and the mass-averaged velocity of the mixture is divergence-free, one need consider the generalization of the Model H called the nonhomogeneous incompressible Navier-Stokes-Cahn-Hilliard  equations, that is, the system \eqref{NSCH}, whose derivation is based on the conservation of mass and linear momentum and the second law of thermodynamic in the form of dissipation inequality as in \cite{AF,GPV,HMR,LT}. From the viewpoint of partial differential equations,  the nonhomogeneous Navier-Stokes-Cahn-Hilliard system is a highly nonlinear system coupling between hyperbolic equations and parabolic equations. Therefore, regarding the mathematical analysis of the system \eqref{NSCH}  is very recent, we mention the work \cite{ZH}, in which the author investigated the local existence of classical solutions in the case of constant viscosity and Landau potential, and \cite{GT} where the authors proved the existence of global weak solutions to the system with   non-constant viscosity,  bounded and strictly positive density and the free energy potential equal to either the Landau potential or the Flory-Huggins logarithmic potential in two and three dimensions and the existence of   strong solutions (global if $d=2$ and local if $d=3$)  in the case of Landau potential. Here we first recall the  existence of strong solutions to the multi-dimensional  initial-boundary value problem \eqref{NSCH}-\eqref{NSCH1}.
\begin{theorem}\label{1.1}\cite{GT}
Let $\Omega$ be a bounded domain of class $C^3$ in $\mathbb{R}^d$, where $d = 2,3$. Assume that $\rho_0\in L^{\infty}(\Omega)$, $\mathbf{u}_0\in \mathbf{V}_{\sigma}(\Omega)$ and $\phi_0\in H^2(\Omega)$ are given such that $0<\rho_*\leq \rho_0\leq \rho^*$, $\partial_{\mathbf{n}}\phi_0 = 0$ on $\partial\Omega$, and $\mathbf{u}_0 = -\frac{\Delta\phi_0}{\rho_0}+\Phi_0'(\phi_0)\in H^1(\Omega)$. Then, we have

\textbf{(i)}\begin{equation}\label{NSCH2}
0<\rho_* \leq \rho(x, t) \leq \rho^* \quad \text{a.e. in } \Omega \times (0, T). \end{equation}

\textbf{(ii)} If \( d = 2 \),  for any \( T > 0 \), there exists a  strong solution \( (\rho, \boldsymbol{u}, P, \phi, \mu) \) to the initial-boundary value problem \eqref{NSCH}-\eqref{NSCH1} satisfying
\begin{align*}
\rho &\in \mathcal{C}([0,T]; L^r(\Omega)) \cap L^\infty(\Omega \times (0,T)) \cap L^\infty(0,T; H^{-1}(\Omega)), \\
\boldsymbol{u} &\in \mathcal{C}([0,T]; \mathbf{V}_\sigma) \cap L^2(0,T; H^2(\Omega)) \cap H^1(0,T; \mathbf{H}_\sigma), \\
P&\in L^2(0,T; H^1(\Omega)), \\
\phi &\in \mathcal{C}([0,T]; (W^{2,q}(\Omega))_w) \cap H^1(0,T; H^1(\Omega)), \\
\mu &\in L^\infty(0,T; H^1(\Omega)) \cap L^2(0,T; W^{2,q}(\Omega)),
\end{align*}
for any \( r \in [2, \infty) \) and any \( q \in [2, \infty) \).

\textbf{(iii)} If $d = 3$, there exist $T>0$ depending on the norms of the initial data, and a strong solution $(\rho,\mathbf{u},P,\phi,\mu)$ to the initial-boundary value problem   \eqref{NSCH}-\eqref{NSCH1} satisfying
    \begin{align*}
        \rho&\in \mathcal{C}([0,T];L^r(\Omega))\cap L^{\infty}(\Omega\times(0,T))\cap L^{\infty}(0,T;H^{-1}(\Omega)),\\
        \mathbf{u}&\in \mathcal{C}([0,T];\mathbf{V}_{\sigma})\cap L^2(0,T;H^2(\Omega))\cap H^1(0,T;\mathbf{H}_{\sigma}),\\
        P&\in L^2(0,T; H^1(\Omega)), \\
        \phi&\in \mathcal{C}([0,T];(W^{2,6}(\Omega))_w)\cap H^1(0,T;H^1(\Omega)),\\
        \mu&\in L^{\infty}(0,T;H^1(\Omega))\cap L^2(0,T;W^{2,6}(\Omega)),
    \end{align*}
    for any $r\in[2,\infty)$.
\end{theorem}
Here, it should be emphasized that, authors in \cite{GT}  remained  an interesting open  issue is to prove the uniqueness of  strong solution constructed in Theorem \ref{1.1}. Recently,  for the  two dimensional case,   Giorgini  et al. \cite{GT1} proved uniqueness of  regular solutions to the initial-boundary value problem \eqref{NSCH}-\eqref{NSCH1} with an additional smoothness assumption on the initial density. More precisely, they obtained  the following theorem.
\begin{theorem}\label{1.1-1} \cite{GT1}  Let $d=2$, and let the assumptions of Theorem \ref{1.1} hold. Assume that $\rho_0\in W^{1,m}(\Omega)$ for some $m>4$. Then, there exists a unique global  regular solution $(\rho,\mathbf{u},P,\phi,\mu)$ for the initial-boundary value problem \eqref{NSCH}-\eqref{NSCH1}.
\end{theorem}

However,  whether uniqueness of  strong solution  constructed by Theorem \ref{1.1} holds 
remains unresolved.  The purpose of this work is to solve this question  and  we finally establish uniqueness of  strong solution to the initial-boundary value problem \eqref{NSCH}-\eqref{NSCH1}  in the framework in \cite{GT}. Our  main result is the following theorem.
\begin{theorem}\label{1.2} For some $T>0$, let $(\rho_{1},\mathbf{u}_{1},P_{1},\phi_{1},\mu_{1})$ and $(\rho_{2},\mathbf{u}_{2},P_{2},\phi_{2},\mu_{2})$ be two  solutions to the initial-boundary value problem  \eqref{NSCH}-\eqref{NSCH1}  on $[0,T]\times\Omega$  constructed by Theorem \ref{1.1} corresponding to the same initial data.  Then, $(\rho_{1},\mathbf{u}_{1},P_{1},\phi_{1},\mu_{1})\equiv(\rho_{2},\mathbf{u}_{2},P_{2},\phi_{2},\mu_{2})$ on $[0,T]\times\Omega$.
\end{theorem}


 We here make a brief interpretation of the main difficulties and strategies involved in the proof.  First, since the density is only bounded, it seems impossible to prove the uniqueness of the strong  solution in the Eulerian coordinates as in \cite{CHK,GT1,Lijinkai}. Indeed, let $(\rho_{1},\mathbf{u}_{1},P_{1},\phi_{1},\mu_{1})$ and $(\rho_{2},\mathbf{u}_{2},P_{2},\phi_{2},\mu_{2})$ be two different solutions of the  system \eqref{NSCH}. Then
 $\delta \rho=\rho_1-\rho_2$ satisfies
 \begin{equation*}
\delta \rho_t + \mathbf{u}_1 \cdot \nabla \delta \rho = - (\mathbf{u}_1 - \mathbf{u}_2) \cdot \nabla \rho_2.
\end{equation*}
Without extra assumptions about the regularity of these solutions, the term $(\mathbf{u}_1 - \mathbf{u}_2) \cdot \nabla \rho_2$ cannot be handled by the energy method because the usual technique to prove uniqueness via Gronwall's inequality  cannot be applied here. And the uniqueness result of Germain \cite{PG} cannot be applied here either, which requires the density function satisfying
 $\nabla \rho \in L^{\infty}\big(0,T; L^{d}(\Omega)\big)$. Consequently, the uniqueness issue is non-trivial due to the roughness of the density and the hyperbolic nature of the continuity equation. To address this problem, we shall use the Lagrangian coordinates defined by the stream lines, which is motivated by \cite{RB1,RB,PZZ,FM,FMQ}. 
 According to the pioneering work by D. Hoff in \cite{DH} or to the recent papers \cite{RB1,RB,PZZ,FM,FMQ}, in most evolutionary fluid mechanics models, the condition  $\nabla \mathbf{u} \in L^{1}\big(0,T; L^{\infty}(\Omega)\big)$   seems to be the minimal requirement in order to get
uniqueness. However,  when the density is rough, propagate enough regularity for the velocity is the main difficulty. In order to  bound the quantity $\int_{0}^{T}\|\nabla \mathbf{u}\|_{L^{\infty}}d\tau$,  we first  exploit some  extra time-weighted energy estimates for the velocity field. Combining with these time-weighted estimates, interpolation results, classical Sobolev embedding, and shift of integrability from the time variable to the space variable, we eventually get the Lipschitz control of the velocity field.
Finally, it should be pointed out that, compared with the  nonhomogeneous incompressible Navier-Stokes equations, we also need  deal with some essential difficulties caused by the more complex nonlinear terms  and the hyperbolic-parabolic  coupling effect among the density,  the velocity field and the phase field in the system \eqref{NSCH}. 

 The rest of the paper unfolds as follows. In the next section, we shall introduce some  functional settings and   related analysis tools. In Section 3,
  we shall exploit  some extra  time-weighted estimates on time derivatives for  the strong solution  to the system \eqref{NSCH}.
  In section 4 we further derive the key estimate for quantity $\int_{0}^{T}\|(\nabla \mathbf{u}, \nabla \mu, \nabla\phi)(\tau)\|_{L^{\infty}}d\tau$.
The last section  is devoted to the proof of Theorem \ref{1.2}
\section{Preliminaries} This section reviews various tools, including some  functional settings,  important inequalities,  and useful lemmas that will be referenced throughout the paper.

Let $X$ be a (real) Banach or Hilbert space with norm denoted by $\| \cdot\|_X$. The boldface letter $\mathbf{X}$ stands for the vectorial space $X^d$ ($d$ is the spatial dimension), which consists of vector-valued functions $\uu$ with all components belonging to $X$, with norm $\| \cdot\|_{\mathbf{X}}$.
Let $\Omega$ be a bounded domain in $\mathbb{R}^d$, where $d=2$ or $d=3$, with smooth boundary $\partial \Omega$.
We denote by $W^{k,p}(\Omega)$, $k\in \mathbb{N}$, the Sobolev space of
functions in $L^p(\Omega)$ with distributional derivatives of order less than or
equal to $k$ in $L^p(\Omega)$ and by $\| \cdot \|_{W^{k,p}(\Omega)}$ its norm.
For $k\in \mathbb{N}$, the Hilbert space $W^{k,2}(\Omega)$ is denoted
by $H^k(\Omega)$ with norm $\|\cdot \|_{H^k(\Omega)}$.
We denote by $H_0^1(\Omega)$ the closure of
$\mathcal{C}_0^{\infty}(\Omega)$ in
$H^1(\Omega)$ and by $H^{-1}(\Omega)$ its dual space.
We define $H=L^2(\Omega)$. Its inner product and norm are denoted by $( \cdot,\cdot )$ and $\| \cdot \|$,
respectively. We set $V=H^{1}(\Omega)$ with norm $\|\cdot \|_V$, and
we denote its dual space by $V^{\prime}$ with norm $\|
\cdot \|_{V'}$. The symbol $\l \cdot, \cdot \r$ will stand
for the duality product between $V$ and $V'$.
We denote by $\overline{u}$ the average of $u$ over $\Omega$, that is $\overline{u}=|\Omega|^{-1}\l u,1\r$, for all $u\in V'$. By the generalized Poincar\'{e}'s inequality (see \cite[Chapter II, Section 1.4]{T}),
we recall that
$u \rightarrow (\| \nabla u\|^2+ |\overline{u}|^2)^\frac12$ is a norm on $V$ equivalent to the natural one.


We now introduce the Hilbert space of solenoidal vector-valued functions. We denote by $\mathcal{C}_{0,\sigma}^\infty(\Omega)$ the space of divergence free vector fields in $\mathcal{C}_{0}^\infty(\Omega)$.
We define $\H_\sigma$ and $\V_\sigma$ as the closure of $\mathcal{C}_{0,\sigma}^\infty(\Omega)$ with respect to the $\H$ and $\H_0^1(\Omega)$ norms, respectively.
We also use $( \cdot ,\cdot )$ and
$\Vert \cdot \Vert $ for
the norm and the inner product in $\H_\sigma$. The space $\V_\sigma$ is endowed with the inner product and norm
$( \uu,\vv )_{\V_\sigma}=
( \nabla \uu,\nabla \vv )$ and  $\|\uu\|_{\V_\sigma}=\| \nabla \uu\|$, respectively.
We denote by $\V_\sigma'$ its dual space.
We recall that  Korn's inequality
entails
\begin{equation}\label{korn}
\|\nabla\uu\|\leq \sqrt2\|D\uu\|\leq \sqrt2 \| \nabla \uu\|,
\quad \forall \, \uu \in \V_\sigma.
\end{equation}
In turn, the above inequality gives that $\uu \rightarrow \|D\uu\|$ is a norm on $\V_\sigma$ equivalent to the initial norm. We consider the Hilbert space
$\W_\sigma= \mathbf{H}^2(\Omega)\cap \V_\sigma$
with inner product and norm
$ ( \uu,\vv)_{\W_\sigma}=( \A\uu, \A \vv )$ and $\| \uu\|_{\W_\sigma}=\|\A \uu \|$, where $\A$ is the Stokes operator.
We recall that there exists $C>0$ such that
\begin{equation}
\label{H2equiv}
 \| \uu\|_{H^2(\Omega)}\leq C\| \uu\|_{\W_\sigma}, \quad \forall \, \uu\in \W_\sigma.
\end{equation}
We also  recall the following Gagliardo-Nirenberg and Agmon inequalities.
\begin{lemma}\label{chazhi}\cite{Temam}
\begin{align}
\label{LADY}
&\| u\|_{L^4(\Omega)}\leq C \|u\|^{\frac12}\|u\|_V^{\frac12}, \quad &&\forall \, u \in V, \quad \text{if} \ d=2,\\
\label{LADY3}
&\| u\|_{L^3(\Omega)}\leq C \|u\|^{\frac12}\|u\|_V^{\frac12}, \quad &&\forall \, u \in V, \quad \text{if} \ d=3,\\
\label{Agmon2d}
&\| u\|_{L^\infty(\Omega)}\leq C \|u\|^{\frac12}\|u\|_{H^2(\Omega)}^{\frac12}, \quad && \forall \, u \in H^2(\Omega), \quad \text{if} \ d=2,\\
\label{DL4}%
&\|\nabla u\|_{\mathbf{L}^4(\Omega)}\leq C \| u\|_{L^\infty(\Omega)}^\frac12 \| u\|_{H^2(\Omega)}^\frac12,  \quad &&\forall \, u \in H^2(\Omega), \quad \text{if} \ d=2,3.
\end{align}
\end{lemma}

 We here list a useful lemma  that have been used in the proof of uniqueness.
\begin{lemma}\cite{DM1,RB} \label{lemma:4.1} Let \( A \) be a matrix valued function on \( [0,T] \times \Omega\) satisfying
\begin{align}
\det A &\equiv 1.
\end{align}
There exists a constant \( c \) depending only on \( d \), such that if
\begin{align}
\| \text{Id} - A \|_{L^\infty(0,T; L^\infty)} + \| A_t \|_{L^2(0,T; L^6)} &\leq c,
\end{align}
then for all function \( R : [0,T] \times \Omega \to \mathbb{R}^d \) satisfying \( \operatorname{div} R \in L^2(0,T \times \Omega) \), \( R \in L^4(0,T; L^2) \), \( R_t \in L^{4/3}(0,T; L^{3/2}) \) and \( R \cdot \mathbf{n} \equiv 0 \) on \( (0,T) \times \partial \Omega \), the equation
\begin{align}
\operatorname{div}(Av) &= \operatorname{div} R =: g \quad \text{in} \quad [0,T] \times \Omega
\end{align}
admits a solution in the space
\[
X_T := \left\{ v \in L^2(0,T; H_0^1(\Omega)), \, v \in L^4(0,T; L^2(\Omega)) \, \text{and} \, v_t \in L^{4/3}(0,T; L^{3/2}(\Omega)) \right\},
\]
satisfying the following inequalities for some constant \( C = C(d) \):
\begin{equation}\label{4.9}
\begin{split}
&\| v \|_{L^4(0,T; L^2)}\leq C \| R \|_{L^4(0,T; L^2)}, \quad \| \nabla v \|_{L^2(0,T; L^2)} \leq C \| g \|_{L^2(0,T; L^2)},  \\
& \| v_t \|_{L^{4/3}(0,T; L^{3/2})} \leq C \| R \|_{L^4(0,T; L^2)} + C \| R_t \|_{L^{4/3}(0,T; L^{3/2})}.
\end{split}
\end{equation}
\end{lemma}
\section{Some extra  weighted energy estimates } In order to perform the shift integrability from time to
space variables,  our main  aim  in this section is to exploit some extra time-weighted estimates, for example, $(\sqrt{\rho t}\mathbf{u}_t, \sqrt{\rho t}\phi_t,\sqrt{t}\nabla\phi_t)$ in $L^{\infty}([0,T]; L^2)$, $(\sqrt{t}\nabla \mathbf{u}_t, \sqrt{t}\nabla \mu_t)$ in $L^2([0,T]; L^2)$  and $\sqrt{t\rho}\mu_t$ in $L^{2}([0,T]; L^2)$ respectively, in terms of the data. This is presented  by the following two lemmas in two and three dimensions, respectively.
\begin{lemma}\label{lemma3.2}
Assume \(d = 3\), and that a strong solution $(\rho,\mathbf{u},P,\phi,\mu)$ to the initial-boundary value problem   \eqref{NSCH}-\eqref{NSCH1}. Then for all \(0\leq t\leq T\), we have
\begin{align}\label{3D}
\|\sqrt{\rho t}\mathbf{u}_t\|_2^2&+\|\sqrt{\rho t}\phi_t\phi\|_2^2+\|\sqrt{ t}\nabla\phi_t\|_2^2+\|\sqrt{\rho t}\phi_t\|_2^2
+\int_{0}^{t}\nu_*\tau\|\nabla \mathbf{u}_t\|_2^2d\tau\\ \notag
&+\int_{0}^{t}\tau\|\nabla\mu_t\|_2^2d\tau+\int_{0}^{t}\tau\rho\|\mu_t\|_2^2d\tau
\leq\exp\Big(\int_{0}^{t}h_1(\tau)d\tau\Big)-1,
\end{align}
where
\begin{align*}
&h_1(t)=
(1 + C + C_{\rho^*} + C_{\rho^*\rho_*T} + C_T)
\Big(\big( \|\phi\|_{12}^3 + \|\phi\|_4 \big)^2
\big( \|\nabla^2 \phi\|_4^2 + \|\nabla \phi\|_4^2 + \|\mathbf{u}\|_4^2 \big) \\
&+ \big( \|\nabla \phi\|_6 + \|\phi\|_\infty^2 \|\nabla\phi\|_6 \big)^2
\big( \|\mathbf{u}\|_6^2 + \|\nabla\phi\|_4^2  + \|\mathbf{u}\|_4^2 \big) + \|\mathbf{u}\|_4^2 \big( \|\nabla^2 \phi\|_6 + \|\nabla \phi\|_{12}^2 \|\phi\|_\infty + \|\nabla^2 \phi\|_6 \|\phi\|_\infty^2 \big)^2 \\
&+ \|\nabla \mathbf{u}\|_2^2 + \| \mu\|_4^2 \|\nabla \phi\|_4^2 + \|\nabla \phi\|_2^2 + \|\phi\|_\infty^4 + \|\nabla \mu\|_2^2 + \|\nabla \phi\|_\infty^2
 + \|\nabla^2 \phi\|_2^2 + \|\nabla \phi\|_4^2 + \|\phi\|_8^4 \|\nabla\phi\|_4^2 \\
&+ \|\nabla \phi\|_4^8 + \big( \|\phi\|_\infty^2 + 1 \big)^2 \|\nabla \phi\|_4^2 + \|\mathbf{u}\|_4^4 + \|\mathbf{u}\|_6^4 + \|\phi\|_6^4 + \|\nabla \mathbf{u}\|_2^2  + \|\mathbf{u}\|_4 \|\nabla \phi\|_4 \|\phi\|_6 + \|\nabla \mu\|_2^2 + \|\nabla \phi\|_4^2\\& + \|\phi\|_\infty^4
+ \|\nabla \phi\|_4^2 + \|\mathbf{u}\|_6^2 + \|\mathbf{u}\|_2^2 + \|\nabla \mathbf{u}\|_2^2 + \|\nabla \mathbf{u}\|_2^{\frac{7}{2}} \\
&+ \|\mathbf{u}\|_4^2 \big( \|\nabla \phi\|_4^2 + \|\nabla^2 \phi\|_4^2 + \|\mu\|_4^2 \big)
+ \|\nabla^2 \phi\|_4^2 \|\mu\|_4^2 + \|\nabla \phi\|_8^2 \|\phi\|_8^2
+ \big( \|\phi\|_\infty^3 + \|\phi\|_\infty\big)^2 \Bigr)\\
&\times\Big( \|\mathbf{u}\|_{\infty}^2 + \|\mathbf{u}\|_{\infty}^4 + \|\mu\|_{\infty}^2 + \|\phi_t\|_2^2 + \| \phi_t\|_4^2 + \| \phi_t\|_6^2 + \|\phi_t\|_2 + \|\mathbf{u}_t\|_2^2 + \|\nabla \phi_t\|_2^2 \\
&+  \|\nabla^2 \mathbf{u}\|_2^2 + \|\nabla \mathbf{u}\|_3^2 + \|\nabla \mu\|_4^2 + \|\nabla \mu\|_6^2 + \|\nabla^2 \mu\|_6^2 + \|\nabla^2 \mu\|_4^2 + \|\nabla \mu\|_{w^{1,6}}^2 \Big)
\end{align*}
is $ L^1_{loc}(\mathbb{R}^+) $ only depends on, \( \rho^* \), \( \rho_* \), $T$ and the initial value.
\end{lemma}
\begin{proof}
Differentiating \eqref{NSCH}$_2$ with respect to $t$, respectively, multiplying by $\sqrt t$, and taking the inner product with $\sqrt t \mathbf{u}_t$, we have
\begin{align*}
\frac{1}{2}\frac{d}{dt}&\int_\Omega\rho t|\mathbf{u}_t|^{2}dx-\int_\Omega\frac{d}{dt}\bigl(\Dv(\nu(\phi)D\mathbf{u}\bigr)\sqrt t\cdot\sqrt t \mathbf{u}_tdx+\int_\Omega\nabla P_t\cdot t\mathbf{u}_tdx\\
=&-\int_\Omega\frac{d}{dt}\bigl(\Dv\nabla\phi\otimes\nabla\phi\bigr)\sqrt t\cdot \sqrt t \mathbf{u}_tdx+\frac{1}{2}\int_\Omega\rho |\mathbf{u}_t|^{2}dx-\frac{1}{2}\int t\rho_t|\mathbf{u}_t|^{2}dx\\
&-\int_\Omega\sqrt t\rho_t \mathbf{u}\cdot\nabla \mathbf{u}\cdot\sqrt t \mathbf{u}_tdx-\int_\Omega\sqrt\rho\mathbf{u}_t\cdot\nabla \mathbf{u}\cdot \sqrt t \mathbf{u}_tdx-\int_\Omega\sqrt t \rho \mathbf{u}\cdot\nabla \mathbf{u}_t\cdot\sqrt t \mathbf{u}_tdx.
\end{align*}
On the other hand,
\begin{align*}
-\int_\Omega\frac{d}{dt}\Dv\bigl(\nu(\phi)D\mathbf{u}\bigr)\sqrt t\cdot\sqrt t \mathbf{u}_tdx
&=\int_\Omega Dv\bigl(\nu'(\phi)\phi_tDu\bigr)\cdot t \mathbf{u}_tdx+\int_\Omega \Dv\bigl(\nu(\phi )D\mathbf{u}_t\bigr)\cdot t \mathbf{u}_tdx\\
&=\int_\Omega \nu'(\phi)\phi_tD\mathbf{u}\cdot \nabla(t \mathbf{u}_t)dx+\int_\Omega \nu(\phi )D\mathbf{u}_t\cdot \nabla (t \mathbf{u}_t)dx,
\end{align*}
which together with  \eqref{NSCH}$_{5}$,  yields that
\begin{align*}
-&\int_\Omega\frac{d}{dt}(\Dv\nabla\phi\otimes\nabla\phi)\sqrt t\cdot \sqrt t \mathbf{u}_tdx\\&=\int_\Omega\frac{d}{dt}\bigl(\rho\mu\cdot\nabla\phi-\rho\nabla\Phi(\phi)
-\nabla(\frac{1}{2}|\nabla\phi|^2)\bigr)\sqrt t\cdot\sqrt t \mathbf{u}_tdx\\
&=\int_\Omega\rho_t\mu\cdot\nabla\phi\cdot t \mathbf{u}_tdx+\int_\Omega t\rho\mu_t\cdot\nabla\phi\cdot \mathbf{u}_tdx+
\int_\Omega t\rho\mu\cdot\nabla \phi_t\cdot\mathbf{u}_tdx\\
&\quad-\int_\Omega t\rho_t\Phi^{'}(\phi)\cdot\nabla\phi \cdot\mathbf{u}_tdx
- \int_\Omega t \rho{\Phi}''(\phi) \phi_t\cdot \nabla \phi\cdot \mathbf{u}_tdx- \int_{\Omega} t \rho \Phi'(\phi)\cdot \nabla \phi_t\cdot \mathbf{u}_tdx.
\end{align*}
Observing that
\begin{align*}
-&\int_{\Omega} \frac{d}{dt} \left( \nabla  \frac{1}{2} |\nabla \phi|^2\right)\cdot  t \mathbf{u}_tdx
\\&= -\int_{\Omega} \frac{d}{dt} \left( \nabla  \frac{1}{2} |\nabla \phi|^2 t \cdot \mathbf{u}_t\right) dx
+\int_{\Omega}\nabla( \frac{1}{2} |\nabla \phi|^2) \cdot\mathbf{u}_tdx
+\int_{\Omega} \nabla( \frac{1}{2} |\nabla \phi|^2 ) \cdot t\mathbf{u}_{tt}dx \\
&= -\frac{d}{dt}  \int_{\Omega}\nabla  (\frac{1}{2} |\nabla \phi|^2 )\cdot t \mathbf{u}_tdx
- \int_{\Omega}  ( \frac{1}{2} |\nabla \phi|^2 )\cdot\text{div} \, \mathbf{u}_tdx
- \int_{\Omega}   \frac{1}{2} |\nabla \phi|^2 \cdot\text{div} \, t \mathbf{u}_{tt}dx  \\
&=-\frac{d}{dt} \int_{\Omega}  \frac{1}{2} |\nabla \phi|^2 \cdot\text{div} \, t \mathbf{u}_t dx= 0,
\end{align*}
and employing  Korn's inequality \eqref{korn}, we conclude that
\begin{align}\notag
&\frac{1}{2} \frac{d}{dt} \int_{\Omega} t\rho |\mathbf{u}_t|^2 dx + \frac{1}{2} \int_{\Omega} \nu(\phi) t|\nabla \mathbf{u}_t|^2 dx\label{eq-31} \\ \notag
&\leq \frac{1}{2} \int_{\Omega} \rho |\mathbf{u}_t|^2dx - \frac{1}{2} \int_{\Omega} t\rho_t \, |\mathbf{u}_t|^2dx - \int_{\Omega} \left( \sqrt{t} \rho_t \mathbf{u} \cdot \nabla \mathbf{u} \right)\cdot \left( \sqrt{t} \mathbf{u}_t \right) dx- \int_{\Omega} \left( \sqrt{t} \rho \mathbf{u}_t \cdot \nabla \mathbf{u} \right) \cdot\left( \sqrt{t} \mathbf{u}_t \right) dx \\
&\quad- \int_{\Omega} \sqrt{t} \rho \mathbf{u} \cdot \nabla \mathbf{u}_t \cdot\left( \sqrt{t} \mathbf{u}_t \right) dx - \int_{\Omega}\nu'(\phi) \, \phi_t D\mathbf{u} \cdot \nabla t \mathbf{u}_tdx + \int_{\Omega} \left( \sqrt{t} \rho_t \mu \nabla \phi \right) \cdot\left( \sqrt{t} \mathbf{u}_t \right) dx \\ \notag
&\quad+ \int_{\Omega} \left( \sqrt{t} \rho \mu_t \nabla \phi \right)\cdot \left( \sqrt{t} \mathbf{u}_t \right) dx + \int_{\Omega} \left( \sqrt{t} \rho \mu \nabla \phi_t \right) \cdot\left( \sqrt{t} \mathbf{u}_t \right) dx - \int_{\Omega} \sqrt{t} \rho_t \Phi'(\phi) \nabla \phi\cdot \left( \sqrt{t} \mathbf{u}_t \right) dx \\
&\quad- \int_{\Omega} \sqrt{t}  \rho \Phi''(\phi) \, \phi_t \nabla \phi\cdot \left( \sqrt{t}\mathbf{u}_t \right) dx - \int_{\Omega} \left( \sqrt{t}  \rho \Phi'(\phi) \nabla \phi_t \right)\cdot \left( \sqrt{t} \mathbf{u}_t \right)dx \notag
\\&\triangleq\sum_{i=1}^{12} I_i \notag.
\end{align}
Differentiating \eqref{NSCH}$_4$ with respect to $t$, respectively, multiplying by $\sqrt t$, and taking the inner product with $\sqrt t \mu_t$, yield that
\begin{align*}-\int_{\Omega} |\nabla \sqrt{t} \mu_t|^2 \, dx&=\int_{\Omega} (\sqrt{t} \rho_t \phi_t)(\sqrt{t} \mu_t) \, dx + \int_{\Omega} \rho (\sqrt{t} \phi_t)_t \sqrt{t} \mu_t \, dx - \frac{1}{2} \int_{\Omega} \rho \phi_t \mu_t \, dx \\
&\quad+ \int_{\Omega} (\sqrt{t} \rho_t \mathbf{u}\cdot \nabla\phi)(\sqrt{t} \mu_t) \, dx+ \int_{\Omega} (\sqrt{t} \rho \mathbf{u}_t\cdot \nabla\phi)(\sqrt{t} \mu_t) \, dx \\
&\quad+ \int_{\Omega} (\sqrt{t} \rho \mathbf{u}\cdot \nabla\phi_t)(\sqrt{t} \mu_t) \, dx.
\end{align*}
By  virtue of \eqref{NSCH}$_5$,  we obtain
\begin{align*}
\int_{\Omega} \rho (\sqrt{t} \phi_t)_t \sqrt{t} \mu_t \, dx &= \int_{\Omega} (\sqrt{t} \phi_t)_t \sqrt{t} \left(-\Delta \phi_t + \rho_t (\phi^3 - \phi) + 3\rho \phi^2 \phi_t - \rho \phi_t - \rho_t \mu \right) dx \\
&= \frac{1}{2} \frac{d}{dt} \int_{\Omega} |\nabla \sqrt{t} \phi_t|^2 dx + \int_{\Omega} (\sqrt{t} \phi_t)_t \sqrt{t} \rho_t \phi^3 dx - \int_{\Omega} (\sqrt{t} \phi_t)_t \sqrt{t} \rho_t \phi dx \\
&\quad+ 3 \int_{\Omega} (\sqrt{t} \phi_t)_t \sqrt{t} \rho \phi^2 \phi_t \, dx - \int_{\Omega} (\sqrt{t} \phi_t)_t \sqrt{t} \rho \phi_t \, dx - \int_{\Omega} (\sqrt{t} \phi_t)_t \sqrt{t} \rho_t \mu dx.
\end{align*}
Then we have
\begin{equation}\label{eq-32}
\begin{split}
&\int_{\Omega} |\nabla \sqrt{t} \mu_t|^2 dx + \frac{1}{2} \frac{d}{dt} \int_{\Omega} |\nabla \sqrt{t} \phi_t|^2 dx\\& = -\int_{\Omega} \sqrt{t} \rho_t \phi_t \sqrt{t} \mu_t dx - \int_{\Omega} (\sqrt{t} \phi_t)_t \sqrt{t} \rho_t \phi^3 dx + \int_{\Omega} (\sqrt{t} \phi_t)_t \sqrt{t} \rho_t \phi \, dx\\
 &\quad- 3\int_{\Omega} (\sqrt{t} \phi_t)_t \sqrt{t} \rho \phi^2 \phi_t dx + \int_{\Omega}
 (\sqrt{t} \phi_t)_t \sqrt{t} \rho \phi_t dx + \int_{\Omega} (\sqrt{t} \phi_t)_t \sqrt{t} \rho_t \mu dx + \frac{1}{2} \int_{\Omega} \rho\phi_t \mu_t dx \\
&\quad- \int_{\Omega} (\sqrt{t} \rho_t \mathbf{u}\cdot \nabla\phi)(\sqrt{t} \mu_t) dx
- \int_{\Omega} (\sqrt{t} \rho \mathbf{u}_t \cdot\nabla\phi)(\sqrt{t} \mu_t) dx + \int_{\Omega} (\sqrt{t} \rho \mathbf{u}\cdot \nabla\phi_t)(\sqrt{t} \mu_t) \\&\triangleq\sum_{i=1}^{10} K_i.
\end{split}
\end{equation}
Furthermore, differentiating \eqref{NSCH}$_4$, \eqref{NSCH})$_5$ with respect to $t$, respectively, multiplying by $\sqrt t$, then taking the inner product with $\sqrt t \phi_t$, $\sqrt t \mu_t$ and summing them together,  we conclude that
\begin{equation}\label{eq-33}
\begin{split}
&\frac{1}{2} \frac{d}{dt} \int_{\Omega}  \rho t|\phi_t|^2 dx + \int_{\Omega}t\rho |\mu_t|^2 dx \leq \frac{1}{2} \int_{\Omega} \rho|\phi_t|^2 dx-\frac{1}{2} \int_{\Omega} t\rho_t|\phi_t|^2 dx - \int_{\Omega} (\sqrt{t} \rho_t \mathbf{u}\cdot \nabla\phi)(\sqrt{t} \phi_t) dx \\
&\quad- \int_{\Omega} (\sqrt{t} \rho \mathbf{u}_t\cdot \nabla\phi)(\sqrt{t} \phi_t) dx - \int_{\Omega} (\sqrt{t} \rho \mathbf{u}\cdot \nabla\phi_t)(\sqrt{t} \phi_t) dx - \int_{\Omega} (\sqrt{t} \rho_t \mu)(\sqrt{t} \mu_t) dx \\
&\quad+ \int_{\Omega} \bigl(\sqrt{t} \rho_t (\phi^3 - \phi)\bigr)(\sqrt{t} \mu_t) dx + \int_{\Omega} \sqrt{t} \rho (3\phi^2 \phi_t - \phi_t)(\sqrt{t} \mu_t)
\\&\triangleq\sum_{i=1}^{8} L_i.
\end{split}
\end{equation}
In what follows, let us bound these terms in the right hand sides of  \eqref{eq-32}-\eqref{eq-33}.  For \( I_2 \), thanks to \( \rho_t = - \mathbf{u} \cdot \nabla \rho \), we have
\begin{equation}
\begin{split}
I_2 &\leq C \Big| \int_{\Omega} t \mathrm{div}(\rho \mathbf{u}) |\mathbf{u}_t|^2 dx \Big|\leq C \int_{\Omega} t \rho |\mathbf{u}| |\nabla \mathbf{u}_t| |u_t| dx  \\
&\leq C\Big( \int_{\Omega} \rho t |\mathbf{u}_t|^2 dx \Big)^{\frac{1}{2}} \Big( \int_{\Omega} t \rho |\mathbf{u}|^2 |\nabla \mathbf{u}_t|^2 dx \Big)^{\frac{1}{2}} \\
&\leq C_\rho*  \| \sqrt{t \rho} \mathbf{u}_t \|_2 \| \mathbf{u} \|_{\infty} \| \sqrt{t} \nabla \mathbf{u}_t \|_2  \\
&\leq \varepsilon \| \sqrt{t} \nabla \mathbf{u}_t \|_2^2 + C_\rho*   \| \mathbf{u} \|_{\infty}^2\| \sqrt{t \rho} \mathbf{u}_t \|_2^2 .
\end{split}
\end{equation}
For \( I_3 \), according to \( \rho_t = - \mathbf{u} \cdot \nabla \rho \) and then performing an integration by parts, we get
\begin{equation}\label{3.6666}
\begin{split}
I_3 &\leq \Big| - \int_{\Omega} (\sqrt{t} \rho_t \mathbf{u} \cdot \nabla \mathbf{u}) \cdot (\sqrt{t} \mathbf{u}_t) dx \Big|\\
&\leq \Big| - \int_{\Omega} t \rho \mathbf{u} \cdot \nabla \big[ (\mathbf{u} \cdot \nabla) \mathbf{u} \cdot \mathbf{u}_t \big] dx \Big| \\
&\leq \int_{\Omega} t \rho |\mathbf{u}| \big( |\nabla \mathbf{u}|^2 |\mathbf{u}_t| + |\mathbf{u}| |\nabla^2 \mathbf{u}| |\mathbf{u}_t| + |\mathbf{u}| |\nabla \mathbf{u}| |\nabla \mathbf{u}_t|)dx  \\
&\triangleq \sum_{i=1}^6 I_{3i}.
\end{split}
\end{equation}
Now we deal with the terms $I_{3i}$$(i=1,2,3,4,5,6)$ in the above inequality \eqref{3.6666}.  For $I_{31}$, it follows from H\"{o}lder's and Young's inequalities and $\dot{H}^1\hookrightarrow L^6(\Omega)$ that
\begin{align*}
I_{31} &\leq \sqrt{\rho^* T} \| \sqrt{\rho t} \mathbf{u}_t \|_4 \| \mathbf{u} \|_6 \| \nabla \mathbf{u} \|_{24/7}^2 \\
&\leq \sqrt{\rho^* T} \| \sqrt{\rho t} \mathbf{u}_t \|_2^{1/4} \| \sqrt{\rho t} \mathbf{u}_t \|_6^{3/4} \| u \|_6 \| \nabla \mathbf{u} \|_{24/7}^2 \\
&\leq \varepsilon \| \nabla \sqrt{t} u_t \|_2^2 + C_{T,\rho^*} \| \sqrt{\rho t} \mathbf{u}_t \|_2^{2/5} \| \nabla \mathbf{u} \|_{24/7}^{16/5} \| \nabla \mathbf{u} \|_2^{8/5}.
\end{align*}
Due to
\[
\| \nabla \mathbf{u} \|_{24/7}^{16/5} \leq C \| \nabla \mathbf{u} \|_2^{6/5} \| \nabla^2 \mathbf{u} \|_2^2,
\]
thus
\begin{align*}
I_{31} &\leq \varepsilon \| \nabla \sqrt{t} \mathbf{u}_t \|_2^2 + C_{T,\rho^*} \| \sqrt{\rho t} \mathbf{u}_t \|_2^{2/5} \| \nabla \mathbf{u} \|_2^{14/5} \| \nabla^2 \mathbf{u} \|_2^2 \\
&\leq \varepsilon \| \nabla \sqrt{t} \mathbf{u}_t \|_2^2 + C_{T,\rho^*} \left( \| \sqrt{\rho t} \mathbf{u}_t \|_2^2 + \| \nabla \mathbf{u}\|_2^{7/2} \right) \| \nabla^2\mathbf{u} \|_2^2.
\end{align*}
Similarly, we also  get
\begin{align*}
I_{32} &= \int_{\Omega} t \rho |\mathbf{u}|^2 | \nabla^2 \mathbf{u} | |\mathbf{u}_t| dx \leq  C _{T \rho^{\ast} } \| \nabla^2\mathbf{u} \|_2^2 + C\| \mathbf{u} \|_{\infty}^4 \| \sqrt{\rho t} \mathbf{u}_t \|_2^2,\\
I_{33} &= \int_{\Omega} t \rho |\mathbf{u}|^2 \|\nabla \mathbf{u}\| \|\nabla \mathbf{u}_t\| dx \\
&\leq \varepsilon \int_{\Omega} |\nabla \sqrt{t} \mathbf{u}_t|^2 dx + C \int_{\Omega} t \rho^2 |\mathbf{u}|^4 \|\nabla \mathbf{u}\|^2 dx \\
&\leq \varepsilon |\nabla \sqrt{t} \|\mathbf{u}_t\|_2^2  + C_{T\rho*} \| \mathbf{u} \|_{\infty}^4 \|\nabla \mathbf{u}\|_2^2.
\end{align*}
For $I_4$ and $I_5$, we conclude from  $\dot{H}^1\hookrightarrow L^6(\Omega)$, that
\begin{align}
I_4 &\leq \left| -\int _{\Omega}\sqrt{t} \rho \mathbf{u}_t \cdot \nabla \mathbf{u} \cdot \sqrt{t} \mathbf{u}_t dx\right| \leq C_{\rho*}\| \nabla \mathbf{u} \|_3 \| \sqrt{\rho t} \mathbf{u}_t \|_2 \| \sqrt{t} \mathbf{u}_t \|_6\\ \notag
&\leq C_{\rho*} \| \nabla \mathbf{u} \|_3 \| \sqrt{\rho t} \mathbf{u}_t \|_2 \| \sqrt{t} \nabla \mathbf{u}_t \|_2 \leq \varepsilon \| \sqrt{t} \nabla \mathbf{u}_t \|_2^2 + C_{\rho*} \| \nabla \mathbf{u} \|_3^2 \| \sqrt{\rho t} \mathbf{u}_t \|_2^2 ,\\
I_5 &\leq \left| -\int _{\Omega}\sqrt{t} \rho \mathbf{u} \cdot \nabla \mathbf{u}_t \cdot \sqrt{t} \mathbf{u}_t dx\right| \leq C_{\rho*} \| \sqrt{\rho t} \mathbf{u}_t \|_2 \| \mathbf{u} \|_{\infty} \| \sqrt{t} \nabla \mathbf{u}_t \|_2 \\ \notag
&\leq \varepsilon \| \sqrt{t} \nabla \mathbf{u}_t \|_2^2 + C_{\rho*} \| \mathbf{u}\|_{\infty}^2 \| \sqrt{\rho t} \mathbf{u}_t \|_2^2 .
\end{align}
For $I_6$, due to $\nu = \nu(s) \in W^{1,\infty}(\mathbb{R})$, $\dot{H}^1\hookrightarrow L^6(\Omega)$, we get
\begin{align}
I_6 &\leq \left|- \int _{\Omega}\nu'(\phi) \phi_t D \mathbf{u} \cdot \nabla t\mathbf{u}_t dx \right| \leq C \| \sqrt{t} \nabla \mathbf{u}_t \|_2 \|\sqrt t \phi_t \|_6 \| D \mathbf{u} \|_3\\ \notag
&\leq \varepsilon \| \sqrt{t} \nabla \mathbf{u}_t \|_2^2 + C\| \nabla \mathbf{u} \|_3^2 \|\nabla\sqrt t\phi_t \|_2^2.
\end{align}
For \( I_7\), according to \( \rho_t = - \mathbf{u}\cdot \nabla \rho \) and then performing an integration by parts, we get
\begin{equation}
\begin{split}\label{3.1000}
I_7 &\leq \left| -\int_{\Omega} (\sqrt{t} \, \text{div} \rho \mathbf{u}) \cdot(\mu \nabla\phi\cdot \sqrt{t} \, \mathbf{u}_t) \, dx \right| \leq\left| \int_{\Omega} t \, \rho \mathbf{u} \cdot \nabla (\mu \nabla\phi\cdot u_t) \, dx \right| \\
&\leq \left| \int_{\Omega} t \, \rho \mathbf{u} \cdot \nabla \mu \cdot\nabla \phi\cdot \mathbf{u}_t \, dx \right| + \left| \int_{\Omega} t \, \rho \mathbf{u} \,\cdot \mu \nabla^2 \phi \cdot \mathbf{u}_t \, dx \right| + \left| \int_{\Omega} t \, \rho \mathbf{u} \,\cdot \mu \nabla \phi \cdot \nabla \mathbf{u}_t \, dx \right|
\\&\triangleq \sum_{i=1}^3 I_{7i}.
\end{split}
\end{equation}
We deal with  the terms $I_{7i}$$(i=1,2,3)$ in the above inequality \eqref{3.1000}. It follows from H\"{o}lder's and Young's inequalities that
\begin{align*}
I_{71} &\leq C_{T} \rho * \, \|\sqrt{\rho t} \, \mathbf{u}_t\|_2 \, \|\mathbf{u}\|_{\infty} \, \|\nabla \mu\|_4 \, \|\nabla \phi\|_4 \leq  C_{T} \rho *\|\mathbf{u}\|_{\infty}^2 \, \|\sqrt{\rho t} \, \mathbf{u}_t\|_2^2 + C\|\nabla \mu\|_4^2 \, \|\nabla \phi\|_4^2, \\
I_{72} &\leq C_{T} \rho * \, \|\sqrt{\rho t} \, \mathbf{u}_t\|_2 \, \|\mathbf{u}\|_{\infty} \, \|\nabla^2 \phi\|_4 \, \| \mu\|_4 \leq  C_{T} \rho *\|\mathbf{u}\|_{\infty}^2 \, \|\sqrt{\rho t} \, \mathbf{u}_t\|_2^2 + C\|\nabla^2 \phi\|_4^2 \, \| \mu\|_4^2, \\
I_{73} &\leq C_{T} \rho *  \, \|\nabla \sqrt t\mathbf{u}_t\|_2 \, \|\mathbf{u}\|_{\infty} \, \|\mu\|_4 \, \|\nabla \phi\|_4 \leq \varepsilon \|\nabla\, \sqrt t \mathbf{u}_t\|_2^2 + C_{\rho*T} \|\mathbf{u}\|_{\infty}^2 \, \|\mu\|_4^2 \, \|\nabla \phi\|_4^2.
\end{align*}
For \( I_8\) and \( I_9\),  according to H\"{o}lder's and Young's inequalities,  we have
\begin{align}
I_8 &\leq \left| \int_{\Omega} ( \sqrt{t}\rho \, \mu_t \, \nabla \phi) \cdot\, (\sqrt{t} \, \mathbf{u}_t) \, dx \right|\notag
\\& \leq C_{\rho*} \, \|\sqrt{\rho t} \, \mu_t\|_2 \, \|\nabla \, \phi\|_4 \, \|\sqrt{t}\mathbf{u}_t\|_4\\ \notag
&\leq C_{\rho*T} \, \|\sqrt{\rho t} \, \mu_t\|_2 \, \|\nabla \, \phi\|_4 \, \|\nabla\sqrt{t}\mathbf{u}_t\|_2^\frac{3}{4}\|\mathbf{u}_t\|_2^\frac{1}{4}\notag\\
 &\leq \varepsilon\|\nabla \sqrt{t} \, \mathbf{u}_t\|_2^2 + C_{\rho*T} \, \|\nabla \phi\|_4^8 \, \|\mathbf{u}_t\|_2^2+\varepsilon\|\sqrt{\rho t} \, \mu_t\|_2^2 ,\notag\\
I_9 &\leq \left| \int_{\Omega} (\sqrt{t} \, \rho \, \mu \, \nabla \phi_t) \,\cdot (\sqrt{t} \, \mathbf{u}_t) \, dx \right| \notag\\&\leq C_{\rho*} \, \|\sqrt{\rho t} \, \mathbf{u}_t\|_2 \, \|\nabla \sqrt{t} \, \phi_t\|_2 \, \|\mu\|_{\infty} \\ \notag
&\leq C_{\rho*} \, \|\sqrt{\rho t} \, \mathbf{u}_t\|_2^2 + C\|\mu\|_{\infty}^2 \, \|\nabla \sqrt{t} \, \phi_t\|_2^2 .
\end{align}
For \( I_{10} \), thanks to \( \rho_t = - \mathbf{u} \cdot \nabla \rho \), we deduce that
\begin{equation}
\begin{split}\label{3.1444}
I_{10} &\leq \left| \int_{\Omega} \sqrt{t} \, \text{div} \, (\rho \, \mathbf{u}) \cdot\,\bigl((\phi^3 - \phi) \, \nabla \phi \,\cdot \sqrt{t} \, \mathbf{u}_t\bigr) \, dx \right|
\leq\left| \int_{\Omega} \rho \, \mathbf{u} \cdot \nabla \bigl(t(\phi^3 - \phi) \, \cdot \nabla \phi \, \mathbf{u}_t \bigr) \, dx \right|\\
&\leq \left| 3\int_{\Omega} \rho \, \mathbf{u}\cdot \left| \nabla \phi \right|^2|\phi|^2\cdot \, t \, \mathbf{u}_t dx \right| + \left| \int_{\Omega} \rho \, \mathbf{u} \cdot\left| \nabla \phi \right|^2 \,\cdot t \, \mathbf{u}_t dx \right| \\
&\quad + \left| \int_{\Omega} \rho \, \mathbf{u} \,\cdot (\phi^3 - \phi) \, \nabla^2 \phi \, \cdot t \mathbf{u}_t dx \right| + \left| \int_{\Omega} \rho \, \mathbf{u} \, \cdot (\phi^3 - \phi) \, \nabla \phi \, \cdot\nabla t\, \mathbf{u}_t dx \right|
\\&\triangleq \sum_{i = 1}^{4} I_{10(i)}.
\end{split}
\end{equation}
Next, we turn to the estimates of $I_{10(i)}(i=1,2,3,4)$ in the above inequality \eqref{3.1444}. Thanks to H\"{o}lder's and Young's inequalities, we have
\begin{equation*}
\begin{split}
I_{10(1)} &\leq C_{\rho*} \, \|\sqrt{\rho t} \, \mathbf{u}_t\|_2 \, \|\mathbf{u}\|_{\infty} \, \| \left| \nabla \phi \right|^2 \|_4\| \left|\phi \right|^2 \|_4\\& \leq C_{\rho*T} \|\mathbf{u}\|_{\infty}^2 \, \|\sqrt{\rho t} \, \mathbf{u}_t\|_2^2 + C \, \|\nabla \phi\|_8^2\|\phi\|_8^2,\\
I_{10(2)} &\leq C_{\rho*T} \|\sqrt{\rho_t} \mathbf{u}_t\|_{2} \|\mathbf{u}\|_{\infty} \| \nabla\phi \|_{4}^2 \leq C_{\rho*T} \|\mathbf{u}\|_{\infty}^2 \|\sqrt{\rho t} \mathbf{u}_t\|_{2}^2 + C \|\nabla \phi \|_{4}^4 ,\\
I_{10(3)} &\leq C_{\rho*T} \|\sqrt{\rho t} \mathbf{u}_t\|_{2} \|\nabla^2 \phi\|_{4} \left( \||\phi|^3\|_{4} + \|\phi\|_{4} \right) \|\mathbf{u}\|_{\infty}\\& \leq C_{\rho*T}\|\mathbf{u}\|_{\infty}^2 \|\sqrt{\rho t} \mathbf{u}_t\|_{2}^2 + C \|\nabla^2 \phi\|_{4}^2 \left( \|\phi\|_{12}^3 + \|\phi\|_{4} \right)^2,\\
I_{10(4)} &\leq C_{\rho*T} \|\nabla\sqrt{t} \mathbf{u}_t\|_{2} \|u\|_{\infty} \left( \||\phi|^3\|_{4} + \|\nabla\phi\|_{4} \right) \|\nabla\phi\|_{4} \\ &\leq \varepsilon\|\nabla\sqrt{t} \mathbf{u}_t\|_{2}^2 + C_{\rho*T} \|\mathbf{u}\|_{\infty}^2 \left( \|\phi\|_{12}^3 + \|\phi\|_{4} \right)^2 \|\nabla\phi\|_{4}^2 .
\end{split}
\end{equation*}
For  $I_{11}$  and $I_{12}$, according to $\Phi_0(s) = \frac{1}{4}(s^2 - 1)^2 \quad \forall s \in \mathbb{R}$, we have
\begin{align}\notag
I_{11} &\leq \left| -\int_{\Omega} \sqrt{t}\rho(3\phi^2-1) \phi_t \nabla\phi\cdot \sqrt{t} \mathbf{u}_t \, dx \right| \\
&\leq C_{\rho*T} \|\sqrt{\rho t} \mathbf{u}_t\|_{2} (\|\phi\|_{\infty}^2+1) \| \phi_t\|_{4} \|\nabla\phi\|_{4} \\ \notag
&\leq  C_{\rho*T}(\|\phi\|_{\infty}^2+1)^2 \|\nabla\phi\|_{4}^2 \| \phi_t\|_{4}^2 + C \|\sqrt{\rho t} \mathbf{u}_t\|_{2}^2 ,\\
I_{12} &\leq \left| -\int_{\Omega} \sqrt{t} \rho(\phi^3 - \phi) \nabla\phi_t \cdot\sqrt{t} \mathbf{u}_t \, dx \right| \leq C_{\rho*T} \|\nabla\phi_t\|_{2} \|\sqrt{\rho t} \mathbf{u}_t\|_{2} \left( \|\phi\|_{\infty}^3 + \|\phi\|_{\infty} \right) \\ \notag
&\leq C_{\rho*T} \|\nabla\phi_t\|_{2}^2 \|\sqrt{\rho t} \mathbf{u}_t\|_{2}^2 + C\left( \|\phi\|_{\infty}^3 + \|\phi\|_{\infty} \right)^2.
\end{align}
For \( K_1\), according to \( \rho_t = - \mathbf{u} \cdot \nabla \rho \) and then performing an integration by parts, we get
\begin{equation}
\begin{split}\label{3.16666}
K_1 &\leq \left| \int_\Omega \sqrt{t} \, \text{div}(\rho \mathbf{u}) \, \phi_t \, \sqrt{t} \, \mu_t \, dx \right|
\leq\left| \int_\Omega \rho \mathbf{u} \cdot \nabla(t\phi_t \mu_t) \, dx \right| \\
&\leq \left| \int_\Omega t\rho \mathbf{u} \cdot\,\nabla \phi_t \,  \, \mu_t \, dx \right|
+ \left| \int_\Omega t\rho \mathbf{u} \,\cdot \phi_t \, \nabla \mu_t \, dx \right|
\\&\triangleq\sum_{i = 1}^{2} K_{1i}.
\end{split}
\end{equation}
Now we  bound the terms $ K_{1i}(i=1,2)$ in the above inequality \eqref{3.16666}. It follows from H\"{o}lder's and Young's inequalities that
\begin{align*}
K_{11} &\le C_{\rho*T} \, \|\sqrt{t} \, \nabla \phi_t\|_2 \, \|\mathbf{u}\|_\infty \, \|\sqrt{\rho t} \, \mu_t\|_2
\le\varepsilon \|\sqrt{\rho t} \,\mu_t\|_2^2
+ C_{\rho*T}   \|\mathbf{u}\|_\infty^2\,\|\sqrt{t} \, \nabla \phi_t\|_2^2,\\
K_{12} &\le C_{\rho*T} \, \|\sqrt{t} \, \nabla \mu_t\|_2 \, \|\phi_t\|_4 \, \|\mathbf{u}\|_4
\le \varepsilon \, \|\sqrt{t}\nabla \, \mu_t\|_2^2
+ C_{\rho*T} \, \|\mathbf{u}\|_4^2 \, \|\phi_t\|_4^2.
\end{align*}
For \( K_2\) and  \( K_3\),  using \( \rho_t = - \mathbf{u} \cdot \nabla \rho \) and then performing an integration by parts, we have
\begin{align} \notag
K_2 + K_3 &=  -\int_\Omega (\sqrt{t} \, \phi_t)_t \, \sqrt{t} \, \rho_t \, \phi^3 \, dx
+ \int_\Omega (\sqrt{t} \, \phi_t)_t \, \sqrt{t} \, \rho_t \, \phi \, dx =  \int_\Omega \rho \mathbf{u} \cdot\nabla \bigl( (\sqrt{t} \, \phi_t)_t \, \sqrt{t} \, (\phi - \phi^3) \bigr) dx \\
&= \int_\Omega \rho \mathbf{u}  \cdot\nabla  (\sqrt{t} \, \phi_t)_t \, \sqrt{t} \, (\phi - \phi^3)  dx
+ \int_\Omega \rho \mathbf{u} \, \cdot(\sqrt{t} \, \phi_t)_t \,\sqrt{t} (\nabla \phi - 3\phi^2 \nabla \phi) dx\\
&\triangleq \sum_{i = 1}^{2} K_{2i} \notag.
\end{align}
For $K_{21}$, we have
\begin{align*}
K_{21}&=\int_{\Omega}\frac{1}{2}\frac{1}{\sqrt{t}}\nabla\phi_{t}{\sqrt{t}}(\phi-\phi^3)\cdot\rho \mathbf{u}\mathrm{d}x+\int_{\Omega}\nabla\sqrt{t}\phi_{tt}{\sqrt{t}}(\phi-\phi^3)\cdot\rho \mathbf{u}\mathrm{d}x\triangleq\sum_{i = 1}^{2}K_{21i}.
\end{align*}
Obviously,
\begin{align*}
K_{211}&=\frac{1}{2}\int_{\Omega}\nabla\phi_{t}(\phi-\phi^3)\cdot\rho \mathbf{u}\mathrm{d}x\leq C_{\rho*}\|\nabla\phi_{t}\|_{2}(\|\phi\|_4+\|\phi^3\|_{4})\|\mathbf{u}\|_{4} \notag
\\&\leq C_{\rho*}\|\nabla\phi_{t}\|_{2}^{2}+C\|\mathbf{u}\|_{4}^{2}(\|\phi\|_4+\|\phi^3\|_{4})^{2}.
\end{align*}
For $K_{212}$, from \eqref{NSCH}$_{4}$,  we have
$$
\rho_{t}\phi_{t}+\rho\phi_{tt}+\rho_{t}\mathbf{u}\cdot\nabla\phi+\rho \mathbf{u}_{t}\cdot\nabla\phi+\rho \mathbf{u}\cdot\nabla\phi_{t}=\Delta\mu_{t},
$$
and then  performing an integration by parts, we conclude that
\begin{equation*}
\begin{split}\label{3.1888}
K_{212}&\leq -C_{\rho^*}\int_{\Omega}\phi_{tt}\cdot \Dv\bigl(t(\phi-\phi^3)\mathbf{u}\bigr)\mathrm{d}x\leq-
C_{\rho^*}\int_{\Omega}\phi_{tt}\cdot t\mathbf{u}\cdot(\nabla\phi-3\phi^2\nabla\phi)\mathrm{d}x\\
&\leq -C_{\rho^*\rho_*}\int_{\Omega}t \mathbf{u}\cdot(\nabla\phi-3\phi^2\nabla\phi)\left(-\rho_{t}\phi_{t}-\rho_{t}\mathbf{u}\cdot\nabla\phi-\rho \mathbf{u}_{t}\cdot\nabla\phi-\rho \mathbf{u}u\cdot\nabla\phi_{t}+\Delta\mu_{t}\right)\mathrm{d}x\\
&\leq C_{\rho^*\rho_*}\int_{\Omega}t\mathbf{u}\cdot(\nabla\phi-3\phi^2\nabla\phi) \rho_t \phi_{t}\mathrm{d}x+C_{\rho^*\rho_*}\int_{\Omega}t\mathbf{u}\cdot(\nabla\phi-3\phi^2\nabla\phi) \rho_t\mathbf{u}\cdot\nabla\phi\mathrm{d}x\\
&\quad+C_{\rho^*\rho_*}\int_{\Omega}t\mathbf{u}\cdot(\nabla\phi-3\phi^2\nabla\phi)
\cdot\rho \mathbf{u}_{t}\cdot\nabla\phi\mathrm{d}x
+C_{\rho^*\rho_*}\int_{\Omega}t\mathbf{u}\cdot(\nabla\phi-3\phi^2\nabla\phi)\cdot \rho \mathbf{u}\cdot\nabla\phi_{t}\mathrm{d}x\\
&\quad-C_{\rho^*\rho_*}\int_{\Omega}t\mathbf{u}\cdot
(\nabla\phi-3\phi^2\nabla\phi)
\cdot\Delta \mu_t\mathrm{d}x\\&\triangleq\sum_{i = 1}^{5}K_{212i}.
\end{split}
\end{equation*}
We shall bound these terms $K_{212i}(i=1,2,3,4,5)$  in the above inequality as follows.
Thanks to  \( \rho_t = - \mathbf{u} \cdot \nabla \rho \) and H\"{o}lder's and Young's inequalities, we have
\begin{align*}
K_{2121} &\leq C_{\rho^*\rho_*} \int_{\Omega} \nabla \bigl(t \mathbf{u} \cdot (\nabla\phi-3\phi^2\nabla\phi) \phi_{t}\bigr) \cdot \rho \mathbf{u} \, dx \leq +C_{\rho^*\rho_*} \int_{\Omega} \nabla t \mathbf{u} \cdot (\nabla\phi-3\phi^2\nabla\phi) \cdot \phi_{t}  \rho \mathbf{u} \, dx\\
 &\quad+C_{\rho^*\rho_*}\int_{\Omega} t \mathbf{u} \cdot (\nabla^2\phi-6\phi|\nabla\phi|^2-3\phi^2\nabla^2\phi)\cdot \phi_{t} \rho \mathbf{u} \, dx
 +C_{\rho^*\rho_*} \int_{\Omega} t \mathbf{u} \cdot(\nabla\phi-3\phi^2\nabla\phi) \cdot \nabla\phi_{t} \cdot\rho \mathbf{u} \, dx \\
&\leq C_{\rho^*\rho_*T}\|\phi_{t}\|_{6} \|\nabla \mathbf{u}\|_{2} (\|\nabla\phi\|_{6}+\|\phi\|_{\infty}^2\|\nabla\phi\|_{6}) \|\mathbf{u}\|_{6} + C_{\rho*\rho_*T} \|\phi_{t}\|_{2}( \|\nabla^2 \phi\|_{6}+\||\nabla\phi|^2\|_{6}\|\phi\|_{\infty}\\
&\quad+\|\nabla^2 \phi\|_{6}\|\phi\|_{\infty}^2) \||\mathbf{u}|^2\|_{3} + C_{\rho*\rho_*T} \|\nabla\phi_{t}\|_{2} (\|\nabla\phi\|_2+\|\phi\|_{\infty}^2\|\nabla\phi\|_2\|) \|\mathbf{u}\|_{\infty}^{2} \\
&\leq C_{\rho^*\rho_*T} \|\phi_{t}\|_{6}^{2} \|\nabla \mathbf{u}\|_{2}^{2}  + C \|\mathbf{u}\|_{6}^{2} (\|\nabla\phi\|_{6}+\|\phi\|_{\infty}^2\|\nabla\phi\|_{6})^2 \\
&\quad+ C_{\rho*\rho_*T} \|\mathbf{u}\|_{6}^{4} \|\phi_{t}\|_{2}^{2}
+ C ( \|\nabla^2 \phi\|_{6}+\||\nabla\phi|^2\|_{6}\|\phi\|_{\infty}+\|\nabla^2 \phi\|_{6}\|\phi\|_{\infty}^2) ^{2} \\
&\quad+ C_{\rho*\rho_*T} (\|\nabla\phi\|_2+\|\phi\|_{\infty}^2\|\nabla\phi\|_2\|)^2\|\nabla\phi_{t}\|_{2}^{2}
+ C \|\mathbf{u}\|_{\infty}^{4},\\
K_{2122} &\leq C_{\rho^*\rho_*} \int_{\Omega} \nabla \bigl(t \mathbf{u} \cdot(\nabla\phi-3\phi^2\nabla\phi) \cdot \mathbf{u} \cdot \nabla \phi) \cdot \rho \mathbf{u} \, dx  \leq C_{\rho^*\rho_*}\int_{\Omega} \nabla t \mathbf{u} \cdot (\nabla\phi-3\phi^2\nabla\phi) \cdot \mathbf{u} \cdot \nabla \phi \cdot \rho \mathbf{u} \, dx \\
&\quad+ C_{\rho^*\rho_*}\int_{\Omega} t \mathbf{u} \cdot (\nabla^2\phi-6\phi|\nabla\phi|^2-3\phi^2\nabla^2\phi) \cdot \mathbf{u} \cdot \nabla \phi \cdot \rho \mathbf{u} \, dx\\
&\quad+ C_{\rho^*\rho_*}\int_{\Omega} t \mathbf{u} \cdot (\nabla\phi-3\phi^2\nabla\phi)\cdot \nabla \mathbf{u} \cdot \nabla \phi \cdot \rho \mathbf{u} \, dx +C_{\rho^*\rho_*}\int_{\Omega} t \mathbf{u} \cdot(\nabla\phi-3\phi^2\nabla\phi)\cdot \mathbf{u} \cdot \nabla^2 \phi \cdot \rho \mathbf{u} \, dx \\
&\leq C_{\rho^*\rho_*T} \|\mathbf{u}\|_{\infty}^{2} \|\nabla \mathbf{u}\|_{2} (\|\nabla\phi\|_4+\|\nabla \phi\|_{4} \|\phi\|_{\infty}^2)\|\nabla \phi\|_{4} \\
&\quad+ C_{\rho*\rho_*T} \|\mathbf{u}\|_{\infty}^{2}( \|\nabla^2 \phi\|_{4}+ \||\nabla \phi|^2\|_{4}\|\phi\|_{\infty}+\|\phi\|_{\infty}^2\|\nabla^2\phi\|_{4} )\|\mathbf{u}\|_{4}\|\nabla \phi\|_{2}\\
&\quad+ C_{\rho*\rho_*T} \|\mathbf{u}\|_{\infty}^{2} \|\nabla^2 \phi\|_{4} (\|\nabla \phi\|_{2}+ \|\nabla \phi\|_{2}\| \phi\|_{\infty}^2) \|\mathbf{u}\|_{4} \\
&\leq C_{\rho^*\rho_*T} \|\mathbf{u}\|_{\infty}^{4} \|\nabla \mathbf{u}\|_{2}^{2} + C  (\|\nabla\phi\|_4+\|\nabla \phi\|_{4} \|\phi\|_{\infty}^2)^{2} \|\nabla \phi\|_{4}^{2}
\\ &\quad+ C_{\rho*\rho_*T} \|\mathbf{u}\|_{\infty}^{4} \|\nabla \phi\|_{2}^{2} + C \|\mathbf{u}\|_{4}^{2}( \|\nabla^2 \phi\|_{4}+ \||\nabla \phi|^2\|_{4}\|\phi\|_{\infty}+\|\phi\|_{\infty}^2\|\nabla^2\phi\|_{4})^{2}\\
&\quad+ C_{\rho*\rho_*T} \|\mathbf{u}\|_{\infty}^{4}(\|\nabla \phi\|_{2}+ \|\nabla \phi\|_{2}\| \phi\|_{\infty}^2)^{2} + C \|\mathbf{u}\|_{4}^{2} \|\nabla^2 \phi\|_{4}^{2},\\
K_{2123} &\leq C_{\rho^*\rho_*T}\|\sqrt {\rho t}\mathbf{u}_t\|_{2} (\|\nabla \phi\|_{4}+ \|\nabla \phi\|_{4}\| \phi\|_{\infty}^2)  \|\nabla\phi\|_{4} \|\mathbf{u}\|_{\infty}\\
&\leq C_{\rho*\rho_*T}\|\mathbf{u}\|_{\infty}^2 \|\sqrt {\rho t}\mathbf{u}_t\|_{2}^2+C(\|\nabla \phi\|_{4}+ \|\nabla \phi\|_{4}\| \phi\|_{\infty}^2)^2 \|\nabla\phi\|_{4}^2 ,\\
K_{2124} &\leq C_{\rho*\rho_*T}\|\nabla\sqrt t\phi_t\|_{2} \|\mathbf{u}\|_{\infty}^2  (\|\nabla \phi\|_{2}+ \|\nabla \phi\|_{2}\| \phi\|_{\infty}^2)\\
&\leq C_{\rho*\rho_*T}\|\mathbf{u}\|_{\infty}^4 \|\nabla\sqrt t\phi_t\|_{2}^2+C (\|\nabla \phi\|_{2}+ \|\nabla \phi\|_{2}\| \phi\|_{\infty}^2)^2.
\end{align*}
Moreover, by integration by parts and embedding inequality, we conclude that
\begin{align*}
K_{2125}\leq& C_{\rho^*\rho_*} \int_{\Omega} \nabla\bigl(t \mathbf{u} \cdot (\nabla\phi-3\phi^2\nabla\phi)\bigr) \cdot \nabla \mu_t \, dx\\
&\leq C_{\rho^*\rho_*}\int_{\Omega} \nabla t \mathbf{u} \cdot (\nabla\phi-3\phi^2\nabla\phi) \cdot \nabla \mu_t \, dx\\&\quad+C_{\rho^*\rho_*} \int_{\Omega} t \mathbf{u} \cdot (\nabla^2\phi-6\phi|\nabla\phi|^2-3\phi^2\nabla^2\phi)\cdot \nabla \mu_t \, dx\\
&\leq C_{\rho^*\rho_*T} \|\nabla\sqrt {t}\mu_t\|_{2}\|\nabla \mathbf{u}\|_{4} (\|\nabla\phi\|_{4}+\|\phi\|_{\infty}^2\|\nabla\phi\|_{4})\\
&\quad+C_{\rho^*\rho_*T} \|\nabla\sqrt {t}\mu_t\|_{2}^2 ( \|\nabla^2 \phi\|_{4}+\||\nabla\phi|^2\|_{4}\|\phi\|_{\infty}+\|\nabla^2 \phi\|_{4}\|\phi\|_{\infty}^2)\|\mathbf{u}\|_{4}\\
&\leq\varepsilon\|\nabla\sqrt {t}\mu_t\|_{2}^2+C_{\rho^*\rho_*T} (\|\nabla\phi\|_{4}+\|\phi\|_{\infty}^2\|\nabla\phi\|_{4})^2 \|\nabla \mathbf{u}\|_{4}^2+\varepsilon\|\nabla\sqrt {t}\mu_t\|_{2}^2+C_{\rho^*\rho_*T} ( \|\nabla^2 \phi\|_{4}\\
&\quad+\||\nabla\phi|^2\|_{4}\|\phi\|_{\infty}+\|\nabla^2 \phi\|_{4}\|\phi\|_{\infty}^2)^2 \|u\|_{4}^2.
\end{align*}
For $K_{22}$, we  easily see that
\begin{align*}
K_{22}&=\int_{\Omega}\frac{1}{2}\frac{1}{\sqrt{t}}\phi_{t}{\sqrt{t}}
(\nabla\phi-3\phi^2\nabla\phi)\cdot\rho \mathbf{u}\mathrm{d}x+\int_{\Omega}\sqrt{t}\phi_{tt}{\sqrt{t}}(\nabla\phi-3\phi^2\nabla\phi)
\cdot\rho \mathbf{u}\mathrm{d}x\triangleq\sum_{i = 1}^{2}K_{22i}.
\end{align*}
Obviously,
\begin{align*}
K_{221}&=\frac{1}{2}\int_{\Omega}\phi_{t}(\nabla\phi-3\phi^2\nabla\phi)\cdot\rho \mathbf{u}\mathrm{d}x\leq C_{\rho*}\|\phi_{t}\|_{2}(\|\nabla\phi\|_4+\|\phi\|_{\infty}^2\|\nabla\phi\|_{4})
\|\mathbf{u}\|_{4}\\
&\leq C_{\rho*}\|\phi_{t}\|_{2}^{2}+C
(\|\nabla\phi\|_4+\|\phi\|_{\infty}^2\|\nabla\phi\|_{4})
\|\mathbf{u}\|_{4}^{2}.
\end{align*}
For $K_{222}$, the derivation  is totally  similar to $K_{212}$, we omit it. 
For \( K_4\), by performing an integration by parts on time $t$, we have
\begin{align}
K_4=-\frac{3}{2}\int_{\Omega}\frac{d}{dt}(|\sqrt{t}\phi_t|^2\rho\phi^2)dx + \frac{3}{2}\int_{\Omega}|\sqrt{t}\phi_t|^2\rho_t\phi^2dx
+3\int_{\Omega}|\sqrt{t}\phi_t|^2\rho\phi\phi_tdx
\triangleq\sum_{i = 1}^3 K_{4i}.
\end{align}
We bound term by term above in what follows. According  to \( \rho_t = - \mathbf{u} \cdot \nabla \rho \), we deduce that
\begin{align*}
K_{42}&=-\frac{3}{2}\int_{\Omega}\nabla(|\sqrt{t}\phi_t|^2|\phi|^2)\cdot\rho \mathbf{u}\ dx\\
&=-\frac{3}{2}\int_{\Omega}2(\sqrt{t}\phi_t)\nabla(\sqrt{t}\phi_t)|\phi|^2\
\cdot\rho \mathbf{u}\ dx
-\frac{3}{2}\int_{\Omega}|\sqrt{t}\phi_t|^2 2\phi\nabla\phi\ \cdot\rho \mathbf{u}\ dx\\
&\leq C_{\rho^*T}\|\nabla\sqrt{t}\phi_t\|_2\|\phi_t\|_2\|\phi\|_{\infty}^2\|\mathbf{u}\|_{\infty}
+C_{\rho^*T}\|\phi_t\|_6^2\|\phi\|_6\|\nabla\phi\|_4\|\mathbf{u}\|_4\\
&\leq C_{\rho^*T}\|\nabla\sqrt{t}\phi_t\|_2^2\|\phi_t\|_2^2 + C\|\mathbf{u}\|_{\infty}^2\|\phi\|_{\infty}^4+ C_{\rho^*T}\|\phi_t\|_6^2\|\phi\|_6\|\nabla\phi\|_4\|\mathbf{u}\|_4.
\end{align*}
It follows from  H\"{o}lder's and Young's inequalities,  that
\begin{align*}
K_{43}&\leq C_{\rho^*T}\|\sqrt{t}\phi_t\|_6^2\|\phi_t\|_2\|\phi\|_6\leq C_{\rho^*T}\|\nabla\sqrt{t}\phi_t\|_2^2\|\phi_t\|_2\|\phi\|_6.
\end{align*}
For $K_{5}$, we have
\begin{align*}
K_{5}&=\int_{\Omega}\frac{1}{2}\frac{1}{\sqrt{t}}\rho\phi_{t}{\sqrt{t}}\phi_t\mathrm{d}x
+\int_{\Omega}\sqrt{t}\phi_{tt}{\sqrt{t}}\rho\phi_t \mathrm{d}x\triangleq\sum_{i = 1}^{2}K_{5i}.
\end{align*}
Obviously,
\begin{align*}
K_{51}&=\frac{1}{2}\int_{\Omega}\rho|\phi_t|^2\mathrm{d}x\leq C_{\rho*}\|\phi_t\|_{2}^2.
\end{align*}
For $K_{52}$, by \eqref{NSCH}$_{4}$, we have
$$
\rho_{t}\phi_{t}+\rho\phi_{tt}+\rho_{t}\mathbf{u}\cdot\nabla\phi+\rho \mathbf{u}_{t}\cdot\nabla\phi+\rho \mathbf{u}\cdot\nabla\phi_{t}=\Delta\mu_{t},
$$
and then  performing  an integration by parts,  we arrive at
\begin{equation*}
\begin{split}\label{3.2000}
K_{52}&=\int_{\Omega}t\rho\phi_{tt}\phi_{t}\mathrm{d}x
= \int_{\Omega}t \phi_{t}\left(-\rho_{t}\phi_{t}-\rho_{t}\mathbf{u}\cdot\nabla\phi-\rho \mathbf{u}_{t}\cdot\nabla\phi-\rho \mathbf{u}\cdot\nabla\phi_{t}+\Delta\mu_{t}\right)\mathrm{d}x\\
&=-\int_{\Omega}t\phi_{t}\rho_t \phi_{t}\mathrm{d}x-\int_{\Omega} t\phi_{t} \rho_t\mathbf{u}\cdot\nabla\phi\mathrm{d}x
-\int_{\Omega}t\phi_{t}
\cdot\rho \mathbf{u}_{t}\cdot\nabla\phi\mathrm{d}x\\
&\quad-\int_{\Omega}t\phi_{t}\cdot \rho \mathbf{u}\cdot\nabla\phi_{t}\mathrm{d}x+\int_{\Omega}t\phi_{t}
\cdot\Delta \mu_t\mathrm{d}x\\&\triangleq\sum_{i = 1}^{5}K_{52i}.
\end{split}
\end{equation*}
Next, we bound  these terms  in the above inequlity. By \( \rho_t = - \mathbf{u} \cdot \nabla \rho \) and H\"{o}lder's and Young's inequalities, we have
\begin{align*}
K_{521} &= -\int_{\Omega} \nabla t |\phi_t|^2 \cdot \rho \mathbf{u} \, dx = -2\int_{\Omega} t\phi_t \nabla\phi_t \cdot \rho \mathbf{u} \, dx
\\&\leq C_{\rho*T} \| \mathbf{u}\|_{\infty} \| \nabla\sqrt t \phi_t \|_2 \| \phi_t \|_2
\leq C_{\rho*T} \| \mathbf{u} \|_{\infty}^2 \| \nabla\sqrt t \phi_t \|_2^2 + C \| \phi_t \|_2^2, \\
K_{522} &= -\int_{\Omega} \nabla(t\phi_t \mathbf{u} \cdot \nabla \phi) \cdot \rho \mathbf{u} \, dx
\\&= -\int_{\Omega} \nabla t \phi_t \cdot \mathbf{u} \cdot \nabla \phi \cdot \rho \mathbf{u} \, dx - \int_{\Omega} t\phi_t \nabla \mathbf{u} \cdot \nabla \phi \cdot \rho \mathbf{u} \, dx - \int_{\Omega} t\phi_t \mathbf{u} \cdot \nabla^2 \phi \cdot \rho \mathbf{u} \, dx \\
&\leq C_{\rho*T} \| \nabla \sqrt t \phi_t \|_2 \| \mathbf{u} \|_{\infty}^2 \| \nabla \phi \|_2 + C_{\rho*T} \| \phi_t \|_6 \| \nabla \mathbf{u} \|_2 \| \nabla \phi \|_6 \| \mathbf{u} \|_6 + C_{\rho*T} \| \nabla ^2\phi \|_2 \| \phi_t \|_2 \| \mathbf{u} \|_{\infty}^2 \\
&\leq C_{\rho*T} \| \mathbf{u} \|_{\infty}^4 \| \nabla\sqrt t \phi_t \|_2^2 + C \| \nabla \phi \|_2^2 + C_{\rho*T} \| \phi_t \|_6^2 + C \| \nabla \mathbf{u} \|_2^2 \| \nabla \phi \|_6^2 \| \mathbf{u} \|_6^2 \\&\quad+ C_{\rho*T} \| \phi_t \|_2^2 + C \| \nabla ^2\phi \|_2^2 \| \mathbf{u} \|_{\infty}^4, \\
K_{523} &\leq C_{\rho*T} \| \sqrt{\rho t} \mathbf{u}_t \|_2 \| \phi_t \|_4 \| \nabla \phi \|_4 \leq C_{\rho*T} \| \nabla \phi \|_4^2 \| \sqrt{\rho t} \mathbf{u}_t \|_2^2 + C \| \phi_t \|_4^2, \\
K_{524} &\leq C_{\rho*T} \| \nabla \sqrt t \phi_t \|_2 \| \phi_t \|_2 \| \mathbf{u} \|_{\infty} \leq C_{\rho*T} \| \mathbf{u} \|_{\infty}^2 \|\nabla \sqrt t \phi_t \|_2^2 + C \| \phi_t \|_2^2,\\
K_{525} &= -\int_{\Omega} \nabla t \phi_t \cdot \nabla \mu_t \, dx \leq \| \nabla \sqrt t \mu_t \|_2 \| \nabla \sqrt t \phi_t \|_2 \leq \varepsilon \| \nabla \sqrt t \mu_t \|_2^2 + C \| \nabla \sqrt t\phi_t \|_2^2.
\end{align*}
For $K_6$, it follows from \( \rho_t = - \mathbf{u}\cdot \nabla \rho \), that
\begin{align}
K_6&=-\int_{\Omega}(\sqrt{t}\phi_t)_t\sqrt{t}\mu\ \text{div}(\rho\mathbf{u})dx = \int_{\Omega}\nabla\bigl((\sqrt{t}\phi_t)_t\sqrt{t}\mu\bigr)\cdot\rho\mathbf{u}dx\notag\\
&=\int_{\Omega}\nabla(\sqrt{t}\phi_t)_t\sqrt{t}\mu \cdot \rho \mathbf{u}\ dx+\int_{\Omega}(\sqrt{t}\phi_t)_t\nabla\sqrt{t}\mu\cdot\rho \mathbf{u}\ dx
\\&\triangleq\sum_{i = 1}^2 K_{6i}\notag.
\end{align}
For $K_{61}$, we have
\begin{align*}
K_{61}&=\int_{\Omega}\frac{1}{2}\frac{1}{\sqrt{t}}\nabla\phi_{t}{\sqrt{t}}\mu\cdot\rho \mathbf{u}\mathrm{d}x+\int_{\Omega}\nabla\sqrt{t}\phi_{tt}{\sqrt{t}}\mu\cdot\rho \mathbf{u}\mathrm{d}x\triangleq\sum_{i = 1}^{2}K_{61i}.
\end{align*}
Obviously,
\begin{align*}
K_{611}&=\frac{1}{2}\int_{\Omega}\nabla\phi_{t}\mu\cdot\rho \mathbf{u}\mathrm{d}x\leq C_{\rho*}\|\nabla\phi_{t}\|_{2}\|\mu\|_{4}\|\mathbf{u}\|_{4}\leq C_{\rho*}\|\nabla\phi_{t}\|_{2}^{2}+C\|\mathbf{u}\|_{4}^{2}\|\mu\|_{4}^{2}.
\end{align*}
For $K_{612}$, by \eqref{NSCH}$_{4}$, we have
$$
\rho_{t}\phi_{t}+\rho\phi_{tt}+\rho_{t}\mathbf{u}\mathbf{u}\cdot\nabla\phi+\rho \mathbf{u}_{t}\cdot\nabla\phi+\rho \mathbf{u}\cdot\nabla\phi_{t}=\Delta\mu_{t}.
$$
Furthermore, it follows from  performing an integration by parts, that
\begin{align*}
K_{612}&\leq C_{\rho^*}\int_{\Omega}\phi_{tt}\cdot \Dv(t\mu \mathbf{u})\mathrm{d}x\leq
C_{\rho^*}\int_{\Omega}\phi_{tt}\cdot t\mathbf{u}\cdot\nabla\mu\mathrm{d}x\\
&\leq C_{\rho^*\rho_*}\int_{\Omega}t \mathbf{u}\cdot\nabla\mu\left(-\rho_{t}\phi_{t}-\rho_{t}\mathbf{u}\cdot\nabla\phi-\rho \mathbf{u}_{t}\cdot\nabla\phi-\rho \mathbf{u}\cdot\nabla\phi_{t}+\Delta\mu_{t}\right)\mathrm{d}x\\
&\leq- C_{\rho^*\rho_*}\int_{\Omega}t\mathbf{u}\cdot\nabla \mu\rho_t \phi_{t}\mathrm{d}x-C_{\rho^*\rho_*}\int_{\Omega}tu\cdot\nabla \mu \cdot\rho_t\mathbf{u}\cdot\nabla\phi\mathrm{d}x
-C_{\rho^*\rho_*}\int_{\Omega}t\mathbf{u}\cdot\nabla \mu\cdot\rho \mathbf{u}_{t}\cdot\nabla\phi\mathrm{d}x\\
&\quad-C_{\rho^*\rho_*}\int_{\Omega}t\mathbf{u}\cdot\nabla \mu\cdot \rho \mathbf{u}\cdot\nabla\phi_{t}\mathrm{d}x+C_{\rho^*\rho_*}\int_{\Omega}t\mathbf{u}\cdot\nabla \mu\cdot\Delta \mu_t\mathrm{d}x\\&\triangleq\sum_{i = 1}^{5}K_{612i}.
\end{align*}
We bound term by term above in what follows. By \( \rho_t = - \mathbf{u}\cdot \nabla \rho \) and H\"{o}lder's and Young's inequalities, we have
\begin{align*}
K_{6121} &\leq -C_{\rho^*\rho_*} \int_{\Omega} \nabla (t \mathbf{u} \cdot \nabla \mu \phi_{t}) \cdot \rho \mathbf{u}\, dx \\
&\leq -C_{\rho^*\rho_*} \int_{\Omega} \nabla t \mathbf{u}\cdot \nabla \mu \cdot \phi_{t} \cdot \rho \mathbf{u} \, dx
 - C_{\rho^*\rho_*}\int_{\Omega} t \mathbf{u} \cdot \nabla^2 \mu  \phi_{t} \cdot \rho \mathbf{u} \, dx
 -C_{\rho^*\rho_*} \int_{\Omega} t \mathbf{u} \cdot \nabla \mu \cdot \nabla\phi_{t} \cdot\rho \mathbf{u} \, dx \\
&\leq C_{\rho^*\rho_*T}\|\phi_{t}\|_{6} \|\nabla \mathbf{u}\|_{2} \|\nabla \mu\|_{6} \|\mathbf{u}\|_{6} + C_{\rho*\rho_*T} \|\phi_{t}\|_{2} \|\nabla^2 \mu\|_{6} \||\mathbf{u}|^2\|_{3}  + C_{\rho*\rho_*T} \|\nabla \phi_{t}\|_{2} \|\nabla \mu\|_{2} \|\mathbf{u}\|_{\infty}^{2} \\
&\leq C_{\rho^*\rho_*T} \|\phi_{t}\|_{6}^{2} \|\nabla \mathbf{u}\|_{2}^{2}  + C \|\mathbf{u}\|_{6}^{2} \|\nabla \mu\|_{6}^{2} + C_{\rho*\rho_*T} \|\mathbf{u}\|_{6}^{4} \|\phi_{t}\|_{2}^{2}  + C \|\nabla^2 \mu\|_{6}^{2} \\&\quad + C_{\rho*\rho_*T} \|\nabla \mu\|_{2}^{2} \|\nabla\phi_{t}\|_{2}^{2} + C \|\mathbf{u}\|_{\infty}^{4},\\
K_{6122} &\leq -C_{\rho^*\rho_*} \int_{\Omega} \nabla (t \mathbf{u} \cdot \nabla\mu \cdot \mathbf{u} \cdot \nabla \phi) \cdot \rho \mathbf{u} \, dx
\\& \leq- C_{\rho^*\rho_*}\int_{\Omega} \nabla t \mathbf{u} \cdot \nabla \mu \cdot \mathbf{u} \cdot \nabla \phi \cdot \rho \mathbf{u} \, dx - C_{\rho^*\rho_*}\int_{\Omega} t \mathbf{u} \cdot \nabla ^2\mu \cdot \mathbf{u} \cdot \nabla \phi \cdot \rho \mathbf{u} \, dx\\
&\quad- C_{\rho^*\rho_*}\int_{\Omega} t \mathbf{u} \cdot \nabla \mu \cdot \nabla \mathbf{u} \cdot \nabla \phi \cdot \rho \mathbf{u} \, dx -C_{\rho^*\rho_*}\int_{\Omega} t \mathbf{u} \cdot \nabla \mu \cdot \mathbf{u} \cdot \nabla^2 \phi \cdot \rho \mathbf{u} \, dx \\
&\leq C_{\rho^*\rho_*T} \|\mathbf{u}\|_{\infty}^{2} \|\nabla \mathbf{u}\|_{2} \|\nabla \mu\|_{4} \|\nabla \phi\|_{4} + C_{\rho*\rho_*T} \|\mathbf{u}\|_{\infty}^{2} \|\nabla^2 \mu\|_{4} \|\nabla \phi\|_{4} \|\mathbf{u}\|_{2}  \\&\quad+ C_{\rho*\rho_*T} \|\mathbf{u}\|_{\infty}^{2} \|\nabla^2 \phi\|_{4} \|\nabla \mu\|_{2} \|\mathbf{u}\|_{4} \\
&\leq C_{\rho^*\rho_*T} \|\mathbf{u}\|_{\infty}^{4} \|\nabla \mathbf{u}\|_{2}^{2} + C \|\nabla \mu\|_{4}^{2} \|\nabla \phi\|_{4}^{2} + C_{\rho*\rho_*T} \|\mathbf{u}\|_{\infty}^{4} \|\nabla \phi\|_{4}^{2} + C \|\mathbf{u}\|_{2}^{2} \|\nabla^2 \mu\|_{4}^{2}\\
&\quad+ C_{\rho*\rho_*T} \|\mathbf{u}\|_{\infty}^{4} \|\nabla \mu\|_{2}^{2} + C \|\mathbf{u}\|_{4}^{2} \|\nabla^2 \phi\|_{4}^{2},\\
K_{6123} &\leq C_{\rho^*\rho_*T}\|\sqrt {\rho t}\mathbf{u}_t\|_{2} \|\nabla \mu\|_{4}  \|\nabla\phi\|_{4} \|\mathbf{u}\|_{\infty}\\& \leq C_{\rho*\rho_*T}\|\mathbf{u}\|_{\infty}^2 \|\sqrt {\rho t}\mathbf{u}_t\|_{2}^2+C\|\nabla \mu\|_{4}^2 \|\nabla\phi\|_{4}^2 ,\\
K_{6124} &\leq C_{\rho*\rho_*T}\|\nabla\sqrt t\phi_t\|_{2} \|\mathbf{u}\|_{\infty}^2  \|\nabla \mu\|_{2}\\& \leq C_{\rho*\rho_*T}\|\mathbf{u}\|_{\infty}^4 \|\nabla\sqrt t\phi_t\|_{2}^2+C \|\nabla\mu\|_{2}^2.
\end{align*}
Moreover,  by integration by parts and embedding inequality, we have
\begin{align*}
K_{6125}&\leq- C_{\rho^*\rho_*} \int_{\Omega} \nabla(t \mathbf{u} \cdot \nabla \mu) \cdot \nabla \mu_t \, dx\\
&\leq -C_{\rho^*\rho_*}\int_{\Omega} \nabla t \mathbf{u} \cdot \nabla \mu \cdot \nabla \mu_t \, dx-C_{\rho^*\rho_*}\int_{\Omega} t \mathbf{u} \cdot \nabla^2 \mu\cdot \nabla \mu_t \, dx\\
&\leq C_{\rho^*\rho_*T} \|\nabla\sqrt {t}\mu_t\|_{2}\|\nabla \mathbf{u}\|_{2} \|\nabla\mu\|_{\infty}+C_{\rho^*\rho_*T} \|\nabla\sqrt {t}\mu_t\|_{2}^2\|\nabla^2 \mu\|_{4}\|\mathbf{u}\|_{4}\\
&\leq\varepsilon\|\nabla\sqrt {t}\mu_t\|_{2}^2+C_{\rho^*\rho_*T}\|\nabla \mu\|_{W^{1.6}}^2 \|\nabla \mathbf{u}\|_{2}^2+\varepsilon\|\nabla\sqrt {t}\mu_t\|_{2}^2+C_{\rho^*\rho_*T}\|\nabla^2 \mu\|_{4}^2 \| \mathbf{u}\|_{4}^2.
\end{align*}
For $K_{62}$, we have
\begin{align*}
K_{62}&=\int_{\Omega}\frac{1}{2}\frac{1}{\sqrt{t}}\phi_{t}\nabla{\sqrt{t}}\mu\cdot\rho \mathbf{u}\mathrm{d}x+\int_{\Omega}\sqrt{t}\phi_{tt}\nabla{\sqrt{t}}\mu\cdot\rho \mathbf{u}\mathrm{d}x\triangleq\sum_{i = 1}^{2}K_{62i}.
\end{align*}
Obviously,
\begin{align*}
K_{621}&=\frac{1}{2}\int_{\Omega}\phi_{t}\nabla\mu\cdot\rho \mathbf{u}\mathrm{d}x\leq C_{\rho*}\|\phi_{t}\|_{2}\|\nabla\mu\|_{2}\|\mathbf{u}\|_{\infty}\leq C_{\rho*}\|\phi_{t}\|_{2}^{2}+C\|\mathbf{u}\|_{\infty}^{2}\|\nabla\mu\|_{2}^{2}.
\end{align*}
 By a similar derivation of $K_{612}$,  we can bound   $K_{622}$ and omit it here.
For $K_{7} $, it follows from H\"{o}lder's and Young's inequalities that
\begin{align}
K_7 = \frac{1}{2} \int_{\Omega} \rho \phi_t \mu_t \, dx
\leq C_{\rho^*} \|\phi_t\|_2 \sqrt{\rho t} \mu_t\|_2\leq C_{\rho^*} \|\phi_t\|_2^2 + \varepsilon\sqrt{\rho t} \mu_t\|_2^2 .
\end{align}
For $K_8$, thanks to \( \rho_t = - \mathbf{u} \cdot \nabla \rho \), we deduce that
\begin{align} \notag
K_8 &= \int_{\Omega} \sqrt{t}\text{div}(\rho \mathbf{u}) \cdot \mathbf{u} \cdot\nabla \phi\sqrt{t} \mu_t \, dx = -\int_{\Omega} \rho \mathbf{u} \cdot \nabla (t\mathbf{u} \cdot\nabla \phi \mu_t) \, dx \\
&= -\int_{\Omega}t \rho \mathbf{u} \cdot \nabla \mathbf{u}\cdot \nabla \phi\mu_t \, dx - \int_{\Omega} t\rho |\mathbf{u}|^2 \cdot\nabla ^2\phi \mu_t \, dx - \int_{\Omega}t \rho |\mathbf{u}|^2\cdot \nabla \phi \nabla \mu_t \\&\triangleq \sum_{i=1}^{3} K_{8i}\notag.
\end{align}
We bound term by term above in what follows.  According to H\"{o}lder's and Young's inequalities, we get
\begin{align*}
K_{81} &\leq C_{\rho^* T} \|\sqrt{\rho t} \mu_t\|_2 \|\nabla \mathbf{u}\|_2 \|\nabla \phi\|_\infty \|\mathbf{u}\|_{\infty} \leq \varepsilon \|\sqrt{\rho t} \mu_t\|_2^2 + C_{\rho^*T} \|\nabla \mathbf{u}\|_2^2 \|\nabla \phi\|_\infty^2 \|\mathbf{u}\|_{\infty}^2, \\
K_{82} &\leq C_{\rho^* T} \|\mathbf{u}\|_{\infty}^2 \|\sqrt{\rho t} \mu_t\|_2 \|\nabla^2 \phi\|_2 \leq  \varepsilon\|\sqrt{\rho t} \mu_t\|_2^2 + C_{\rho^* T}\|\mathbf{u}\|_{\infty}^4 \|\nabla ^2\phi\|_2^2,\\
K_{83} &\leq C_{\rho^* T} \|\nabla\sqrt{t} \mu_t\|_2 \|\nabla \phi\|_2 \|\mathbf{u}\|_{\infty}^2 \leq \varepsilon \|\nabla\sqrt{t} \mu_t\|_2^2 + C_{\rho^* T} \|\mathbf{u}\|_{\infty}^4 \|\nabla \phi\|_2^2.
\end{align*}
It follows from  H\"{o}lder's and Young's inequalities, that
\begin{align}
K_9& \leq C_{\rho^* T} \|\sqrt{\rho t} \mathbf{u}_t\|_2 \|\sqrt{\rho t} \mu_t\|_2 \|\nabla \phi\|_\infty \leq C_{\rho^* T} \|\nabla \phi\|_\infty^2 \|\sqrt{\rho t} \mathbf{u}_t\|_2^2 + \varepsilon \|\sqrt{\rho t} \mu_t\|_2^2,\\
K_{10} &\leq C_{\rho^*} \|\nabla \sqrt{t}\phi_t\|_2 \|\sqrt{\rho t} \mu_t\|_2 \|\mathbf{u}\|_{\infty} \leq C_{\rho^*}\|\mathbf{u}\|_{\infty}^2 \|\nabla\sqrt{t}\phi_t\|_2^2 + \varepsilon \|\sqrt{\rho t} \mu_t\|_2^2.
\end{align}
For $L_{2}$ and $L_{3}$, by \( \rho_t = - \mathbf{u} \cdot \nabla \rho \), we have
\begin{equation}
\begin{split}\label{3.2666}
L_2 &= \frac{1}{2} \int_{\Omega} t \, \text{div}(\rho \mathbf{u}) |\phi_t|^2 dx = -\frac{1}{2} \int_{\Omega} \rho \mathbf{u} \cdot \nabla t |\phi_t|^2 dx = -\frac{1}{2} \int_{\Omega} t \rho \mathbf{u} \cdot 2\phi_t \nabla \phi_t \\
& \leq C_{\rho^* T} \|\nabla\sqrt{t} \phi_t\|_2 \|\phi_t\|_2 \|\mathbf{u}\|_{\infty} \leq C_{\rho^* T} \|\mathbf{u}\|_{\infty}^2 \|\nabla\sqrt{t} \phi_t\|_2^2 + C \|\phi_t\|_2^2,
\end{split}
\end{equation}
\begin{equation}
\begin{split}\label{3.2777}
L_3 &= \int_{\Omega} t \, \text{div}(\rho \mathbf{u}) \cdot \mathbf{u} \cdot\nabla \phi \phi_t dx = -\int_{\Omega} t \rho \mathbf{u} \cdot \nabla (\mathbf{u}\cdot \nabla \phi \phi_t) dx = -\int_{\Omega} t \rho \mathbf{u} \cdot \nabla \mathbf{u} \cdot\nabla \phi \phi_t dx \\
&\quad-\int_{\Omega} t \rho |\mathbf{u}|^2\cdot \nabla^2 \phi \phi_t dx - \int_{\Omega} t \rho |\mathbf{u}|^2 \cdot \nabla \phi\cdot \nabla \phi_t dx  \\&\triangleq \sum_{i=1}^{3} L_{3i}.
\end{split}
\end{equation}
According to H\"{o}lder's and Young's inequalities, we get
\begin{align*}
L_{31}& \leq C_{\rho^* T} \|\mathbf{u}\|_{\infty} \|\nabla \mathbf{u}\|_2 \|\nabla \phi\|_4 \| \phi_t\|_4 \leq  C_{\rho^* T}\|\mathbf{u}\|_{\infty}^2 \| \nabla\phi\|_4^2 + C \|\nabla \mathbf{u}\|_2^2 \| \phi_t\|_4^2, \\
L_{32} &\leq C_{\rho^* T} \|\nabla^2 \phi\|_2 \|\mathbf{u}\|_{\infty}^2 \| \phi_t\|_2 \leq C_{\rho^* T} \|\nabla^2 \phi\|_2^2 \|\mathbf{u}\|_{\infty}^4 +C \|\phi_t\|_2^2,\\
L_{33} &\leq C_{\rho^* T} \|\nabla \phi\|_2 \|\nabla\sqrt{t} \phi_t\|_2 \|\mathbf{u}\|_{\infty}^2 \leq C_{\rho^* T}\|\mathbf{u}\|_{\infty}^4 \|\nabla\sqrt{t} \phi_t\|_2^2 + C \|\nabla \phi\|_2^2.
\end{align*}
For $L_4$ and  $L_5$, it follows from H\"{o}lder's and Young's inequalities that
\begin{equation}
\begin{split}\label{3.2888}
L_4 &= -\int_{\Omega} t \rho \mathbf{u}_t\cdot \nabla \phi \phi_t dx
\leq C_{\rho^* T} \|\phi_t\|_4 \|\sqrt{\rho t} \mathbf{u}_t\|_2 \|\nabla \phi\|_4 \\
&\leq C_{\rho^* T} \|\sqrt{\rho t} \mathbf{u}_t\|_2^2 + C \|\nabla \phi\|_4^2 \|\phi_t\|_4^2,
\end{split}
\end{equation}
\begin{equation}
\begin{split}\label{3.2999}
L_5& = -\int_{\Omega} t \rho \mathbf{u}\cdot \nabla \phi_t \phi_t dx \leq C_{\rho^* T} \|\mathbf{u}\|_{\infty} \|\nabla\sqrt{t} \phi_t\|_2 \|\phi_t\|_2 \leq C_{\rho^* T}\|\mathbf{u}\|_{\infty}^2 \|\nabla\sqrt{t} \phi_t\|_2^2 + C\|\phi_t\|_2^2.
\end{split}
\end{equation}
For $L_{6}$, by \( \rho_t = - \mathbf{u} \cdot \nabla \rho \),  H\"{o}lder's and Young's inequalities, we have
\begin{equation}
\begin{split}\label{3.3000}
L_6 &= \int_{\Omega} t \, \text{div}(\rho \mathbf{u}) \mu \mu_t dx = -\int_{\Omega} \rho \mathbf{u} \cdot \nabla (t \mu \mu_t) dx \\
\quad &= -\int_{\Omega} t\rho \mathbf{u} \cdot \nabla \mu \mu_t dx - \int_{\Omega} t \rho \mathbf{u} \mu \cdot\nabla \mu_t dx
\\&\leq C_{\rho^* T} \|\sqrt{\rho t} \mu_t\|_2 \|\nabla \mu\|_2 \|\mathbf{u}\|_{\infty} +C_{\rho^* T} \|\nabla\sqrt{t} \mu_t\|_2 \|\mathbf{u}\|_4 \|\mu\|_4
\\&\leq \varepsilon \|\sqrt{\rho t} \mu_t\|_2^2 + C_{\rho^* T} \|\mathbf{u}\|_{\infty}^2 \|\nabla \mu\|_2^2+\varepsilon \|\nabla\sqrt{t} \mu_t\|_2^2 + C_{\rho^* T} \|\mathbf{u}\|_4^2 \|\mu\|_4^2.
\end{split}
\end{equation}
For $L_{7}$, by \( \rho_t = - \mathbf{u} \cdot \nabla \rho \), we have
\begin{align}
L_7 &= -\int_{\Omega} t \, \text{div}(\rho \mathbf{u}) (\phi^3 - \phi) \mu_t dx = \int_{\Omega} \rho \mathbf{u} \cdot\nabla \bigl( t (\phi^3 - \phi) \mu_t \bigr) dx \\ \notag
\quad& = \int_{\Omega} 3t\rho \mathbf{u} \phi^2 \cdot\nabla \phi  \mu_t dx - \int_{\Omega} t\rho\mathbf{u} \cdot \nabla \phi \mu_t dx + \int_{\Omega}t \rho \mathbf{u}  (\phi^3 - \phi)\cdot \nabla \mu_t dx\\&\triangleq \sum_{i=1}^{3} L_{7i}\notag.
\end{align}
It follows from H\"{o}lder's and Young's inequalities that
\begin{align*}
L_{71} &\leq C_{\rho^* T} \|\mathbf{u}\|_{\infty} \|\phi^2\|_4 \|\nabla \phi\|_4 \|\sqrt{\rho t} \mu_t\|_2\leq \varepsilon \|\sqrt{\rho t} \mu_t\|_2^2 + C_{\rho^* T}\|\mathbf{u}\|_{\infty}^2 \|\phi\|_8^4 \|\nabla \phi\|_4^2, \\
L_{72} &\leq C_{\rho^* T} \|\sqrt{\rho t} \mu_t\|_2 \|\nabla \phi\|_4 \|\mathbf{u}\|_4 \leq \varepsilon \|\sqrt{\rho t} \mu_t\|_2^2 + C_{\rho^* T} \|\nabla \phi\|_4^2 \|\mathbf{u}\|_4^2, \\
L_{73} &\leq C_{\rho^* T} \|\nabla\sqrt{t} \mu_t\|_2 \left( \|\phi^3\|_4 + \|\phi\|_4 \right) \|\mathbf{u}\|_4 \leq \varepsilon \|\nabla\sqrt{t} \mu_t\|_2^2 + C_{\rho^* T} \left( \|\phi\|_{12}^3 + \|\phi\|_4 \right)^2 \|\mathbf{u}\|_4^2.
\end{align*}
For $L_8$, by H\"{o}lder's and Young's inequalities, we have
\begin{align}
L_8 &\leq C_{\rho^* T} \|\sqrt{\rho t} \mu_t\|_2 \|\phi_t\|_2 \|\phi\|_{\infty}^2 + C_{\rho^* T} \|\sqrt{\rho t} \mu_t\|_2 \|\phi_t\|_2
\\ \notag
&\leq \varepsilon \|\sqrt{\rho_t} \mu_t\|_2^2 + C_{\rho^* T}\left( \|\phi\|_{\infty}^4 + 1 \right) \|\phi_t\|_2^2.
\end{align}
Summing  \eqref{eq-31}, \eqref{eq-32}, and \eqref{eq-33} together, and  combining with  estimates of $\sum\limits_{i=1}^{12} I_i$, $\sum\limits_{i=1}^{10} K_i$, and $\sum\limits_{i=1}^{8} L_i$, we finally conclude that
\begin{equation}\label{3.30}
\begin{split}
\frac{d}{dt}& \left( \|\sqrt{\rho_t}\, \mathbf{u}_t\|_2^2 + \|\nabla\sqrt{t}\ \phi_t\|_2^2 + \|\sqrt{\rho_t}\, \phi\, \phi_t\|_2^2+ \|\sqrt{\rho t}\phi_t\|_2^2\right)\\
&\quad+ \int_0^t \rho\tau \|\mu_t\|_2^2 \, dt
+ \int_0^t \nu(\phi)\, \tau \|\nabla \mathbf{u}_t\|_2^2 \, dt
+ \int_0^t \|\nabla\sqrt{\tau}\, \mu_t\|_2^2 \, dt \\
&\leq h_1(t) \left( 1 + \|\sqrt{\rho_t}\, \mathbf{u}_t\|_2^2 + \|\nabla\sqrt{t}\, \phi_t\|_2^2 + \|\sqrt{\rho_t}\, \phi\, \phi_t\|_2^2+\|\sqrt{\rho t}\phi_t\|_2^2 \right),
\end{split}
\end{equation}
where
\begin{align*}
&h_1(t)=
(1 + C + C_{\rho^*} + C_{\rho^*\rho_*T} + C_T)
\Big(\big( \|\phi\|_{12}^3 + \|\phi\|_4 \big)^2
\big( \|\nabla^2 \phi\|_4^2 + \|\nabla \phi\|_4^2 + \|\mathbf{u}\|_4^2 \big) \\
&+ \big( \|\nabla \phi\|_6 + \|\phi\|_\infty^2 \|\nabla\phi\|_6 \big)^2
\big( \|\mathbf{u}\|_6^2 + \|\nabla\phi\|_4^2  + \|\mathbf{u}\|_4^2 \big) + \|\mathbf{u}\|_4^2 \big( \|\nabla^2 \phi\|_6 + \|\nabla \phi\|_{12}^2 \|\phi\|_\infty + \|\nabla^2 \phi\|_6 \|\phi\|_\infty^2 \big)^2 \\
&+ \|\nabla \mathbf{u}\|_2^2 + \| \mu\|_4^2 \|\nabla \phi\|_4^2 + \|\nabla \phi\|_2^2 + \|\phi\|_\infty^4 + \|\nabla \mu\|_2^2 + \|\nabla \phi\|_\infty^2
 + \|\nabla^2 \phi\|_2^2 + \|\nabla \phi\|_4^2 + \|\phi\|_8^4 \|\nabla\phi\|_4^2 \\
&+ \|\nabla \phi\|_4^8 + \big( \|\phi\|_\infty^2 + 1 \big)^2 \|\nabla \phi\|_4^2 + \|\mathbf{u}\|_4^4 + \|\mathbf{u}\|_6^4 + \|\phi\|_6^4 + \|\nabla \mathbf{u}\|_2^2  + \|\mathbf{u}\|_4 \|\nabla \phi\|_4 \|\phi\|_6 + \|\nabla \mu\|_2^2 \\&+ \|\nabla \phi\|_4^2
 + \|\phi\|_\infty^4
+ \|\nabla \phi\|_4^2 + \|\mathbf{u}\|_6^2 + \|\mathbf{u}\|_2^2 + \|\nabla \mathbf{u}\|_2^2 + \|\nabla \mathbf{u}\|_2^{\frac{7}{2}} \\
&+ \|\mathbf{u}\|_4^2 \big( \|\nabla \phi\|_4^2 + \|\nabla^2 \phi\|_4^2 + \|\mu\|_4^2 \big)
+ \|\nabla^2 \phi\|_4^2 \|\mu\|_4^2 + \|\nabla \phi\|_8^2 \|\phi\|_8^2
+ \big( \|\phi\|_\infty^3 + \|\phi\|_\infty\big)^2 \Bigr)\\
&\times\Big( \|\mathbf{u}\|_{\infty}^2 + \|\mathbf{u}\|_{\infty}^4 + \|\mu\|_{\infty}^2 + \|\phi_t\|_2^2 + \| \phi_t\|_4^2 + \| \phi_t\|_6^2 + \|\phi_t\|_2 + \|\mathbf{u}_t\|_2^2 + \|\nabla \phi_t\|_2^2 \\
&+ \|\nabla^2 \mathbf{u}\|_2^2 + \|\nabla \mathbf{u}\|_3^2 + \|\nabla \mu\|_4^2 + \|\nabla \mu\|_6^2 + \|\nabla^2 \mu\|_6^2 + \|\nabla^2 \mu\|_4^2 + \|\nabla \mu\|_{w^{1,6}}^2 \Big)
\end{align*}
only depends on, \( \rho^* \), \( \rho_* \), $T$ and the initial value.
From  Theorem \ref{1.1}, we have 
\begin{align*}
 \mathbf{u}\in& \mathcal{C}([0,T_0];\mathbf{V}_{\sigma})\cap L^2(0,T_0;H^2(\Omega))\cap H^1(0,T_0;\mathbf{H}_{\sigma}),\\
 \phi\in &\mathcal{C}([0,T_0];(W^{2,6}(\Omega))_w)\cap H^1(0,T_0;H^1(\Omega)),\\
 \mu\in &L^{\infty}(0,T_0;H^1(\Omega))\cap L^2(0,T_0;W^{2,6}(\Omega)),
\end{align*}
and the  3-D Gagliardo-Nirenberg interpolation inequality: \( \| u \|_{\infty}^4 \leq \|\nabla u \|_2^2 \| \nabla^2 u \|_2^2 \),
we get
\begin{align*}
&\int_{0}^{t} \| \mathbf{u} \|_{\infty}^{4} dt \leq \int_{0}^{t} \| \nabla \mathbf{u} \|_{2}^{2} \| \nabla ^2 \mathbf{u} \|_{2}^{2} dt \leq C\int_{0}^{t} \|\nabla ^2 \mathbf{u} \|_{2}^{2} dt \leq C_{T},\\
&\int_{0}^{t} \| \mathbf{u} \|_{\infty}^{2} dt \leq \int_{0}^{t} \| \nabla \mathbf{u} \|_{2} \| \nabla ^2 u \|_{2} dt \leq C\int_{0}^{t} \|  \nabla ^2 u \|_{2} dt \leq \int_{0}^{t} \|  \nabla ^2 \mathbf{u} \|_{2}^{2} dt\cdot t^{\frac{1}{2}} \leq C_{T},\\
&\int_{0}^{t} \| \mu \|_{\infty}^{2} dt \leq \int_{0}^{t} \| \nabla \mu \|_{2} \| \nabla ^2 \mu \|_{2} dt \leq \int_{0}^{t} \|  \nabla ^2 \mu \|_{2} dt \leq \int_{0}^{t} \| \nabla ^2 \mu\|_{2}^{2} dt \cdot t^{\frac{1}{2}} \leq C_{T}.
\end{align*}
Then we conclude that $ h_1(t)\in L^1_{loc}(\mathbb{R}^+) $. This along with \eqref{3.30} and Gronwall's inequality yields that \eqref{3D}  holds for \( t \geq 0 \).
\end{proof}
\begin{lemma}
Assume \(d = 2\), and that that a strong solution $(\rho,\mathbf{u},P,\phi,\mu)$ to the initial-boundary value problem   \eqref{NSCH}-\eqref{NSCH1}. Then for all \(0\leq t\leq T\), we have
\begin{equation}\label{2D}
\begin{split}
\|\sqrt{\rho t}\mathbf{u}_t\|_2^2&+\|\sqrt{\rho t}\phi_t\phi\|_2^2+\|\sqrt{ t}\nabla\phi_t\|_2^2+\|\sqrt{\rho t}\phi_t\|_2^2+\int_{0}^{t}\nu(\phi)\tau\|\nabla \mathbf{u}_t\|_2^2d\tau \\
&+\int_{0}^{t}\tau\|\nabla\mu_t\|_2^2d\tau+\int_{0}^{t}\tau\rho\|\mu_t\|_2^2d\tau
\leq\exp\Big(\int_{0}^{t}h_2(\tau)d\tau\Big)-1,
\end{split}
\end{equation}
where
 \begin{align*}
&h_2(t)=(1 + C + C_{\rho^*} + C_{\rho^*\rho_*T} + C_T)
\Big(\big( \|\phi\|_{12}^3 + \|\phi\|_4 \big)^2
\big( \|\nabla^2 \phi\|_4^2 + \|\nabla \phi\|_4^2 + \|\mathbf{u}\|_4^2 \big) \\
&+ \big( \|\nabla \phi\|_6 + \|\phi\|_\infty^2 \|\nabla\phi\|_6 \big)^2
\big( \|\mathbf{u}\|_6^2 + \|\nabla\phi\|_4^2  + \|\mathbf{u}\|_4^2 \big) + \|\mathbf{u}\|_4^2 \big( \|\nabla^2 \phi\|_6 + \|\nabla \phi\|_{12}^2 \|\phi\|_\infty + \|\nabla^2 \phi\|_6 \|\phi\|_\infty^2 \big)^2 \\
&+ \|\nabla \mathbf{u}\|_2^2 + \| \mu\|_4^2 \|\nabla \phi\|_4^2 + \|\nabla \phi\|_2^2 + \|\phi\|_\infty^4 + \|\nabla \mu\|_2^2 + \|\nabla \phi\|_\infty^2
 + \|\nabla^2 \phi\|_2^2 + \|\nabla \phi\|_4^2 + \|\phi\|_8^4 \|\nabla\phi\|_4^2 \\
&+ \|\nabla \phi\|_4^8 + \big( \|\phi\|_\infty^2 + 1 \big)^2 \|\nabla \phi\|_4^2 + \|\mathbf{u}\|_4^4 + \|\mathbf{u}\|_6^4 + \|\phi\|_6^4 + \|\nabla \mathbf{u}\|_2^2  + \|\mathbf{u}\|_4 \|\nabla \phi\|_4 \|\phi\|_6 + \|\nabla \mu\|_2^2 + \|\nabla \phi\|_4^2 \\
&+ \|\phi\|_\infty^4 + \|\nabla \phi\|_4^2 + \|\mathbf{u}\|_6^2 + \|\mathbf{u}\|_2^2 + \|\nabla \mathbf{u}\|_2^2
+ \|\mathbf{u}\|_4^2 \big( \|\nabla \phi\|_4^2 + \|\nabla^2 \phi\|_4^2 + \|\mu\|_4^2 \big)
\\&+ \|\nabla^2 \phi\|_4^2 \|\mu\|_4^2 + \|\nabla \phi\|_8^2 \|\phi\|_8^2
+ \big( \|\phi\|_\infty^3 + \|\phi\|_\infty\big)^2 \Big)\times\Big( \|\mathbf{u}\|_{\infty}^2 + \|\mathbf{u}\|_{\infty}^4 + \|\mu\|_{\infty}^2 + \|\phi_t\|_2^2 \\
&+ \| \phi_t\|_4^2 + \| \phi_t\|_6^2 + \|\phi_t\|_2 + \|\mathbf{u}_t\|_2^2 + \|\nabla \phi_t\|_2^2 +  \|\nabla^2 \mathbf{u}\|_2^2 + \|\nabla \mathbf{u}\|_3^2 + \|\nabla \mu\|_4^2 + \|\nabla \mu\|_6^2\\
& + \|\nabla^2 \mu\|_6^2 + \|\nabla^2 \mu\|_4^2 + \|\nabla \mu\|_{w^{1,6}}^2 \Big)
\end{align*}
is $ h_2(t) \in L^1_{loc}(\mathbb{R}^+)$  only depends on, \( \rho^* \), \( \rho_* \), $T$ and the initial value.
\end{lemma}
\begin{proof}
Compared with the proof of the 3-D case, we here only show some different parts for \(I_{31}$, $I_4$, $I_6$, $I_8 \), \( K_{43} \) and \( K_{9} \) in what follows. Combining H\"{o}lder's inequality and Sobolev embedding, we have
\begin{align*} \notag
I_{31} &= \int_{\mathbb{T}^2} \sqrt{\rho t} |\mathbf{u}| \|\nabla \mathbf{u}\| \|\nabla \mathbf{u}\| \sqrt{\rho t} \mathbf{u}_t dx \leq \|\mathbf{u}\|_{\infty}^2 \| \sqrt{\rho t} \mathbf{u}_t \|_2^2 + C _{T \rho^*} \| \nabla \mathbf{u} \|_4^4,\\
&\leq \|\mathbf{u}\|_{\infty}^2 \| \sqrt{\rho t} \mathbf{u}_t \|_2^2 + C_{ T \rho^*} \| \nabla \mathbf{u} \|_2^2\| \mathbf{u} \|_2^2,\\ \notag
I_4 &\leq C_{\rho*T}\| \nabla \mathbf{u} \|_2 \| \sqrt{t} \mathbf{u}_t \|_4^2
\leq C_{\rho*\rho_*T} \| \nabla \mathbf{u} \|_2 \| \sqrt{\rho t} \mathbf{u}_t \|_2 \| \sqrt{t} \nabla \mathbf{u}_t \|_2 \\ &\leq \varepsilon \| \sqrt{t} \nabla \mathbf{u}_t \|_2^2 + C_{\rho*\rho_*T} \| \nabla \mathbf{u} \|_2^2 \| \sqrt{\rho t} \mathbf{u}_t \|_2^2,\\ \notag
I_6 &\leq \left|- \int _{\Omega}\nu'(\phi) \phi_t D \mathbf{u} \cdot \nabla t\mathbf{u}_t dx \right| \leq C \| \sqrt{t} \nabla \mathbf{u}_t \|_2 \|\sqrt t \phi_t \|_4 \| D \mathbf{u} \|_4\\
&\leq \varepsilon \| \sqrt{t} \nabla \mathbf{u}_t \|_2^2 + C_T\| \nabla \mathbf{u} \|_2\| \nabla^2 \mathbf{u} \|_2 \|\nabla\sqrt t\phi_t \|_2\|\phi_t \|_2\\ \notag
&\leq \varepsilon \| \sqrt{t} \nabla \mathbf{u}_t \|_2^2 + C_T\| \nabla \mathbf{u} \|_2^2\| \nabla^2 \mathbf{u} \|_2^2+C \|\nabla\sqrt t\phi_t \|_2^2\|\phi_t \|_2^2\\
I_8 &\leq \left| \int_{\Omega} ( \|\sqrt{t}\rho \, \mu_t \, \nabla \phi) \, (\sqrt{t} \, \mathbf{u}_t) \, dx \right| \leq C_{\rho*} \, \|\sqrt{\rho t} \, \mu_t\|_2 \, \|\nabla \, \phi\|_4 \, \|\sqrt{t}\mathbf{u}_t\|_4\\ \notag
&\leq C_{\rho*T} \, \|\sqrt{\rho t} \, \mu_t\|_2 \, \|\nabla \, \phi\|_4 \, \|\nabla\sqrt{t}u_t\|_2^\frac{1}{2}\|\mathbf{u}_t\|_2^\frac{1}{2} \leq \varepsilon\|\nabla \sqrt{t} \, \mathbf{u}_t\|_2^2 + C_{\rho*T} \, \|\nabla \phi\|_4^2 \, \|\mathbf{u}_t\|_2^2+\varepsilon\|\sqrt{\rho t} \, \mu_t\|_2^2,
\end{align*}
\begin{align*}
K_{43}&=3\int_{\Omega}|\sqrt{t}\phi_t|^2\rho\phi\phi_t dx
\leq C_{\rho*T}\|\sqrt{\rho t}\phi_t\phi\|_2\|\phi_t\|_4\|\phi_t\|_4\\ \notag
&\leq C_{\rho*T}\|\phi_t\|_4^2\|\sqrt{\rho t}\phi_t\phi\|_2^2+ C\|\phi_t\|_4^2,\\
K_9& \leq C_{\rho^* T} \|\sqrt{\rho_t} \mathbf{u}_t\|_2 \|\sqrt{t} \mu_t\|_4 \|\nabla \phi\|_4 \leq C_{\rho^* T} \|\sqrt{\rho_t} \mathbf{u}_t\|_2 \|\sqrt{\rho t} \mu_t\|_2^\frac{1}{2} \|\nabla\sqrt{ t} \mu_t\|_2^\frac{1}{2} \|\nabla \phi\|_4\\ \notag
&\leq C_{\rho^* T} \|\nabla \phi\|_4^2 \|\sqrt{\rho_t} u_t\|_2^2 + \varepsilon \|\nabla\sqrt{t} \mu_t\|_2^2+\varepsilon \|\sqrt{\rho t} \mu_t\|_2^2.
\end{align*}
Then summing \eqref{eq-31}, \eqref{eq-32}, and \eqref{eq-33} together, and combining with   estimates of $\sum\limits_{i=1}^{12} I_i$, $\sum\limits_{i=1}^{10} K_i$, and $\sum\limits_{i=1}^{8} L_i$, we deduce that
\begin{equation}\label{3.38}
\begin{split}
\frac{d}{dt} &\left( \|\sqrt{\rho_t}\, \mathbf{u}_t\|_2^2 + \|\nabla\sqrt{t}\ \phi_t\|_2^2 + \|\sqrt{\rho_t}\, \phi\, \phi_t\|_2^2+\|\sqrt{\rho t}\phi_t\|_2^2 \right)
+ \int_0^t \rho\tau \|\mu_t\|_2^2 \, dt
\\&+ \int_0^t \nu(\phi)\, \tau \|\nabla \mathbf{u}_t\|_2^2 \, dt
+ \int_0^t \|\nabla\sqrt{\tau}\, \mu_t\|_2^2 \, dt\\
&\leq h_2(t) \left( 1 + \|\sqrt{\rho_t}\, \mathbf{u}_t\|_2^2 + \|\nabla\sqrt{t}\, \phi_t\|_2^2 + \|\sqrt{\rho_t}\, \phi\, \phi_t\|_2^2+\|\sqrt{\rho t}\phi_t\|_2^2 \right),
\end{split}
\end{equation}
where
\begin{align*}
&h_2(t)=(1 + C + C_{\rho^*} + C_{\rho^*\rho_*T} + C_T)
\Big(\big( \|\phi\|_{12}^3 + \|\phi\|_4 \big)^2
\big( \|\nabla^2 \phi\|_4^2 + \|\nabla \phi\|_4^2 + \|\mathbf{u}\|_4^2 \big) \\
&+ \big( \|\nabla \phi\|_6 + \|\phi\|_\infty^2 \|\nabla\phi\|_6 \big)^2
\big( \|\mathbf{u}\|_6^2 + \|\nabla\phi\|_4^2  + \|\mathbf{u}\|_4^2 \big) + \|\mathbf{u}\|_4^2 \big( \|\nabla^2 \phi\|_6 + \|\nabla \phi\|_{12}^2 \|\phi\|_\infty + \|\nabla^2 \phi\|_6 \|\phi\|_\infty^2 \big)^2 \\
&+ \|\nabla \mathbf{u}\|_2^2 + \| \mu\|_4^2 \|\nabla \phi\|_4^2 + \|\nabla \phi\|_2^2 + \|\phi\|_\infty^4 + \|\nabla \mu\|_2^2 + \|\nabla \phi\|_\infty^2
 + \|\nabla^2 \phi\|_2^2 + \|\nabla \phi\|_4^2 + \|\phi\|_8^4 \|\nabla\phi\|_4^2 \\
&+ \|\nabla \phi\|_4^8 + \big( \|\phi\|_\infty^2 + 1 \big)^2 \|\nabla \phi\|_4^2 + \|\mathbf{u}\|_4^4 + \|\mathbf{u}\|_6^4 + \|\phi\|_6^4 + \|\nabla \mathbf{u}\|_2^2  + \|\mathbf{u}\|_4 \|\nabla \phi\|_4 \|\phi\|_6 + \|\nabla \mu\|_2^2 + \|\nabla \phi\|_4^2 \\
&+ \|\phi\|_\infty^4 + \|\nabla \phi\|_4^2 + \|\mathbf{u}\|_6^2 + \|\mathbf{u}\|_2^2 + \|\nabla \mathbf{u}\|_2^2
+ \|\mathbf{u}\|_4^2 \big( \|\nabla \phi\|_4^2 + \|\nabla^2 \phi\|_4^2 + \|\mu\|_4^2 \big)
\\&+ \|\nabla^2 \phi\|_4^2 \|\mu\|_4^2 + \|\nabla \phi\|_8^2 \|\phi\|_8^2
+ \big( \|\phi\|_\infty^3 + \|\phi\|_\infty\big)^2 \Big)\times\Big( \|\mathbf{u}\|_{\infty}^2 + \|\mathbf{u}\|_{\infty}^4 + \|\mu\|_{\infty}^2 + \|\phi_t\|_2^2 \\
&+ \| \phi_t\|_4^2 + \| \phi_t\|_6^2 + \|\phi_t\|_2 + \|\mathbf{u}_t\|_2^2 + \|\nabla \phi_t\|_2^2 +  \|\nabla^2 \mathbf{u}\|_2^2 + \|\nabla \mathbf{u}\|_3^2 + \|\nabla \mu\|_4^2 + \|\nabla \mu\|_6^2\\
& + \|\nabla^2 \mu\|_6^2 + \|\nabla^2 \mu\|_4^2 + \|\nabla \mu\|_{w^{1,6}}^2 \Big)
\end{align*}
only depends on, \( \rho^* \), $T$ and the initial value.
Combining with  Theorem \ref{1.1}, and the 2-D Ladyzhenskaya inequality: \( \| \mathbf{u} \|_{\infty}^4 \leq \| \mathbf{u} \|_2^2 \| \nabla^2 \mathbf{u} \|_2^2 \) yields that
$h_2(t) \in L^1_{loc}(\mathbb{R}^+)$. Then applying Gronwall's inequality to \eqref{3.38} implies  that \eqref{2D} holds for \( t \geq 0 \).
\end{proof}
\section{The estimate  for quantity $\int_{0}^{T}\|(\nabla \mathbf{u}, \nabla \mu, \nabla\phi)(\tau)\|_{L^{\infty}}d\tau$} In  this section, our main  goal is to achieve the bound  estimate  for quantity $\int_{0}^{T}\|(\nabla \mathbf{u}, \nabla \mu, \nabla\phi)(\tau)\|_{L^{\infty}}d\tau$ in terms of the initial data and of $T$ by performing the shift of integrability method. This is given by the following two lemmas in two and three dimensions, respectively.
\begin{lemma}\label{3DDD}
Assume \( d = 3 \), then for all \( T > 0 \), \( p \in [2,\infty] \), we have
\begin{equation}\label{3DD}
\left\lVert (\nabla^2 \sqrt{t} \mathbf{u}, \nabla^2 \sqrt{t}\phi, \nabla^2 \sqrt{t} \mu) \right\rVert_{L^p(0,T;L^r)} + \left\lVert \nabla \sqrt{t} P \right\rVert_{L^p(0,T;L^r)} \leq C_{0,T}, \quad \text{for} \quad 2 \leq r \leq \frac{6p}{3p - 4},
\end{equation}
where \( C_{0,T} \) depends only on \( \rho^* \), \( p \) and the initial value. Furthermore, for \( s \in \left[1,\frac{4}{3}\right) \), then for some \( \theta > 0 \), we have
\begin{equation}
\int_0^T \left( \left\lVert \nabla \mathbf{u} \right\rVert_\infty^s + \left\lVert \nabla \mu\right\rVert_\infty^s + \left\lVert \nabla\phi \right\rVert_\infty^s \right) dt \leq C_{0,T} T^\theta.\label{eq-s}
\end{equation}
\end{lemma}
\begin{proof}
From \eqref{NSCH}, we have
\begin{align}
\begin{cases}
&-\mathrm{div}\left( \sqrt{t}\, v(\phi)\, D\mathbf{u} \right) + \nabla \sqrt{t}\, p = -\mathrm{div}\sqrt{t}\, (\nabla \phi \otimes \nabla \phi) - \rho \sqrt{t}\, (\mathbf{u}_t + \mathbf{u} \cdot \nabla \mathbf{u}),\label{eq-n} \\
&\mathrm{div}\sqrt{t}\, \mathbf{u} = 0,\\
&\Delta \sqrt{t}\mu = \rho \sqrt{t}\, \phi_t + \rho \sqrt{t}\, \mathbf{u} \nabla \phi, \\
&-\Delta \sqrt{t}\, \phi + \rho \sqrt{t}\, \Phi'(\phi) = \rho \sqrt{t}\, \mu.
\end{cases}
\end{align}
Obviously,
\begin{equation*}
\begin{split}
\Phi'(\phi)&=\phi^3-\phi,\\
-\operatorname{div}(\nabla \phi \otimes \nabla \phi) &= -\Delta \phi \nabla \phi - \nabla \Big( \frac{1}{2} |\nabla \phi|^2 \Big) \\
&= \rho \mu \nabla \phi - \rho \Phi'(\phi) \nabla \phi - \nabla \Big( \frac{1}{2} |\nabla \phi|^2 \Big) \\
&= \rho \mu \nabla \phi - \rho \nabla \Phi(\phi) - \nabla \Big( \frac{1}{2} |\nabla \phi|^2 \Big)\\
&= \rho \mu \nabla \phi - \rho (\phi^3-\phi)\nabla\phi - \nabla\phi \nabla^2\phi.
\end{split}
\end{equation*}
From Theorem \ref{1.1}, \eqref{NSCH2} and \eqref{3D}, we get
\[
\rho\sqrt{t}\mathbf{u}_t, \rho\sqrt{t}\phi_t \in L^\infty(0,T;L^2) \cap L^2(0,T;H^1),
\]
which together with $\dot{H}^1(\Omega) \hookrightarrow L^6(\Omega)$, yields that
\[
\rho\sqrt{t}\mathbf{u}_t, \rho\sqrt{t}\phi_t \in L^\infty(0,T;L^2) \cap L^2(0,T;L^q) \quad \text{with} \quad q \leq 6.
\]
This along with the interpolation inequality  gives rise to
\begin{equation*}
\|\rho\sqrt{t}\mathbf{u}_t\|_{L^p(0,T;L^r)} \leq \|\rho\sqrt{t}\mathbf{u}_t\|_{L^\infty(0,T;L^2)}^{1 - \frac{2}{p}} \|\rho\sqrt{t}\mathbf{u}_t\|_{L^2(0,T;L^q)}^{\frac{2}{p}},
\end{equation*}
and
\begin{equation*}
\|\rho\sqrt{t}\phi_t\|_{L^p(0,T;L^r)} \leq \|\rho\sqrt{t}\phi_t\|_{L^\infty(0,T;L^2)}^{1 - \frac{2}{p}} \|\rho\sqrt{t}\phi_t\|_{L^2(0,T;L^q)}^{\frac{2}{p}},
\end{equation*}
with \(\frac{1}{r} = \frac{p - 2}{2p} + \frac{2}{pq}\). Here, when \(q\) takes \(6\), then \(r\) may take the maximum value of \(\frac{6p}{3p - 4}\). Then we readily get
\begin{equation}
\|(\rho\sqrt{t}\mathbf{u}_t, \rho\sqrt{t}\phi_t)\|_{L^p(0,T;L^r)} \leq C_{0,T} \quad \text{for} \quad p \in [2,\infty], \quad r \in \Big[2, \frac{6p}{3p - 4}\Big].\label{eq-11}
\end{equation}
It follows from Theorem \ref{1.1} and \eqref{NSCH2}, that
\[
\rho\sqrt{t}\phi, \rho\sqrt{t}\mu, \rho\sqrt{t}\phi^3 \in L^\infty(0,T;L^2) \cap L^2(0,T;H^1).
\]
Then from $\dot{H}^1(\Omega) \hookrightarrow L^6(\Omega)$, we have
\[
\rho\sqrt{t}\phi, \rho\sqrt{t}\mu, \rho\sqrt{t}\phi^3 \in L^\infty(0,T;L^2) \cap L^2(0,T;L^q) \quad \text{with} \quad q \leq 6.
\]
Employing  interpolation inequality, we conclude that
\begin{align*}
\|\rho\sqrt{t}\phi\|_{L^p(0,T;L^r)} \leq \|\rho\sqrt{t}\phi\|_{L^\infty(0,T;L^2)}^{1 - \frac{2}{p}} \|\rho\sqrt{t}\phi\|_{L^2(0,T;L^q)}^{\frac{2}{p}},\\
\|\rho\sqrt{t}\mu\|_{L^p(0,T;L^r)} \leq \|\rho\sqrt{t}\mu\|_{L^\infty(0,T;L^2)}^{1 - \frac{2}{p}} \|\rho\sqrt{t}\mu\|_{L^2(0,T;L^q)}^{\frac{2}{p}},\\
\|\rho\sqrt{t}\phi^3\|_{L^p(0,T;L^r)} \leq \|\rho\sqrt{t}\phi^3\|_{L^\infty(0,T;L^2)}^{1 - \frac{2}{p}} \|\rho\sqrt{t}\phi^3\|_{L^2(0,T;L^q)}^{\frac{2}{p}}
\end{align*}
with \(\frac{1}{r} = \frac{p - 2}{2p} + \frac{2}{pq}\). Here, when \(q\) takes \(6\), then \(r\) may take the maximum value of \(\frac{6p}{3p - 4}\). Then we have
\begin{equation}
\|(\rho\sqrt{t}\phi, \rho\sqrt{t}\mu,\rho\sqrt{t}\phi^3)\|_{L^p(0,T;L^r)} \leq C_{0,T} \quad \text{for} \quad p \in [2,\infty], \quad r \in \Big[2, \frac{6p}{3p - 4}\Big].\label{eq-12}
\end{equation}
According to Theorem \ref{1.1}, we have $\nabla \mathbf{u}, \nabla \phi \in L^\infty(0,T;L^2) \cap L^2(0,T;H^1)$, and from $\dot{H}^1(\Omega) \hookrightarrow L^6(\Omega)$, we get $\nabla \mathbf{u}, \nabla \phi \in L^\infty(0,T;L^2) \cap L^2(0,T;L^q)\) with \(q \leq 6$. By interpolation inequality, we obtain
\begin{equation*}
\|\nabla \mathbf{u}\|_{L^p(0,T;L^r)} \leq \|\nabla \mathbf{u}\|_{L^\infty(0,T;L^2)}^{1 - \frac{2}{p}} \|\nabla \mathbf{u}\|_{L^2(0,T;L^q)}^{\frac{2}{p}},
\end{equation*}
\begin{equation*}
\|\nabla \phi\|_{L^p(0,T;L^r)} \leq \|\nabla \phi\|_{L^\infty(0,T;L^2)}^{1 - \frac{2}{p}} \|\nabla \phi\|_{L^2(0,T;L^q)}^{\frac{2}{p}}
\end{equation*}
with \(\frac{1}{r} = \frac{p - 2}{2p} + \frac{2}{pq}\). Then we have
\begin{equation*}
\|(\nabla \mathbf{u}, \nabla \phi)\|_{L^p(0,T;L^r)} \leq C_{0,T} \quad \text{for} \quad p \in [2,\infty], \quad r \in \Big[2, \frac{6p}{3p - 4}\Big],
\end{equation*}
which means that
\begin{equation}
\nabla \mathbf{u}, \nabla \phi \in L^4(0, T; L^3).\label{eq-15}
\end{equation}
On the other hand, using Gagliardo-Nirenberg interpolation inequality $\|v\|_{L^\infty}^4 \leq C\|\nabla v\|_{L^2}^2 \|\nabla^2 v\|_{L^2}^2$ leads to
\begin{align*}
\|\mathbf{u}\|_{L^4(0,T;L^\infty)} &\leq \|\nabla \mathbf{u}\|_{L^\infty(0,T;L^2)}^{\frac{1}{2}} \|\nabla^2 \mathbf{u}\|_{L^2(0,T;L^2)}^{\frac{1}{2}},\label{eq-16}  \\
\|\phi\|_{L^4(0,T;L^\infty)} &\leq \|\nabla \phi\|_{L^\infty(0,T;L^2)}^{\frac{1}{2}} \|\nabla^2 \phi\|_{L^2(0,T;L^2)}^{\frac{1}{2}},\\
\|\phi^3\|_{L^4(0,T;L^\infty)}&=\|\phi\|^3_{L^{12}(0,T;L^\infty)}\leq \|\nabla \phi\|_{L^\infty(0,T;L^2)}^{\frac{1}{2}} \|\nabla^2 \phi\|_{L^{12}(0,T;L^2)}^{\frac{3}{2}},\\
\|\mu\|_{L^4(0,T;L^\infty)} &\leq \|\nabla \mu\|_{L^\infty(0,T;L^2)}^{\frac{1}{2}} \|\nabla^2 \mu\|_{L^2(0,T;L^2)}^{\frac{1}{2}}.
\end{align*}
Thanks to Theorem  \ref{1.1}, we conclude that
\begin{align}
\sqrt{t}\rho \mathbf{u}, \sqrt{t}\rho \phi,\sqrt{t}\rho |\phi|^3,\sqrt{t}\rho \mu\in L^4(0,T;L^\infty).
\end{align}
Using H\"{o}lder's inequality, and then combining with \eqref{eq-15} and \eqref{eq-16}, yield that
\begin{align*}
\sqrt{t}\rho \mathbf{u} \cdot \nabla \mathbf{u}, \sqrt{t}\rho \mathbf{u}\cdot  \nabla \phi,\sqrt{t}\rho |\phi|^3\nabla \phi,\sqrt{t}\rho \mu  \nabla \phi,\sqrt{t}\rho\phi\nabla \phi \in L^2(0,T;L^3).
\end{align*}
Similarly, we also have
\begin{align*}
\sqrt{t}\rho \mathbf{u}, \sqrt{t}\rho \phi,\sqrt{t}\rho |\phi|^3,\sqrt{t}\rho\mu\in L^\infty(0,T;L^6);
\nabla \mathbf{u}, \cdot\nabla \phi \in L^\infty(0,T;L^2),
\end{align*}
which implies that
\begin{equation*}
\sqrt{t}\rho \mathbf{u}\cdot \nabla\mathbf{u}, \sqrt{t}\rho \mathbf{u} \nabla \phi ,\sqrt{t}\rho |\phi|^3\nabla \phi,\sqrt{t}\rho \mu \nabla \phi,\sqrt{t}\rho\phi\nabla \phi\in L^\infty(0, T; L^{3/2}).
\end{equation*}
It follows from the interpolation inequality and H\"{o}lder's inequality that
\begin{align*}
\|\sqrt{t}\rho \mathbf{u} \cdot \nabla \mathbf{u}\|_{L^p(0,T;L^r)} &\leq \|\sqrt{t}\rho \mathbf{u} \cdot \nabla \mathbf{u}\|_{L^2(0,T;L^3)}^{\frac{2}{p}} \|\sqrt{t}\rho \mathbf{u} \cdot \nabla \mathbf{u}\|_{L^\infty(0,T;L^{3/2})}^{1 - \frac{2}{p}}, \\
\|\sqrt{t}\rho \mathbf{u} \cdot\nabla \phi\|_{L^p(0,T;L^r)} &\leq \|\sqrt{t}\rho \mathbf{u}\cdot\nabla \phi\|_{L^2(0,T;L^3)}^{\frac{2}{p}} \|\sqrt{t}\rho \mathbf{u} \cdot\nabla \phi\|_{L^\infty(0,T;L^{3/2})}^{1 - \frac{2}{p}},\\
\|\sqrt{t}\rho \mu\nabla \phi\|_{L^p(0,T;L^r)} &\leq \|\sqrt{t}\rho \mu\nabla \phi\|_{L^2(0,T;L^3)}^{\frac{2}{p}} \|\sqrt{t}\rho \mu\nabla \phi\|_{L^\infty(0,T;L^{3/2})}^{1 - \frac{2}{p}},\\
\|\sqrt{t}\rho \phi\nabla \phi\|_{L^p(0,T;L^r)} &\leq \|\sqrt{t}\rho \phi\nabla \phi\|_{L^2(0,T;L^3)}^{\frac{2}{p}} \|\sqrt{t}\rho \phi\nabla \phi\|_{L^\infty(0,T;L^{3/2})}^{1 - \frac{2}{p}},\\
\|\sqrt{t}\rho \phi^3\nabla \phi\|_{L^p(0,T;L^r)} &\leq \|\sqrt{t}\rho \phi^3\nabla \phi\|_{L^2(0,T;L^3)}^{\frac{2}{p}} \|\sqrt{t}\rho \phi^3\nabla \phi\|_{L^\infty(0,T;L^{3/2})}^{1 - \frac{2}{p}}.
\end{align*}
By Theorem  \ref{1.1}, we conclude that
$\sqrt{t}\nabla^2\phi$,$\nabla\phi \in L^4(0, T; L^6)$,
thus
$\sqrt{t}\nabla^2\phi\nabla\phi \in L^2(0, T; L^3)$.
Similarly, we have
\begin{align*}
\sqrt{t}\nabla\phi\in L^\infty(0,T;L^6) ,
\nabla^2\phi\in L^\infty(0,T;L^2),
\end{align*}
which implies that
\begin{equation*}
\sqrt{t}\nabla\phi\nabla^2\phi\in L^\infty(0, T; L^{3/2}).
\end{equation*}
Hence
\begin{align*}
\|\sqrt{t} \nabla \phi\nabla^2\phi\|_{L^p(0,T;L^r)} &\leq \|\sqrt{t} \nabla \phi\nabla^2\phi\|_{L^2(0,T;L^3)}^{\frac{2}{p}} \|\sqrt{t} \nabla \phi\nabla^2\phi\|_{L^\infty(0,T;L^{3/2})}^{1 - \frac{2}{p}},
\end{align*}
with \(\frac{2}{p} + \frac{3}{r} = 2,\, p \geq 2\). Using the maximal regularity estimate for the Stokes equations and the standard estimate of elliptic equations for \eqref{eq-n} yields that
\begin{equation}
\left\lVert (\nabla^2 \sqrt{t} \mathbf{u}, \nabla^2 \sqrt{t} \phi,  \nabla^2 \sqrt{t} \mu) \right\rVert_{L^p(0,T;L^r)} + \left\lVert \nabla \sqrt{t} P \right\rVert_{L^p(0,T;L^r)} \leq C_{0,T} ,\label{eq-v}
\end{equation}
for $p \geq 2 $ and $ \frac{2}{p} + \frac{3}{r} = 2 $. Furthermore, from \eqref{eq-v} and the embedding \( W^1_r(\Omega) \hookrightarrow L^q(\Omega) \) with \( \frac{3}{q} = \frac{3}{r} - 1 \) if \( 1 \leq r < 3 \), we have
\begin{equation}
\nabla \sqrt{t} \mathbf{u}, \nabla \sqrt{t} \phi, \nabla \sqrt{t} \mu \in L^p(0, T; L^q)  \quad \text{with} \quad \frac{2}{p} + \frac{3}{r} = 2.\label{eq-t}
\end{equation}
Employing bounds  of \( \rho \mathbf{u}, \rho\mu, \rho\phi, \rho\phi^3, \nabla^2\phi \in L^\infty(0, T; L^6) \), \eqref{eq-t} and H\"older's inequality, yield that
\begin{align}
\sqrt{t} \rho \mathbf{u} \cdot \nabla \mathbf{u}, \sqrt{t} \rho \mathbf{u} \cdot \nabla \phi, \sqrt{t} \rho \mu \nabla \phi, \sqrt{t} \rho \phi \nabla \phi, \sqrt{t} \rho \phi^3 \nabla \phi, \sqrt{t} \nabla \phi\nabla^2\phi, \in L^p(0, T; L^r), \label{eq-13}
\end{align}
for all $(p, r)$, such that $p \geq 2$, $2 \leq r \leq \frac{6p}{3p - 4}$. From \eqref{eq-11}, \eqref{eq-12} and \eqref{eq-13}, we deduce that
\begin{equation}
\bigl\| (\nabla^2 \sqrt{t} \mathbf{u}, \nabla^2 \sqrt{t} \phi, \nabla^2 \sqrt{t} \mu) \bigr\|_{L^p(0, T; L^r)} + \bigl\| \nabla \sqrt{t} P \bigr\|_{L^p(0, T; L^r)} \leq C_{0, T},\label{eq-14}
\end{equation}
for all \( (p, r) \) such that $ p \geq 2, 2\leq r \leq \frac{6p}{3p - 4} $. Fix $p \in (2,4)$ such that $ps < 2p - 2s$,  and taking  $r=\frac{6p}{3p - 4}$,  thanks to $W_r^1(\Omega) \hookrightarrow L^\infty(\Omega)$(since $r > 3$ for $2 < p < 4$) and \eqref{eq-14}, we have
\[
\begin{split}
&\Big( \int_0^T \left\lVert \nabla \mathbf{u} \right\rVert_\infty^s dt \Big)^{\frac{1}{s}} + \Big( \int_0^T \left\lVert \nabla \phi \right\rVert_\infty^s dt \Big)^{\frac{1}{s}}
+\Big( \int_0^T \left\lVert \nabla\mu \right\rVert_\infty^s dt \Big)^{\frac{1}{s}}  \\
&\leq C \Big( \int_0^T \left\lVert \sqrt{t} \nabla \mathbf{u} \right\rVert_{W_r^1}^s \frac{dt}{(\sqrt{t})^s} \Big)^{\frac{1}{s}} + C \Big( \int_0^T \left\lVert \sqrt{t} \nabla \phi \right\rVert_{W_r^1}^s \frac{dt}{(\sqrt{t})^s} \Big)^{\frac{1}{s}}
+C \Big( \int_0^T \left\lVert \sqrt{t} \nabla \mu \right\rVert_{W_r^1}^s \frac{dt}{(\sqrt{t})^s} \Big)^{\frac{1}{s}} \\
&\leq C \Big( \int_0^T t^{-\frac{ps}{2p - 2s}} dt \Big)^{\frac{1}{s} - \frac{1}{p}} \Big( \left\lVert \nabla \sqrt{t} \mathbf{u} \right\rVert_{L_p(0,T;W_r^1)} + \left\lVert \nabla \sqrt{t} \phi \right\rVert_{L_p(0,T;W_r^1)}+\left\lVert \nabla \sqrt{t} \mu \right\rVert_{L_p(0,T;W_r^1)}  \Big) \\
&\leq C_{0} T^{-\frac{2p - 2s - ps}{2ps}},
\end{split}
\]
which concludes that \eqref{eq-s} holds.
\end{proof}
\begin{lemma}
Assume \( d = 2 \), then for all \( T > 0 \), \( p \in [2,\infty] \),  we have
\begin{equation}
\left\lVert (\nabla^2 \sqrt{t} \mathbf{u}, \nabla^2 \sqrt{t}\phi, \nabla^2 \sqrt{t} \mu) \right\rVert_{L^p(0,T;L^{r})} + \left\lVert \nabla \sqrt{t} P \right\rVert_{L^p(0,T;L^{r})} \leq C_{0,T},\label{eq-r}
\end{equation}
where $\ r \in [2,\frac{6p}{3p - 4}]$, \( C_{0,T} \) depends only on \( \rho^* \), \( p \) and the initial value. Furthermore, for \( s \in \left[1,2\right) \), then for some \( \theta > 0 \), we have
\begin{equation}
\int_0^T \left( \left\lVert \nabla \mathbf{u} \right\rVert_\infty^s + \left\lVert \nabla \mu\right\rVert_\infty^s + \left\lVert \nabla\phi \right\rVert_\infty^s \right) dt \leq C_{0,T} T^\theta.\label{eq-x}
\end{equation}
\end{lemma}
\begin{proof}
Using \eqref{2D} and \eqref{NSCH2}  yields that
\begin{equation*}
(\rho\sqrt{t}\mathbf{u}_t, \rho\sqrt{t}\phi_t) \in L^\infty(0, T; L^2), (\rho\sqrt{t}\mu_t) \in L^2(0, T; L^2).
\end{equation*}
 We have
\begin{equation*}
(\rho\sqrt{t}\mathbf{u}_t, \rho\sqrt{t}\phi_t) \in L^2(0, T; L^q) \quad \text{for} \quad q < \infty.
\end{equation*}
Then, it follows from the interpolation inequality, that
\begin{equation*}
\|\rho\sqrt{t}\mathbf{u}_t\|_{L^p(0,T;L^r)} \leq \|\rho\sqrt{t}\mathbf{u}_t\|_{L^\infty(0,T;L^2)}^{1 - \frac{2}{p}} \|\rho\sqrt{t}\mathbf{u}_t\|_{L^2(0,T;L^q)}^{\frac{2}{p}},
\end{equation*}
\begin{equation*}
\|\rho\sqrt{t}\phi_t\|_{L^p(0,T;L^r)} \leq \|\rho\sqrt{t}\phi_t\|_{L^\infty(0,T;L^2)}^{1 - \frac{2}{p}} \|\rho\sqrt{t}\phi_t\|_{L^2(0,T;L^q)}^{\frac{2}{p}},
\end{equation*}
with $\frac{1}{r} = \frac{p - 2}{2p} + \frac{2}{pq}$, $2 \leq r < p^*$, $p^*=\frac{2p}{p-2}$.
Thus
\begin{equation}
\|(\rho\sqrt{t}\mathbf{u}_t, \rho\sqrt{t}\phi_t)\|_{L^p(0,T;L^r)} \leq C_{0,T} \quad \text{for} \quad p \in [2, \infty], \quad r \in [2, p^*).
\end{equation}
And similarly, we have
\begin{equation}
\|(\rho\sqrt{t}\mathbf{u}, \rho\sqrt{t}\phi, \rho\sqrt{t}\mu, \rho\sqrt{t}\phi^3)\|_{L^p(0,T;L^r)} \leq C_{0,T} \quad \text{for} \quad p \in [2, \infty], \quad r \in [2, p^*).
\end{equation}
It is known from Theorem \ref{1.1} that $(\nabla \mathbf{u}, \nabla \phi)$ is bounded in  $L^\infty(0, T; L^2) \cap L^2(0, T; H^1)$. Employing the interpolation inequality, it is easy to get,  for $\frac{1}{r} = \frac{p - 2}{2p} + \frac{2}{pq}$, $2 \leq r < p^*$, that
\begin{align*}
\|\nabla \mathbf{u}\|_{L^p(0,T;L^r)} &\le \|\nabla \mathbf{u}\|_{L^\infty(0,T;L^2)}^{1 - \frac{2}{p}} \|\nabla \mathbf{u}\|_{L^2(0,T;L^q)}^{\frac{2}{p}} \\
&\le \|\nabla \mathbf{u}\|_{L^\infty(0,T;L^2)}^{1 - \frac{2}{p}} \|\nabla \mathbf{u}\|_{L^2(0,T;H^1)}^{\frac{2}{p}}, \\
\|\nabla \phi\|_{L^p(0,T;L^r)} &\le \|\nabla \phi\|_{L^\infty(0,T;L^2)}^{1 - \frac{2}{p}} \|\nabla \phi\|_{L^2(0,T;L^q)}^{\frac{2}{p}} \\
&\le \|\nabla \phi\|_{L^\infty(0,T;L^2)}^{1 - \frac{2}{p}} \|\nabla \phi\|_{L^2(0,T;H^1)}^{\frac{2}{p}},
\end{align*}
which implies that
\begin{equation}
\|(\nabla \mathbf{u}, \nabla \phi)\|_{L^p(0,T;L^r)} \le C_{0,T} \quad \text{for} \quad p \ge 2, \ r < p^*.
\end{equation}
Similarly, since $\nabla^2\phi$ is bounded in $L^\infty(0, T; L^2) \cap L^2(0, T; L^6)$, we get
\begin{align*}
\|\nabla ^2\phi\|_{L^p(0,T;L^r)} &\le \|\nabla ^2\phi\|_{L^\infty(0,T;L^2)}^{1 - \frac{2}{p}} \|\nabla ^2\phi\|_{L^2(0,T;L^q)}^{\frac{2}{p}},
\end{align*}
with \(\frac{1}{r} = \frac{p - 2}{2p} + \frac{2}{pq}\). Here, when \(q\) takes \(6\), then \(r\) may take the maximum value of \(\frac{6p}{3p - 4}\). Then we easily get
\begin{equation}
\|\nabla^2\phi\|_{L^p(0,T;L^r)} \le C_{0,T} \quad \text{for} \quad p \ge 2,\quad 2\leq r \leq \frac{6p}{3p - 4}.
\end{equation}
As obvious, we conclude that
\begin{align}
\| (\sqrt{t} \rho \mathbf{u} \cdot \nabla \mathbf{u}, \sqrt{t} \rho \mathbf{u} \cdot\nabla \phi, \sqrt{t} \rho \phi \nabla \phi, \sqrt{t}& \rho \phi^3 \nabla \phi, \sqrt{t} \rho \mu \nabla \phi ) \|_{L^p(0,T; L^r)} \leq C_{0,T}\, \text{for}\, p \in [2,\infty],  r \in [2,p^*),\\
\| (\sqrt{t} \nabla \phi\nabla^2 \phi) \|_{L^p(0,T; L^r)} &\leq C_{0,T} \quad \text{for} \quad p \in [2,\infty], \ r \in \Big[2, \frac{6p}{3p - 4}\Big].
\end{align}
Applying the maximal regularity estimate for the Stokes equations and the standard estimate for elliptic equations for \eqref{eq-n} yields that
\begin{equation}
\| (\nabla^2 \sqrt{t} \mathbf{u}, \nabla^2 \sqrt{t} \phi, \nabla^2 \sqrt{t} \phi, \nabla \sqrt{t} P ) \|_{L^p(0,T; L^r)} \leq C_{0,T} \quad \text{for} \quad p \in [2,\infty], \ r \in \Big[2, \frac{6p}{3p - 4}\Big].\label{eq-m}
\end{equation}
Furthermore, using the bound for $(\rho \mathbf{u}, \rho \phi,\rho\mu) \in L^\infty(0,T;L^6)$ and when $d =2$, the embedding $W_r^1(\Omega) \hookrightarrow L^q(\Omega)$ with $\frac{3}{q} = \frac{3}{r} - 1$ if $2 \leq r < 3$ ,  which implies that $(\nabla \sqrt{t} \mathbf{u}, \nabla \sqrt{t} \phi,\nabla \sqrt{t} \mu)$ is bounded in $L^p(0,T;L^q)$, we get \eqref{eq-m} for the full range of indices.

Fix \( p \in [2,\infty) \) so that \( ps < 2(p - s) \) and \( 1 \leq s < 2 \), which means that \( \Big( \int_0^T t^{-\frac{ps}{2p - 2s}} dt \Big)^{\frac{1}{s} - \frac{1}{p}} \leq C_{0,T} \).
Taking \( r \in \Big[2, \frac{6p}{3p - 4}\Big] \) such that the embedding \( W_r^1 \hookrightarrow L^\infty \), we conclude that
\[
\begin{split}
&\Big( \int_0^T \big\lVert \nabla \mathbf{u} \big\rVert_\infty^s dt \Big)^{\frac{1}{s}} + \Big( \int_0^T \big\lVert \nabla \phi \big\rVert_\infty^s dt \Big)^{\frac{1}{s}}
+\Big( \int_0^T \big\lVert \nabla\mu \big\rVert_\infty^s dt \Big)^{\frac{1}{s}}  \\
&\leq C \Big( \int_0^T \big\lVert \sqrt{t} \nabla \mathbf{u} \big\rVert_{W_r^1}^s \frac{dt}{(\sqrt{t})^s} \Big)^{\frac{1}{s}} + C \Big( \int_0^T \big\lVert \sqrt{t} \nabla \phi \big\rVert_{W_r^1}^s \frac{dt}{(\sqrt{t})^s} \Big)^{\frac{1}{s}}
+C \Big( \int_0^T \big\lVert \sqrt{t} \nabla \mu \big\rVert_{W_r^1}^s \frac{dt}{(\sqrt{t})^s} \Big)^{\frac{1}{s}} \\
&\leq C \Big( \int_0^T t^{-\frac{ps}{2p - 2s}} dt \Big)^{\frac{1}{s} - \frac{1}{p}} \Big( \big\lVert \nabla \sqrt{t} \mathbf{u} \big\rVert_{L_p(0,T;W_r^1)} + \big\lVert \nabla \sqrt{t} \phi \big\rVert_{L_p(0,T;W_r^1)}+\big\lVert \nabla \sqrt{t} \mu \big\rVert_{L_p(0,T;W_r^1)}  \Big) \\
&\leq C_{0} T^{-\frac{2p - 2s - ps}{2ps}},
\end{split}
\]
which deduces that \eqref{eq-x} holds.
\end{proof}
\section{The proof of Theorem \ref{1.2}} In order to prove  Theorem \ref{1.2}, we first show the Lagrangian formulation of the system  \eqref{NSCH}.
Now, we introduce the flow \( X : \mathbb{R}^+ \times \Omega \to \Omega\) of \( \mathbf{u} \) by
\[
\partial_t X(t, y) = \mathbf{u}(t, X(t, y)), \quad X(0, y) = y.
\]
Note that
\[
X(t, y) = y + \int_0^t \mathbf{u}(\tau, X(\tau, y)) d\tau,
\]
and
\[
\nabla_y X(t, y) = \text{Id} + \int_0^t \nabla_y \mathbf{u}(\tau, X(\tau, y)) d\tau.
\]
In Lagrangian coordinates \((t, y)\), a solution \((\rho, \mathbf{u}, \phi,\mu, P)\) to the system \eqref{NSCH} recasts in \((\bar{\rho}, \bar{\mathbf{u}}, \bar{\phi},\bar{\mu}, \bar{P})\) with
\begin{equation}
\begin{split}
\bar{\rho}(t, y) &= \rho(t, X(t, y)), \quad \bar{\mathbf{u}}(t, y) = \mathbf{u}(t, X(t, y)), \\
\bar{\phi}(t, y) &= \phi(t, X(t, y)),\quad \bar{\mu}(t, y) = \mu(t, X(t, y)),\quad \bar{P}(t, y) = P(t, X(t, y)),
\end{split}
\end{equation}
and the triplet \((\bar{\rho}, \bar{\mathbf{u}}, \bar{\phi},\bar{\mu},\bar{P})\) thus satisfies
\begin{equation}
\begin{cases}
\bar{\rho} \bar{\mathbf{u}}_t - \Dv_\mathbf{u}(\nu(\bar{\phi})\nabla_\mathbf{u}\bar{\mathbf{u}}+ \nabla_\mathbf{u} \bar{P} = -\Dv_\mathbf{u}(\nabla_\mathbf{u}\bar{\mathbf{u}}\otimes\nabla_\mathbf{u}\bar{\mathbf{u}}), \label{eq-z}
\\
\operatorname{div}_\mathbf{u} \bar{\mathbf{u}} = 0 , \\
\bar{\rho} \bar{\phi}_t = -\Delta_\mathbf{u}\bar{\mu},&  \\
\bar{\rho}_t = 0 , \\
\bar{\rho} \bar{\mu} = -\Delta_\mathbf{u}\bar{\phi}+\bar{\rho}(\bar{\phi}^3-\bar{\phi}),\\
\bar{\rho}(y, 0) = \rho_0(y), \bar{\mathbf{u}}(y, 0) = u_0(y), \bar{\phi}(y, 0) = \phi_0(y) ,
\bar{\mu}(y, 0) = \mu_0(y),&
\end{cases}
\end{equation}
where operators \(\nabla_\mathbf{u}, \Delta_\mathbf{u}, \nabla_\mathbf{u} \operatorname{div}_\mathbf{u}\) and \(\operatorname{div}_\mathbf{u}\) correspond to the original operators \(\nabla, \Delta, \nabla \operatorname{div}\), respectively, after performing the change to the Lagrangian coordinates.
As pointed out in \cite{RB1, B2, PZZ}, in our regularity framework, that latter system \eqref{eq-z} is equivalent to the system \eqref{NSCH}. Thanks to \eqref{eq-s} and \eqref{eq-x}, we can take the time $T$ to be small enough so that
\begin{equation}
\int_0^T \|\nabla \mathbf{u}\|_\infty d\tau \leq \frac{1}{2}.\label{eq-8}
\end{equation}
Set
\[
\begin{aligned}
A &= (\nabla X)^{-1} \text{(inverse of deformation tensor)}, \\
J &= \det \nabla X \text{(Jacobian determinant)}, \\
a &= JA \text{(transpose of cofactor matrix)}.
\end{aligned}
\]
Thus, in the \((t, y)\)-coordinates, operators \(\nabla\), \(\operatorname{div}\) and \(\Delta\) translate into
\begin{equation}
\nabla_u := {}^T\!A \nabla_y, \quad \operatorname{div}_\mathbf{u} := \operatorname{div}_y(A \cdot), \quad \text{and} \quad \Delta_\mathbf{u} := \operatorname{div}_\mathbf{u}\nabla_\mathbf{u}. \label{eq-6}
\end{equation}
Moreover, given some matrix \( N \), we define the divergence operator (acting on vector fields \( v \)) by the formulation
\begin{equation}
\operatorname{div}_\mathbf{u} N v = \operatorname{div}_y(N \cdot v) \stackrel{\text{def}}{=} {}^T\!N : \nabla v,
\end{equation}
where \( N : B = \sum_{i,j} N_{ij} B_{ji} \) for \( N = (N_{ij})_{1 \leq i,j \leq d} \) and \( B = (B_{ij})_{1 \leq i,j \leq d} \) two \( d \times d \) matrices.
Of course, if the condition \eqref{eq-8} is fulfilled then we have
\begin{equation}
A = \Big( \text{Id} + (\nabla_y X - \text{Id}) \Big)^{-1} = \sum_{k=0}^{+\infty} (-1)^k \Big( \int_0^t \nabla_y \bar{\mathbf{u}}(\tau, \cdot) d\tau \Big)^k,\label{eq-9}
\end{equation}
which yields that
\begin{equation}
\delta A = \Big( \int_0^t \nabla \delta \mathbf{u} d\tau \Big) \cdot \Big( \sum_{k \geq 1} \sum_{0 \leq j < k} C_1^j C_2^{k - 1 - j} \Big) \quad \text{with} \quad C_i(t) = \int_0^t \nabla \bar{\mathbf{u}}^i d\tau, \label{eq-7}
\end{equation}
where \( \delta A \stackrel{\text{def}}{=} A_2 - A_1 \) and \( \delta \mathbf{u} \stackrel{\text{def}}{=} \bar{\mathbf{u}}^2 - \bar{\mathbf{u}}^1 \). From \eqref{eq-8} and the convergence property of series \eqref{eq-9}, we known that $A$ is bounded.

Finally, we shall prove   Theorem \ref{1.2}. Let $(\rho^1, \mathbf{u}^1, \phi^1,\mu^1, P^1)$and $(\rho^2, \mathbf{u}^2,\phi^2,\mu^2,P^2)$ be two solutions of the system \eqref{NSCH} fulfilling the properties of Theorem \ref{1.1} with the same initial data, and denote by \((\bar{\rho}^1, \bar{\mathbf{u}}^1, \bar{\phi}^1,\bar{\mu}^1,  \bar{P}^1)\) and \((\bar{\rho}^2, \bar{\mathbf{u}}^2,\bar{\phi}^2, \bar{\mu}^2, \bar{P}^2)\) in Lagrangian coordinates. Of course, we have \(\bar{\rho}^1 = \bar{\rho}^2 = \rho_0\), which explains the choice of our approach here. In what follows, we shall use repeatedly the fact that for \( i = 1,2 \),
\begin{equation}
\begin{split}
&t^{\frac{1}{2}} \nabla \bar{\mathbf{u}}^i \in L^2(0,T; L^\infty), \quad t^{\frac{1}{2}} \nabla \bar{\phi}^i \in L^2(0,T; L^\infty), \quad t^{\frac{1}{2}} \nabla \bar{\mu}^i \in L^2(0,T; L^\infty), t^{\frac{1}{2}} \nabla \bar{P}^i \in L^2(0,T; L^4), \\
&t^{\frac{1}{2}} \bar{\mathbf{u}}_t^i \in L^{4/3}(0,T; L^6), \, \nabla \bar{\mathbf{u}}^i \in L^1(0,T; L^\infty) \cap L^2(0,T; L^6) \cap L^4(0,T; L^3),\, \nabla \bar{\phi}^i \in L^1(0,T; L^\infty) \cap L^\infty(0,T; w^{1,6}) ,\\
&\nabla \bar{\mu}^i \in L^1(0,T; L^\infty) \cap L^2(0,T; w^{1,6}) ,\quad
\bar{\mathbf{u}}^i \in L^4(0,T; L^\infty),\quad\bar{\phi}^i \in L^4(0,T; L^\infty),\quad\bar{\mu}^i \in L^4(0,T; L^\infty).\label{eq-5}
\end{split}
\end{equation}
It should be noted that the terms in \eqref{eq-5} is less than or equal to \( c(T) \), where \( c(T) \) designates a nonnegative continuous increasing function of \( T \), with \( c(0) = 0 \) and \( c(T) \to 0 \) when \( T \to 0 \). For example, in 3-D, using Lemma \ref{chazhi},  Theorem \ref{1.1} and Lemma  \ref{3DDD}, we have
\[
\begin{aligned}
\left\| t^{\frac{1}{2}} \nabla \bar{\mathbf{u}}^i \right\|_{L^2(0,T; L^\infty)}^2
&= \int_0^T t \left\| {}^T\!A \nabla \mathbf{u}^i \right\|_{L^\infty}^2 dt \\
&\leq C \int_0^T t \left\| \nabla \mathbf{u}^i \right\|_{L^2}^{\frac{1}{2}} \left\| \nabla^2 \mathbf{u}^i \right\|_{L^6}^{\frac{3}{2}} dt \\
&\leq C \sup_{t \in [0,T]} \left\| \nabla \mathbf{u}^i \right\|_{L^2}^{\frac{1}{2}} \int_0^T t \left\| \nabla^2 \mathbf{u}^i \right\|_{L^6}^{\frac{3}{2}} dt \\
&\leq C T^{\frac{1}{2}} \left\| \sqrt{t} \nabla^2 \mathbf{u}^i \right\|_{L^2(0,T; L^6)}^{\frac{3}{2}} \\
&\leq c(T),
\end{aligned}
\]
\[
\begin{aligned}
\left\| t^{\frac{1}{2}} \nabla \bar{\phi}^i \right\|_{L^2(0,T; L^\infty)}^2
&= \int_0^T t \left\| {}^T\!A \nabla \phi^i \right\|_{L^\infty}^2 dt \\
&\leq C \int_0^T t \left\| \nabla \phi^i \right\|_{L^2}^{\frac{1}{2}} \left\| \nabla^2 \phi^i \right\|_{L^6}^{\frac{3}{2}} dt \\
&\leq C \sup_{t \in [0,T]} \left\| \nabla \phi^i \right\|_{L^2}^{\frac{1}{2}} \int_0^T t \left\| \nabla^2 \phi^i \right\|_{L^6}^{\frac{3}{2}} dt \\
&\leq C T^{\frac{1}{2}} \left\| \sqrt{t} \nabla^2 \phi^i \right\|_{L^2(0,T; L^6)}^{\frac{3}{2}} \\
&\leq c(T),
\end{aligned}
\]
\[
\begin{aligned}
\left\| t^{\frac{1}{2}} \nabla \bar{\mu}^i \right\|_{L^2(0,T; L^\infty)}^2
&= \int_0^T t \left\| {}^T\!A \nabla \mu^i \right\|_{L^\infty}^2 dt \\
&\leq C \int_0^T t \left\| \nabla \mu^i \right\|_{L^2}^{\frac{1}{2}} \left\| \nabla^2 \mu^i \right\|_{L^6}^{\frac{3}{2}} dt \\
&\leq C \sup_{t \in [0,T]} \left\| \nabla \mu^i \right\|_{L^2}^{\frac{1}{2}} \int_0^T t \left\| \nabla^2 \mu^i \right\|_{L^6}^{\frac{3}{2}} dt \\
&\leq C T^{\frac{1}{2}} \left\| \sqrt{t} \nabla^2 \mu^i \right\|_{L^2(0,T; L^6)}^{\frac{3}{2}} \\
&\leq c(T).
\end{aligned}
\]
Due to $\|A_i\|_{\infty} < \infty\ (i = 1, 2)$, we have,\\
\[
\begin{aligned}
\left\|\bar{\phi}^i \right\|_{L^4(0,T; L^\infty)}^4
&= \int_0^T \left\| {}^T\!A \phi^i \right\|_{L^\infty}^4 dt \\
&\leq C \int_0^T  \left\| \nabla \phi^i \right\|_{L^2}^{2} \left\| \nabla^2 \phi^i \right\|_{L^2}^{2} dt \\
&\leq c(T),
\end{aligned}
\]
\[
\begin{aligned}
\left\|\bar{\mu}^i \right\|_{L^4(0,T; L^\infty)}^4
&= \int_0^T  \left\| {}^T\!A \nabla \mu^i \right\|_{L^\infty}^4 dt \\
&\leq C \int_0^T  \left\| \nabla \mu^i \right\|_{L^2}^{2} \left\| \nabla^2 \mu^i \right\|_{L^2}^{2} dt \\
&\leq C \sup_{\in [0,T]} \left\| \nabla \mu^i \right\|_{L^2}^{2} \int_0^T \left\| \nabla^2 \mu^i \right\|_{L^2}^{2} dt \\
&\leq C \left\|\nabla^2 \mu^i \right\|_{L^2(0,T; L^2)}^{2} \\
&\leq c(T).
\end{aligned}
\]
By Theorem \ref{1.1}, \eqref{eq-s} and $\|A_i\|_{\infty} < \infty\ (i = 1, 2)$, we get
$$\nabla \bar{\phi}^i \in L^1(0,T; L^\infty) \cap L^\infty(0,T; w^{1,6}) ,\quad
\nabla \bar{\mu}^i \in L^1(0,T; L^\infty) \cap L^2(0,T; w^{1,6}).$$
Denoting
\[
\delta u \stackrel{\text{def}}{=} \bar{\mathbf{u}}^2 - \bar{\mathbf{u}}^1, \quad \delta \phi \stackrel{\text{def}}{=} \bar{\phi}^2 - \bar{\phi}^1, \quad \delta\mu \stackrel{\text{def}}{=} \bar{\mu}^2 - \bar{\mu}^1
\quad\text{and} \quad \delta P \stackrel{\text{def}}{=} \bar{P}^2 - \bar{P}^1,
\]
we get
\begin{equation}
\begin{cases}
\rho_0 \, \delta \mathbf{u}_t - \mathop{\mathrm{div}_{\mathbf{u}^1}} \left( \nu(\bar{\phi}^1) \, \nabla_{\mathbf{u}^1} \delta \mathbf{u} + \nabla_{\mathbf{u}^1} \delta P \right.
- \rho_0 \, \bar{\mu}^1 \, \nabla_{\mathbf{u}^1} \delta \phi + \rho_0 \left( |\bar{\phi}^1|^3 - \bar{\phi}^1 \right) \, \nabla_{\mathbf{u}^1} \delta \phi
+ \nabla_{\mathbf{u}^1}^2 \bar{\phi}^1 \, \nabla_{\mathbf{u}^1} \delta \phi \Bigr) = \delta f_1, \\
\mathop{\mathrm{div}_{\mathbf{u}^1}} \delta \mathbf{u} = (\mathop{\mathrm{div}_{\mathbf{u}^1}} - \mathop{\mathrm{div}_{\mathbf{u}^2}}) \bar{\mathbf{u}}^2, \\
\rho_0 \, \delta \phi_t - \Delta_{\mathbf{u}^1} \delta \mu = \delta f_2, \\
\rho_0 \, \delta \mu + \Delta_{\mathbf{u}^1} \delta \phi - \rho_0 \left( \delta \phi^3 - \delta \phi \right) = \delta f_3, \\
(\delta \mathbf{u}, \delta \phi, \delta \mu) \big\vert_{t=0}= (0, 0, 0)
\end{cases}
\end{equation}
with
\[
\begin{aligned}
\delta f_1 \stackrel{\text{def}}{=} \big[ \mathrm{div}_{\mathbf{u}^2} (\nu(\bar{\phi}^1) \nabla _{\mathbf{u}^2} - \mathrm{div}_{\mathbf{u}^1}(\nu(\bar{\phi}^1) \nabla_{\mathbf{u}^1}) \big] \bar{\mathbf{u}}^2 - (\nabla_{\mathbf{u}^2} - \nabla_{\mathbf{u}^1}) \bar{P}^2 + (\rho_0 \bar{\mu}^2 \nabla_{\mathbf{u}^2} - \rho_0 \bar{\mu}^1 \nabla_{\mathbf{u}^1}) \bar{\phi}^2\\
- \rho_0 \big( |\bar{\phi}^2|^3 \nabla_{\mathbf{u}^2} - |\bar{\phi}^1|^3 \nabla_{\mathbf{u}^1} \big) \bar{\phi}^2 + \rho_0 \big( \bar{\phi}^2 \nabla_{\mathbf{u}^2} - \bar{\phi}^1 \nabla_{\mathbf{u}^1}\big) \bar{\phi}^2 - \big( \nabla^2_{\mathbf{u}^2} \bar{\phi}^2 \nabla_{\mathbf{u}^2} - \nabla^2_{\mathbf{u}^1} \bar{\phi}^1 \nabla_{\mathbf{u}^1}\big) \bar{\phi}^2
\end{aligned}
\]
\[
\delta f_2 \stackrel{\text{def}}{=} (\Delta_{\mathbf{u}^2} - \Delta_{\mathbf{u}^1}) \bar{\mu}^2, \quad \delta f_3 \stackrel{\text{def}}{=} -(\Delta_{\mathbf{u}^2} - \Delta_{\mathbf{u}^1}) \bar{\phi}^2
\]
We claim for sufficiently small \( T > 0 \):
\begin{align*}
\int_0^T \int_{\Omega} \bigl(|\delta \mathbf{u}(t, y)|^2 + |\delta \phi(t, y)|^2+|\delta \mu(t, y)|^2 + |\nabla \delta \mathbf{u}(t, y)|^2 \nonumber + |\nabla \delta \phi(t, y)|^2\bigr) \, dy \, dt = 0.
\end{align*}
To prove our claim, we first decompose \( \delta \mathbf{u} \) into:
\begin{align}
\delta \mathbf{u} = \varphi + \psi,
\end{align}
where \( \varphi \) is the solution given by Lemma \ref{lemma:4.1} to the following problem:
\begin{align}
\mathop{\mathrm{div}_{\mathbf{u}^1}} \varphi = (\mathop{\mathrm{div}_{\mathbf{u}^1}} - \mathop{\mathrm{div}_{\mathbf{u}^2}}) \bar{\mathbf{u}}^2 \nonumber= \mathop{\mathrm{div}}(\delta A \bar{\mathbf{u}}^2).
\end{align}
Then \eqref{4.9} and \eqref{eq-9} ensure that there exist two universal positive constants \( c \) and \( C \) such that if
\begin{align}
\|\nabla \bar{\mathbf{u}}^1\|_{L^1(0,T; L^\infty)} + \|\nabla \bar{\mathbf{u}}^1\|_{L^2(0,T; L^6)} &\leq c,\label{eq-q}
\end{align}
and then the following inequalities hold true:
\begin{align}
&\|\varphi\|_{L^4(0,T;L^2)} \leq C\|\delta A \bar{\mathbf{u}}^2\|_{L^4(0,T;L^2)}, \|\nabla \varphi\|_{L^2(0,T;L^2)} \leq C\|^T \delta A : \nabla \bar{\mathbf{u}}^2\|_{L^2(0,T;L^2)}, \notag\label{eq-a} \\
&\text{and } \|\varphi_t\|_{L^{4/3}(0,T;L^{3/2})} \leq C\|\delta A \bar{\mathbf{u}}^2\|_{L^4(0,T;L^2)} + C\|(\delta A \bar{\mathbf{u}}^2)_t\|_{L^{4/3}(0,T;L^{3/2})}.
\end{align}
Now, let us bound these terms in the right hand side  of \eqref{eq-a}. Regarding $^T \delta A : \nabla \bar{\mathbf{u}}^2$, it follows from Holder's inequality, \eqref{eq-7} and \eqref{eq-q} that
\begin{align}
\sup_{t \in [0,T]} \|t^{-1/2} \delta A\|_2 &\leq C \sup_{t \in [0,T]} \left\|t^{-1/2} \int_0^t \nabla \delta \mathbf{u} d\tau \right\|_2 \leq C\|\nabla \delta \mathbf{u}\|_{L^2(0,T;L_2)}.\label{eq-4}
\end{align}
According to \eqref{eq-5} and \eqref{eq-4}, we obtain
\begin{align}
\|^T \delta A : \nabla \bar{\mathbf{u}}^2\|_{L^2(0,T \times \mathbb{T}^d)} &\leq \sup_{t \in [0,T]} \|t^{-1/2} \delta A\|_2 \|t^{1/2} \nabla \bar{\mathbf{u}}^2\|_{L^2(0,T;L^\infty)} \notag \\
&\leq c(T)\|\nabla \delta \mathbf{u}\|_{L^2(0,T;L^2)}.
\end{align}
Similarly, we also have
\begin{align}
\|\delta A \bar{\mathbf{u}}^2\|_{L^4(0,T;L^2)} &\leq \|t^{-1/2} \delta A\|_{L^\infty(0,T;L^2)} \|t^{1/2} \bar{\mathbf{u}}^2\|_{L^4(0,T;L^\infty)}.
\end{align}
Using \eqref{eq-5}, \eqref{eq-a} and \eqref{eq-4} yields that
\begin{align}
\|\nabla \varphi\|_{L^2(0,T;L^2)} &\leq c(T)\|\nabla \delta \mathbf{u}\|_{L^2(0,T;L^2)},\label{eq-b}
\end{align}
and
\begin{align}
\|\varphi\|_{L^4(0,T;L^2)} &\leq c(T)\|\nabla \delta \mathbf{u}\|_{L^2(0,T;L^2)}.
\label{eq-c}
\end{align}
In order to bound $\varphi_t$, it suffices to derive an appropriate estimate in $L^{4/3}(0,T;L^{3/2})$ for
\begin{align*}
(\delta A \bar{\mathbf{u}}^2)_t = \delta A \bar{\mathbf{u}}_t^2 + (\delta A)_t \bar{\mathbf{u}}^2.
\end{align*}
Thanks to \eqref{eq-5} and \eqref{eq-4}, we have
\begin{align*}
\|\delta A \bar{\mathbf{u}}_t^2\|_{L^{4/3}(0,T;L^{3/2})} &\leq \|t^{-1/2} \delta A\|_{L^\infty(0,T;L^2)} \|t^{1/2} \nabla \bar{\mathbf{u}}_t^2\|_{L^{4/3}(0,T;L^6)} \notag \\
&\leq c(T)\|\nabla \delta \mathbf{u}\|_{L^2(0,T;L^2)}.
\end{align*}
For the term $(\delta A)_t \bar{\mathbf{u}}^2$, it follows from H\"{o}lder's inequality that
\begin{align*}
\|(\delta A)_t \bar{\mathbf{u}}^2\|_{L^{4/3}(0,T;L^{3/2})} &\leq \|(\delta A)_t\|_{L^2(0,T \times \mathbb{T}^d)} \|\bar{\mathbf{u}}^2\|_{L^4(0,T;L^6)}.
\end{align*}
Furthermore, differentiating \eqref{eq-7} with respect to \( t \) and using \eqref{eq-q} for \( \bar{\mathbf{u}}^1 \) and \( \bar{\mathbf{u}}^2 \) yield that
\begin{align*}
\| (\delta A)_t \|_2 &\leq C \Big( \| \nabla \delta \mathbf{u} \|_2 + \Big\| t^{-1/2} \int_0^t \nabla \delta \mathbf{u} \, dr \Big\|_2 \big( \| t^{1/2} \nabla \bar{\mathbf{u}}^1 \|_\infty + \| t^{1/2} \nabla \bar{\mathbf{u}}^2 \|_\infty \big) \Big),
\end{align*}
which implies that  $$ \| (\delta A)_t \|_{L^2(0,T \times \mathbb{T}^d)} \leq C \| \nabla \delta \mathbf{u} \|_{L^2(0,T \times \mathbb{T}^d)}.$$
Combining this with \eqref{eq-5} yields  that
\begin{align}
\| (\delta A)_t \bar{\mathbf{u}}^2 \|_{L^{4/3}(0,T; L^{3/2})} &\leq c(T) \| \nabla \delta \mathbf{u} \|_{L^2(0,T \times \mathbb{T}^d)},
\label{eq-3}\\ \notag
\text{and thus} \quad \| \varphi_t \|_{L^{4/3}(0,T; L^{3/2})} &\leq c(T) \| \nabla \delta \mathbf{u} \|_{L^2(0,T; L^2)}.
\end{align}
Collecting results from \eqref{eq-a}, \eqref{eq-b}, \eqref{eq-c} and \eqref{eq-3}, we obtain
\begin{equation}
\| \varphi \|_{L^4(0,T; L^2)} + \| \nabla \varphi \|_{L^2(0,T \times \mathbb{T}^d)} + \| \varphi_t \|_{L^{4/3}(0,T; L^{3/2})} \leq c(T) \| \nabla \delta \mathbf{u} \|_{L^2(0,T \times \mathbb{T}^d)}.\label{eq-1}
\end{equation}
Next, let us restate the equations for $(\delta \mathbf{u}, \delta \phi,\delta \mu, \delta P)$ as the following system for $(\psi, \delta \phi,\delta\mu, \delta P)$:
\begin{equation}
\begin{cases}
\rho_0 \,\psi_t - \mathop{\mathrm{div}_{\mathbf{u}^1}} \left( \nu(\bar{\phi}^1) \, \nabla_{\mathbf{u}^1} \psi + \nabla_{\mathbf{u}^1} \delta P \right.
= \delta f_1-\rho_0 \,\varphi_t+\mathop{\mathrm{div}_{\mathbf{u}^1}} \left( \nu(\bar{\phi}^1) \, \nabla_{\mathbf{u}^1} \varphi \right)+\rho_0 \, \bar{\mu}^1 \, \nabla_{\mathbf{u}^1} \delta \phi \label{eq-10}\\
- \rho_0 \left( |\bar{\phi}^1|^3 + \bar{\phi}^1 \right) \, \nabla_{\mathbf{u}^1} \delta \phi
-\nabla_{\mathbf{u}^1}^2 \bar{\phi}^1 \, \nabla_{\mathbf{u}^1} \delta \phi  , \\
\mathop{\mathrm{div}_{\mathbf{u}^1}} \psi = 0, \\
\rho_0 \, \delta \phi_t - \Delta_{\mathbf{u}^1} \delta \mu = \delta f_2, \\
\rho_0 \, \delta \mu + \Delta_{\mathbf{u}^1} \delta \phi - \rho_0 \left( \delta \phi^3 - \delta \phi \right) = \delta f_3, \\
(\delta \mathbf{u}, \delta \phi, \delta \mu) \big\vert_{t=0}= (0, 0, 0).
\end{cases}
\end{equation}
Due to $\operatorname{div}_{\mathbf{u}^1} \psi = 0$, we have
\[
\int_{\Omega} (\nabla_{\mathbf{u}^1} \delta P) \cdot \psi  \, dx = - \int_{\Omega} \operatorname{div}_{\mathbf{u}^1} \psi  \cdot \delta P \, dx = 0.
\]
Taking the \( L^2 \)-scalar product of the first equation in the system \eqref{eq-10} with \( \psi \), the third equation with \( \delta \phi \) and the fourth equation with \( \delta \mu \) respectively,   finally adding the three equations together, we infer that
\begin{equation}
\begin{split}
&\frac{1}{2} \frac{d}{dt} \int_{\Omega} \rho_0 \big( |\psi|^2 + |\delta \phi|^2 \big) dx + \int_{\Omega} \big( \nu(\bar{\phi}^1) |\nabla_{\mathbf{u}^1} \psi|^2 + \rho_0 |\ \delta \mu|^2 \big) dx \\
&\leq -\int_\Omega \rho_0 \partial_t \varphi \cdot \psi \, dx + \int_\Omega div_{\mathbf{u}^1} (\nu(\bar{\phi^1}) \nabla_{\mathbf{u}^1} \varphi) \cdot \psi \, dx + \int_\Omega \rho_0 \bar{\mu}_1 \nabla_{\mathbf{u}^1} \delta\phi \, \psi \, dx \\
&\quad-\int_\Omega \rho_0 (|\bar{\phi}^1|^3 - \bar{\phi}^1) \nabla_{\mathbf{u}^1} \delta\phi \, \psi \, dx - \int_\Omega\nabla^2_{\mathbf{u}^1} \bar{\phi^1} \nabla_{\mathbf{u}^1} \delta\phi\, \psi \, dx + \int_\Omega (\nu(\bar{\phi^1}) \nabla_{\mathbf{u}^2} \bar{\mathbf{u}}^2 \nabla_{\mathbf{u}^2}\psi - \nu(\bar{\phi^1}) \nabla_{\mathbf{u}^1} \bar{u}^2 \nabla_{u^1}\psi) \, dx \\
&\quad-\int_\Omega (\nabla_{\mathbf{u}^2} - \nabla_{\mathbf{u}^1}) \bar{P}^2 \, \psi \, dx + \int_\Omega (\rho_0 \bar{\mu}^2 \nabla_{\mathbf{u}^2} - \rho_0 \bar{\mu}^1 \nabla_{\mathbf{u}^1}) \bar{\phi}^2 \, \psi \, dx \\
&\quad-\int_\Omega \rho_0 (|\bar{\phi}^2|^3 \nabla_{\mathbf{u}^2} - |\bar{\phi}^1|^3 \nabla_{\mathbf{u}^1}) \bar{\phi}^2 \, \psi \, dx + \int_\Omega \rho_0 (\bar{\phi}^2 \nabla_{\mathbf{u}^2} - \bar{\phi}^1 \nabla_{\mathbf{u}^1}) \bar{\phi}^2 \, \psi \, dx \\
&\quad-\int_\Omega (\nabla_{\mathbf{u}^2}^2 \bar{\phi}^2 \nabla_{\mathbf{u}^2} - \nabla_{\mathbf{u}^1}^2 \bar{\phi}^1 \nabla_{\mathbf{u}^1}) \bar{\phi}^2 \, \psi \, dx + \int_\Omega \Delta_{\mathbf{u}^1} \delta\mu \delta\phi \, dx + \int_\Omega (\Delta_{\mathbf{u}^2} - \Delta_{\mathbf{u}^1}) \bar{\mu}^2 \delta\phi \, dx \\
&\quad-\int_\Omega \Delta_{\mathbf{u}^1} \delta\phi \delta \mu \, dx + \int_\Omega \rho_0 (\delta\phi^3 - \delta\phi) \delta \mu \, dx - \int_\Omega (\Delta_{\mathbf{u}^2} - \Delta_{\mathbf{u}^1}) \bar{\phi}^2 \delta \mu \, dx
\\&\triangleq\sum_{i = 1}^{16} M_i
\end{split}
\end{equation}
Here and in what follows, we shall estimate term by term above. For \( M_1 \), it follows from H\"{o}lder's inequality that
\[
\int_0^T M_1(t)dt \leq \| \rho_0 \|_\infty^{3/4} \| \varphi_t \|_{L^{4/3}(0,T; L^{3/2})} \| \rho_0^{1/4} \psi \|_{L^4(0,T; L^3)}.
\]
Using H\"{o}lder's inequality and the Sobolev embedding \( H^1(\Omega) \hookrightarrow L^6(\Omega) \) yields that
\[
\| \rho_0^{1/4} \psi \|_{L^4(0,T; L^3)} \leq \| \sqrt{\rho_0} \psi \|_{L^\infty(0,T; L^2)}^{1/2} \| \psi \|_{L^2(0,T; L^6)}^{1/2} \leq C \| \sqrt{\rho_0} \psi \|_{L^\infty(0,T; L^2)}^{1/2} \| \psi \|_{L^2(0,T; H^1)}^{1/2}.
\]
Taking advantage of \eqref{eq-7} and \eqref{eq-1}, we conclude that
\[
\int_0^T M_1(t)dt \leq c(T) \Big( \| \sqrt{\rho_0} \psi \|_{L^\infty(0,T; L^2)} + \| \nabla \psi \|_{L^2(0,T \times \Omega} \Big)^{1/2} \| \sqrt{\rho_0} \psi \|_{L^\infty(0,T; L^2)}^{1/2} \| \nabla \delta \mathbf{u} \|_{L^2(0,T \times \Omega)}.
\]
For \( M_2 \), it follows from integrating by parts and using \eqref{eq-1}, that
\[
\begin{split}
\int_0^T M_2(t)dt &\leq \nu^* \int_0^T \left| \int_{\Omega} \nabla_{\mathbf{u}^1} \varphi \nabla_{\mathbf{u}^1} \psi dx \right| dt \\
&\leq \nu^* \int_0^T \int_{\Omega} |\nabla_{u^1} \varphi| |\nabla_{\mathbf{u}^1} \psi| dxdt \\
&\leq \frac{\nu^*}{2} \int_0^T \|\nabla_{\mathbf{u}^1}\psi\|_2^2 dt + \frac{\nu^*}{2} \int_0^T \|\nabla_{\mathbf{u}^1} \varphi\|_2^2 dt \\
&\leq \frac{\nu^*}{2} \int_0^T \|\nabla_{\mathbf{u}^1} \psi\|_2^2 dt + c(T) \int_0^T \|\nabla \delta \mathbf{u}\|_2^2 dt.
\end{split}
\]
For \( M_3 \) and \( M_4 \),  it follows from H\"{o}lder's inequality that
\begin{align*}
\int_{0}^{t} M_{3}(t) dt &\leq \|\sqrt\rho_0 \psi\|_{L^{\infty}(0,T; L^2(\Omega)} \Big( \|\nabla \bar{\phi^1}\|_{L^2(0,T; L^4(\Omega)} + \|\nabla \bar{\phi^2}\|_{L^2(0,T; L^4(\Omega)} \Big) \|\bar{\mu^1}\|_{L^2(0,T; L^4(\Omega)}, \\
\int_{0}^{t} M_{4}(t) dt &\leq \|\sqrt\rho_0 \psi\|_{L^{\infty}(0,T; L^2(\Omega))} \Big( \|\nabla \bar{\phi^1}\|_{L^2(0,T; L^4(\Omega))} + \|\nabla \bar{\phi^2}\|_{L^2(0,T; L^4(\Omega)} \Big)\\
&\quad\times\Big( \||\nabla \bar{\phi^1}|^{3}\|_{L^2(0,T; L^4(\Omega)} + \|\bar{\phi^1}\|_{L^2(0,T; L^4(\Omega)} \Big).
\end{align*}
Thanks to \eqref{eq-8} and \eqref{eq-9}, we get
\begin{align*}
&\| \nabla A \|_{L^\infty (0,T; L^3(\Omega)} \leq C \| \nabla^2 \mathbf{u}(t, X(\tau, \cdot)) \|_{L^1(0,T; L^3(\Omega)} \| \nabla_y X \|_{L_t^\infty (L_y^\infty)} \leq C t^{\frac{1}{4}},\\
&\int_{0}^{t} M_{5}(t) dt \leq\int_{0}^{t} \bigl\lvert (A_1^T \nabla A_1^T \bar{\phi^1} + |A_1^T|^2 \nabla^2 \bar{\phi^1}) \, |A_1^T \psi \, (|\nabla \bar{\phi^2}| + |\nabla \bar{\phi^1}|) \bigr\rvert \, dt \\
&\leq C \rho_0 \,\|\sqrt\rho_0 \psi\|_{L^\infty(0,T; L^2(\Omega)}\Big(\|\nabla A_1^T\|_{L^\infty(0,T; L^3(\Omega)} \,\|A_1^T\|_{L^\infty(0,T; L^\infty(\Omega)}^2
\|\bar{\phi^1}\|_{L^\infty(0,T\times\Omega)}\\
&\quad+\|A_1^T\|_{L^\infty(0,T; L^\infty(\Omega)}^3\,
\|\nabla^2 \bar{\phi^1}\|_{L^2(0,T; L^3(\Omega)} \,\Big)\Big( \|\nabla \bar{\phi^2}\|_{L^2(0,T; L^{6}(\Omega)} + \|\nabla \bar{\phi^1}\|_{L^2(0,T; L^{6}(\Omega)} \Big).
\end{align*}
For \( M_6 \), using \eqref{eq-6} and \eqref{eq-7}, we deduce that
\begin{align*}
M_6 &\leq C \left| \int_{\Omega}\Dv\big((\delta A^T A_2 + A_1^T \delta A)\nabla \bar{\mathbf{u}}^2\big) \cdot \psi \, dx \right| \\
&\leq C \int_{\Omega} \big| \delta A^T A_2 + A_1^T \delta A \big| \| \nabla \bar{\mathbf{u}}^2 \| |\nabla \psi| \, dx \\
&\leq C \| t^{-1/2} \delta A \|_2 \| t^{1/2} \nabla \bar{\mathbf{u}}^2 \|_\infty \| \nabla \psi \|_2,
\end{align*}
which together with \eqref{eq-5} and \eqref{eq-4} implies that
\begin{align*}
\int_0^T M_6(t) dt &\leq \| t^{-1/2} \delta A \|_{L^\infty(0,T; L^2)} \| t^{1/2} \nabla \bar{\mathbf{u}}^2 \|_{L^2(0,T; L^\infty)} \| \nabla \psi \|_{L^2(0,T \times \Omega)} \\
&\leq c(T) \| \nabla \delta \mathbf{u} \|_{L^2(0,T; L^2)} \| \nabla \psi \|_{L^2(0,T \times \Omega)}.
\end{align*}
For \( M_7 \), using Hl\"{o}der's inequality, we obtain
\[
M_7(t) \leq \left| \int_{\Omega} \delta A \nabla \bar{P}^2 \psi \, dx \right| \leq C \| t^{-1/2} \delta A \|_2 \| t^{1/2} \nabla \bar{P}^2 \|_3 \| \psi \|_6.
\]
It then follows from \eqref{eq-5}, \eqref{eq-4} and Sobolev embedding that
\begin{align*}
\int_0^T M_7(t) dt &\leq \| t^{-1/2} \delta A \|_{L^\infty(0,T; L^2)} \| t^{1/2} \nabla \bar{P}^2 \|_{L^2(0,T; L^3)} \| \psi \|_{L^2(0,T; H^1)} \\
&\leq C(T) \| \nabla \delta \mathbf{u}\|_{L^2(0,T; L^2)} \left( \| \sqrt{\rho_0} \psi \|_{L^\infty(0,T; L^2(\Omega))} + \| \nabla \psi \|_{L^2(0,T \times \Omega)} \right),\\
\int_{0}^{T} M_8(t) \, dt &\leq \|\rho_0 \psi\|_{L^\infty(0,T; L^2(\Omega)} \, (\|\bar{\mu}^2\|_{L^2(0,T; L^4(\Omega))} \|\nabla A_2^T\|_{L^\infty(0,T\times \Omega)}
\\ &\quad+ \|\bar{\mu}^1\|_{L^2(0,T; L^4(\Omega))} \|\nabla A_1^T\|_{L^\infty(0,T\times \Omega)})\|\nabla\bar{\phi}^2\|_{L^2(0,T; L^4(\Omega)},\\
\int_0^T M_9(t) \, dt &\leq \| \rho_0 \psi \|_{L^\infty(0,T;L^2(\Omega)} \left( \| A_2^T \|_{L^\infty(0,T\times\Omega)} \| |\bar{\phi^2}|^3 \|_{L^2(0,T;L^4(\Omega)} \right. \\
&\quad+ \| A_1^T \|_{L^\infty(0,T\times\Omega)} |\bar{\phi^1}|^3 \|_{L^2(0,T;L^4(\Omega)} \left. \| \nabla \bar{\phi^2}\|_{L^2(0,T;L^4(\Omega)} \right) \\
\int_0^T M_{10}(t) \, dt &\leq \|\rho_0 \psi \|_{L^\infty(0,T;L^2(\Omega)} \left( \| A_2^T \|_{L^\infty(0,T\times\Omega)} \| \bar{\phi^2} \|_{L^2(0,T;L^4(\Omega)} \right.\\
&\quad+ \| A_1^T \|_{L^\infty(0,T\times\Omega)} \|\bar{\phi^1} \|_{L^2(0,T;L^4(\Omega)} \left. \| \nabla \bar{\phi^2} \|_{L^2(0,T;L^4(\Omega)} \right), \\
\int_0^T M_{11}(t) \, dt &\leq C_0 \| \rho_0 \psi \|_{L^\infty(0,T;L^2(\Omega)}\| \nabla \bar{\phi^2} \|_{L^1(0,T;L^\infty(\Omega)} \\
 &\quad \times\Big( \| \nabla A_2^T \|_{L^\infty(0,T;L^3(\Omega)} \| A_2^T \|_{L^\infty(0,T\times\Omega)}^2 | \nabla \bar{\phi^2} \|_{L^\infty(0,T;L^6(\Omega)}
 \\&\quad + \| \nabla A_1^T \|_{L^\infty(0,T;L^3(\Omega)}\| A_1^T \|_{L^\infty(0,T\times\Omega)}^3 \| \nabla^2 \bar{\phi}^1 \|_{L^\infty(0,T;L^6(\Omega)} \Big),
\end{align*}
\begin{align*}
\int_{0}^{T} M_{12}(t) \, dt &\leq \| A_{1}^T \|_{L^{\infty}(0,T \times \Omega)}^2 \left( \| \nabla \bar{\mu}^2 \|_{L^2(0,T \times \Omega)} + \| \nabla \bar{\mu}^1 \|_{L^2(0,T \times \Omega)} \right) \\
&\quad \times \left( \| \nabla \bar{\phi}^2 \|_{L^2(0,T \times \Omega)} + \| \nabla \bar{\phi}^1 \|_{L^2(0,T \times \Omega)} \right),\\
\int_{0}^{T} M_{13}(t) \, dt &\leq \left( \| A_2^T \|_{L^{\infty}(0,T \times \Omega)}^2 + \| A_1^T \|_{L^{\infty}(0,T \times \Omega)}^2 \right) \| \nabla \bar{\mu}^2 \|_{L^2(0,T \times \Omega)}\\&\quad\times \left( \| \nabla \bar{\phi}^2 \|_{L^2(0,T \times \Omega)} + \| \nabla \bar{\phi}^1 \|_{L^2(0,T \times \Omega)} \right), \\
\int_{0}^{T} M_{14}(t) \, dt &\leq \| A_1^T \|_{L^{\infty}(0,T \times \Omega)}^2 \left( \| \nabla \bar{\phi}^2 \|_{L^2(0,T \times \Omega)} + \| \nabla \bar{\phi}^1 \|_{L^2(0,T \times \Omega)} \right)\\&\quad\times \left( \| \nabla \bar{\mu}^2 \|_{L^2(0,T \times \Omega)} + \| \nabla \bar{\mu}^1 \|_{L^2(0,T \times \Omega)} \right), \\
\int_{0}^{T} M_{15}(t) dt &\leq \left( \left\| |\bar{\phi}^2|^3 \right\|_{L^2(0,T \times \Omega)} + \left\| |\bar{\phi}^1|^3 \right\|_{L^2(0,T \times \Omega)} \right) \left\| \sqrt{\rho_0} \delta \mu \right\|_{L^2(0,T \times \Omega)} \\
&\quad + \left\| \sqrt{\rho_0} \delta \phi \right\|_{L^{\infty}; L^2(\Omega)} \left( \left\| \bar{\mu}^2 \right\|_{L^1(0,T; L^2(\Omega))} + \left\| \bar{\mu}^1 \right\|_{L^1(0,T; L^2(\Omega))} \right), \\
\int_{0}^{T} M_{16}(t) dt &\leq \left( \left\| A_2^T \right\|_{L^{\infty}(0,T \times \Omega)}^2 + \left\| A_1^T \right\|_{L^{\infty}(0,T \times \Omega)}^2 \right)
\| \nabla \bar{\phi}^2 \|_{L^2(0,T \times \Omega)} \\&\quad+ \left(\| \nabla \bar{\mu}^2 \|_{L^2(0,T \times \Omega)}+\| \nabla \bar{\mu}^1 \|_{L^2(0,T \times \Omega)} \right).
\end{align*}
So altogether, and using \eqref{eq-3}, this gives for all small enough \( T > 0 \),
\begin{align}
\sup_{t \in [0,T]}& \big\| (\sqrt{\rho_0} \psi, \sqrt{\rho_0} \delta\psi) \big\|_2^2 + \big\| (\nabla \delta \mathbf{u}, \sqrt{\rho_0} \delta \mu) \big\|_{L^2(0,T; L^2)}^2  \label{eq-2}\\ \notag
&\leq c(T) \big\| (\nabla \delta \mathbf{u}, \sqrt{\rho_0} \delta \mu) \big\|_{L^2(0,T; L^2)}^2+c(T)\|\sqrt{\rho_0} \delta \phi\|_{L^\infty(0,T; L^2)}^2 .
\end{align}
By \eqref{eq-1}, we conclude that
\begin{equation*}\begin{split}
&\big\| (\nabla \delta \mathbf{u}, \sqrt{\rho_0} \delta \mu) \big\|_{L^2(0,T; L^2)}^2+\|\sqrt{\rho_0} \delta \phi\|_{L^\infty(0,T; L^2)}^2
\\&\leq c(T) \big\| (\nabla \delta \mathbf{u}, \sqrt{\rho_0} \delta \mu) \big\|_{L^2(0,T; L^2)}^2+c(T)\|\sqrt{\rho_0} \delta \phi\|_{L^\infty(0,T; L^2)}^2.
\end{split}
\end{equation*}
Hence \( \nabla \delta \mathbf{u} =  \delta \phi=\delta \mu \equiv 0 \) on \( [0,T] \times \Omega \) if \( T \) is small enough. Plugging that information into \eqref{eq-2} yields
that
\[
\big\| (\sqrt{\rho_0} \psi, \sqrt{\rho_0} \delta \phi) \big\|_{L^\infty(0,T; L^2)}^2 + \big\| (\nabla \psi, \sqrt{\rho_0} \delta \mu) \big\|_{L^2(0,T \times \Omega)}^2 = 0.
\]
Combining with \eqref{NSCH2} finally implies that \( \psi \equiv 0 \) on \( [0,T] \times \mathbb{T}^d \), and \eqref{eq-1} clearly yields \( \varphi \equiv 0 \). Therefore, for small enough \( T > 0 \),  we  conclude,  by \eqref{eq-1}, that
\[
\bar{\mathbf{u}}^1 = \bar{\mathbf{u}}^2, \quad \bar{\phi}^1 = \bar{\phi}^2, \quad \bar{\mu}^1 = \bar{\mu}^2 \quad \text{on} \quad [0,T] \times \Omega.
\]
Reverting to Eulerian coordinates, we finally deduce that  two  strong solutions of the system \eqref{NSCH} coincide on \( [0,T] \times \Omega\). 
\section{Declarations}



\noindent{\bf Competing interests  }\

On behalf of all authors, the corresponding author states that there is no potential conflicts of interest with respect to the research of this article.

\noindent{\bf Authors' contributions   }\

Lingxin Jiang, Jiahong Wu and Fuyi  Xu contributed equally to this work.

\noindent{\bf Funding   }\

Jiang and Xu were partially supported by the
National Natural Science Foundation of China (12326430), the  Natural Science Foundation of Shandong Province (ZR2021MA017).  Wu was partially supported by the  National Science Foundation of the United
States under DMS 2104682 and DMS 2309748.

\noindent{\bf Availability of data and materials  }\

Data and materials  sharing not applicable to this article as no data and  materials
 were generated or analyzed during the current study.

\begin{center}

\end{center}


\begin{thebibliography}{99}
\addcontentsline{toc}{section}{References}

\bibitem{A}
H. Abels, {\it On a diffuse interface model for two-phase flows of viscous, incompressible fluids with matched densities,}  Arch. Ration. Mech. Anal., 194 (2009), 463--506.

\bibitem{ADG}
H. Abels, D. Depner and  H. Garcke, {\it Existence of weak solutions for a diffuse interface model for two-phase flows of incompressible fluids with different densities,} J. Math. Fluid Mech., 15  (2013), 453--480.

\bibitem{AF}
H. Abels and  E. Feireisl, {\it On a diffuse interface model for a two-phase flow of compressible viscous fluids,} Indiana Univ. Math. J, 57 (2008), 659--698.

\bibitem{AW}
H. Abels and M. Wilke, {\it Convergence to equilibrium for the Cahn-Hilliard equation with a logarithmic free energy,} Nonlinear Anal., 67  (2007), 3176--3193.



\bibitem{HW}
H. Abels  and J. Weber, {\it Stationary Solutions for a Navier-Stokes/Cahn-Hilliard System with Singular Free Energies,} Recent Developments of Mathematical Fluid Mechanics. Basel: Springer Basel, (2016), 25--41.


\bibitem{AMW} D.M. Anderson, G.B. McFadden and  A.A. Wheeler, {\it Diffuse-interface methods in fluid mechanics,} Annu. Rev. Fluid Mech., 30 (1998), 139--165.

\bibitem{BDM}
T. Biswas, S. Dharmatti, P. Mahendranath and M. Mohan, {\it On the stationary nonlocal cahn-chilliard-navier-stokes system: existence, uniqueness and exponential stability,} Asymptot. Anal., 125 (2021), 59--99.


\bibitem{Bo} F. Boyer, {\it Mathematical study of multi-phase flow under shear through order parameter formulation,} Asymptot. Anal., 20 (1999), 175--212.


\bibitem{CG} C. Cao and  C.G. Gal,  {\it Global solutions for the 2D NSCH model for a two-phase flow of viscous, incompressible fluids with mixed partial viscosity and mobility,} Nonlinearity, 25 (2012), 3211--3234.




\bibitem{CHK} H. Choe and H. Kim,  {\it  Strong solutions of the Navier-Stokes equations for nonhomogeneous incompressible fluids,} Comm. Partial Differential Equations, 28
(2003), 1183--1201.




\bibitem{DM1} R. Danchin and P. Mucha,  {\it Divergence,} Discrete and Cont. Dyn. Systems S,
6 (2013), 1163--1172.

\bibitem{RB1}
R. Danchin and P. Mucha, {\it Incompressible flows with piecewise constant density,}  Arch. Ration. Mech. Anal., 207 (2013), 991--1023.


\bibitem{RB}
R. Danchin and P. Mucha, {\it The Incompressible Navier-Stokes Equations in Vacuum,} Comm. Pure
Appl. Math.,   72 (2019), 1351--1385.


\bibitem{B2}
B. Desjardins, {\it Regularity of weak solutions of the compressible isentropic Navier-Stokes equations,}  Commun. Partial Differ. Equ.,  22 (1997), 977--1008.


\bibitem{UM}
G. De Ugo Ue and  T.T. Medjo, {\it Convergence of the solution of the stochastic 3d globally modified cahn-hilliard-navier-stokes equations,} J. Differ. Equ., 265(2018), 545--592.


\bibitem{AMH}
A. Di Primio, M. Grasselli  and H. Wu, {\it Well-posedness of a Navier-Stokes-Cahn-Hilliard system for incompressible two-phase flows with surfactant,}  Math. Models Methods Appl. Sci., 33 (2023), 755--828.


\bibitem{TD}
T. Dlotko, {\it Navier-Stokes-Cahn-Hilliard system of equations,}  J. Math. Phys., 63  (2022), Article number  111511.

\bibitem{SGM}
S. Frigeri, C. G. Gal  and M. Grasselli, {\it Two-dimensional nonlocal Cahn-Hilliard-Navier-Stokes systems with variable viscosity, degenerate mobility and singular potential,} Nonlinearity,  32 (2019), Article number 678.



\bibitem{GKL} M-H. Giga, A. Kirshtein and C. Liu, {\it  Variational modeling and complex fluids, in: Handbook of Mathematical Analysis in Mechanics of Viscous Fluids,} Springer, Cham, (2018), 73--113.


\bibitem{GGM} C.G. Gal, M. Grasselli and  A. Miranville,  {\it Cahn-Hilliard-Navier-Stokes systems with moving contact lines,} Calc. Var. Partial Differ. Equ., 55(50) (2016), 1--47.

\bibitem{PG} P. Germain, {\it Strong solutions and weak-strong uniqueness for the nonhomogeneous Navier-Stokes system,} J. Anal. Math., 105 (2008), 169--196.

\bibitem{GPV}M.E. Gurtin, D. Polignone, J. Vi\~{n}als, {\it  Two-phase binary fluids and immiscible fluids described by an order parameter,} Math. Models Methods Appl. Sci., 6 (1996), 815--831.


\bibitem{AAR}
A. Giorgini, A. Miranville and R. Temam, {\it Uniqueness and regularity for the Navier--Stokes--Cahn--Hilliard system,} SIAM J. Math. Anal., 51 (2019), 2535--2574.



\bibitem{GT} A. Giorgini and  R. Temam, {\it Weak and strong solutions to the nonhomogeneous incompressible Navier-Stokes-Cahn-Hilliard system,}  J. Math. Pures Appl., 144 (2020), 194--249.

\bibitem{GT1} A. Giorgini,  A. Ndongmo Ngana, T. Tachim Medjo  and  R. Temam, {\it  Existence and regularity of strong solutions to a nonhomogeneous Kelvin-Voigt-Cahn-Hilliard system,} J. Differ. Equ., 372 (2023), 612--656.

\bibitem{GZ} H. Gomez and  K.G. van der Zee, {\it Computational phase-field modeling, in: Encyclopedia of Computational Mechanics}, Second Edition, (2018), 1--35.


\bibitem{HMR} M. Heida, J. M\'{a}lek and  K.R. Rajagopal, {\it On the development and generalizations of Cahn-Hilliard equations within a thermodynamic framework,} Z. Angew. Math. Phys., 63 (2012), 145--169.


\bibitem{HH} P.C. Hohenberg and  B.I. Halperin,  {\it Theory of dynamic critical phenomena,} Rev. Mod. Phys., 49 (1977), 435--479.

\bibitem{DH} D. Hoff, {\it Uniqueness of weak solutions of the Navier-Stokes equations of multidimensional compressible flow,} SIAM J. Math. Anal., 37 (2006), 1742--1760.


\bibitem{Lijinkai} J. Li,  {\it Local existence and uniqueness of strong solutions to the Navier-Stokes equations
with nonnegative density,}  J. Differ. Equ.,  263 (2017), 6512--6536.





\bibitem{LS} C. Liu and J. Shen, {\it A phase field model for the mixture of two incompressible fluids and its approximation by a Fourier-spectral method,} Phys. D., 179 (2003), 211--228.

\bibitem{LT} J. Lowengrub and L. Truskinovsky, {\it Quasi-incompressible Cahn-Hilliard fluids and
topoligical transitions,} Proc. R. Soc. Lond. Ser. A Math. Phys. Eng. Sci., 454 (1998), 2617--2654.

\bibitem{PZZ}
M. Paicu, P. Zhang P and Z. Zhang, {\it Global unique solvability of inhomogeneous Navier-Stokes equations with bounded density,} Commun. Partial Differ. Equ.,  38 (2013), 1208--1234.


\bibitem{T} R. Temam,
 {\it Infinite-dimensional dynamical systems in mechanics and physics},
Springer-Verlag, New York, 1997.

\bibitem{Temam} R. Temam, {\it Navier-Stokes equations,}
AMS Chelsea Publishing, Providence, 2001.


\bibitem{FM}
F. Xu, M. Zhang,   L. Qiao  and P.  Fu, {\it Global well-posedness of 3-D nonhomogeneous incompressible MHD equations with bounded nonnegative density,} J. Math. Anal. Appl., 512 (2022), Article number 126146.



\bibitem{FMQ} F. Xu, M. Zhang  and L. Qiao, {\it The unique global solvability of nonhomogeneous incompressible asymmetric fluids in two and three dimensions,} J. Math. Phys., 65, (2024), Article number 111505 .



\bibitem{B1}
B. You, {\it Global attractor of the Cahn-Hilliard-Navier-Stokes system with moving contact lines,} Comm. Pure
Appl. Anal., 18 (2019), 2283--2298.


\bibitem{ZH} L. Zhao, {\it Strong solutions to the density-dependent incompressible Cahn-Hilliard-Navier-Stokes system,}  J. Hyperbolic Differ. Equ., 16 (2019), 701--742.


\end{thebibliography}
\end{document}